\documentclass[11pt,a4paper,reqno]{amsart} 
\usepackage{amsmath}
\usepackage{amssymb}
\usepackage{euscript}
\usepackage{graphicx}
\usepackage[all]{xy}
\usepackage[dvipsnames]{xcolor}
\usepackage[colorlinks=true,linktocpage=true,pagebackref=false, citecolor=black,linkcolor=black]{hyperref}
\usepackage{pifont}
\usepackage{tikz,pgfplots}
\pgfplotsset{compat=1.18}
\usepgfplotslibrary{fillbetween}
\usetikzlibrary{cd,arrows,decorations.pathmorphing,backgrounds,automata,positioning,fit,matrix}
\usepackage{caption,subcaption}
\usepackage{comment}
\usepackage{mathabx}

\pgfplotsset{soldot/.style={color=black,only marks,mark=*}} \pgfplotsset{holdot/.style={color=black,fill=white,only marks,mark=*}}
\pgfplotsset{bluedot/.style={color=blue,fill=blue,only marks,mark=*}}

\newtheorem{thm}{Theorem}[section]

\newtheorem{lem}[thm]{Lemma}
\newtheorem{prop}[thm]{Proposition}

\newtheorem{cor}[thm]{Corollary}

\newtheorem{assumption}[thm]{Assumption}

\theoremstyle{definition}
\newtheorem{defn}[thm]{Definition}

\theoremstyle{remark}
\newtheorem{remark}[thm]{Remark}

\newtheorem{example}[thm]{Example}

\numberwithin{equation}{section}
\numberwithin{figure}{section}

 \newcommand{\N}{{\mathbb N}}
 \newcommand{\R}{{\mathbb R}}

 \newcommand{\Cont}{{\mathcal C}}

\newcommand{\gtm}{{\mathfrak m}}

\newcommand{\Ff}{{\EuScript F}}
\newcommand{\Pp}{{\EuScript P}}
\newcommand{\Ss}{{\EuScript S}}
\newcommand{\Tt}{{\EuScript T}}

\newcommand{\Bb}{{\EuScript B}}
\newcommand{\Cc}{{\EuScript C}}

\newcommand{\Qq}{{\EuScript Q}}

\newcommand{\Hh}{{\EuScript H}}

\newcommand{\Rr}{{\EuScript R}}

\newcommand{\Jj}{{\EuScript J}}

\newcommand{\Sing}{\operatorname{Sing}}
\newcommand{\Int}{\operatorname{Int}}
\newcommand{\cl}{\operatorname{Cl}}
\newcommand{\dist}{\operatorname{dist}}
\newcommand{\id}{\operatorname{id}}

\newcommand{\zar}{\operatorname{zar}}

\newcommand{\Sth}{\operatorname{Sth}}

\newcommand{\x}{{\tt x}} \newcommand{\y}{{\tt y}} 
\newcommand{\z}{{\tt z}} \renewcommand{\t}{{\tt t}}
\newcommand{\s}{{\tt s}}

\newcommand{\veps}{\varepsilon}
\newcommand{\ol}{\overline}

\usepackage{scalerel}

\usepackage{colortbl}

\definecolor{silver}{rgb}{0.808,0.808,0.808}
\definecolor{grey}{rgb}{0.9,0.9,0.9}
\newcolumntype{a}{>{\columncolor{grey}}c}
\newcolumntype{b}{>{\columncolor{silver}}c}

\begin{document}

\title[Applications of the Nash double of a Nash manifold with corners]{Applications of the Nash double\\ 
of a Nash manifold with corners}

\author{Antonio Carbone}
\address{Dipartimento di Scienze dell'Ambiente e della Prevenzione, Palazzo Turchi di Bagno, C.so Ercole I D'Este, 32, Università di Ferrara, 44121 Ferrara (ITALY)}
\email{antonio.carbone@unife.it}
\thanks{The first author is supported by GNSAGA of INDAM}

\author{Jos\'e F. Fernando}
\address{Departamento de \'Algebra, Geometr\'\i a y Topolog\'\i a, Facultad de Ciencias Matem\'aticas, Universidad Complutense de Madrid, Plaza de Ciencias 3, 28040 MADRID (SPAIN)}
\email{josefer@mat.ucm.es}
\thanks{The second author is supported by Spanish STRANO PID2021-122752NB-I00. This article has been developed during several one month research stays of the second author in the Dipartimento di Matematica of the Universit\`a di Trento. The second author would like to thank the department for the invitations and the very pleasant working conditions. This article contains a part of the results of the Ph.D. Thesis of the first author written under the supervision of the second author.
}

\date{15/01/2026}

\begin{abstract}
In this work we study some properties and applications of Nash manifolds with corners. Our first main result shows how to `build' a Nash manifold with corners $\Qq\subset\R^n$ from a suitable Nash manifold $M\subset\R^n$ (of its same dimension), that contains $\Qq$ as a closed subset, by folding $M$ along the irreducible components of a normal-crossings divisor of $M$ (the smallest Nash subset of $M$ that contains the boundary $\partial\Qq$ of $\Qq$). Our second main results shows that we can choose as the Nash manifold $M$ the Nash `double' $D(\Qq)$ of $\Qq$, which is the analogous to the Nash double of a Nash manifold with (smooth) boundary, but $D(\Qq)$ takes into account the peculiarities of the boundary of a Nash manifold with corners. 

We propose several applications of the previous results: (1) Nash ramified coverings of closed semialgebraic sets, (2) Weak Nash uniformization of closed semialgebraic sets using Nash manifolds with (smooth) boundary, (3) Representation of compact semialgebraic sets connected by analytic paths as images under Nash maps of closed unit balls, (4) Explicit construction of Nash models for compact orientable smooth surfaces of genus $g\geq0$, and (5) Nash approximation of continuous semialgebraic maps whose target spaces are Nash manifolds with corners. 
\end{abstract}

\maketitle

\section{Introduction}\label{s1}

Nash manifolds with boundary and more generally Nash manifolds with corners are a useful tool in semialgebraic geometry to deal with several problems: (1) Compactifications of Nash manifolds \cite[Thm.VI.2.1]{sh}, (2) To prove that a $d$-dimensional non-compact Nash manifold is Nash diffeomorphic to a $d$-dimensional non-singular algebraic set \cite[Rem.VI.2.11]{sh}, (3) representation of semialgebraic sets connected by analytic paths as images under Nash maps of affine spaces \cite{fe2}, or in the compact case as images under Nash maps of closed balls or spheres \cite{cf1}, (4) Nash uniformization of closed semialgebraic sets by Nash manifold with corners \cite{cf2}, and (5) Approximation of continuous semialgebraic or ${\mathcal S}^\mu$ maps (that is, $\Cont^r$ semialgebraic maps) from a locally compact semialgebraic sets to Nash manifolds with corners by Nash maps \cite{cf3,fgh,fgh2}, among many others. 

In this article we develop new tools to deal with Nash manifolds with corners: (1) the folding of a Nash manifold $M$ along the irreducible components of a normal-crossings divisor of $M$ to construct a Nash manifold with corners (Theorem \ref{fold}), and (2) the construction of the Nash double of a Nash manifold with corners (Theorem \ref{ndnmwc}) as the natural generalization of the construction of the double of a Nash manifold with (smooth) boundary \cite[\S4.B]{fe2} (general case) and \cite[\S VI]{sh} (compact case), and we present further application of such results. Let us introduce the precise terminology to state with precision the main results of this article.

\subsection{Semialgebraic setting}
A set $\Ss\subset\R^n$ is \em semialgebraic \em if it admits a description in terms of a finite boolean combination of polynomial equalities and inequalities, which we will call a \em semialgebraic description\em. The category of semialgebraic sets is closed under basic boolean operations but also under usual topological operations: taking closures (denoted by $\cl(\cdot)$), interiors (denoted by $\Int(\cdot)$), connected components, etc. Every semialgebraic subset of $\R^n$ can be written as a finite union of {\em basic semialgebraic sets}, that is, sets of the type $\{f_1>0,\ldots,f_r>0,g=0\}$ where $f_i,g\in\R[\x]:=\R[\x_1,\ldots,\x_n]$. Closed semialgebraic subsets of $\R^n$ can be represented as finite unions of {\em closed basic semialgebraic sets}, that is, semialgebraic sets of the type $\{f_1\geq0,\ldots,f_r\geq0\}$ where $f_i\in\R[\x]$ (see \cite[Thm.2.7.2]{bcr}). Analogously, open semialgebraic subsets of $\R^n$ can be written as finite unions of {\em open basic semialgebraic sets}, that is, semialgebraic sets of the type $\{f_1>0,\ldots,f_r>0\}$ where $f_i\in\R[\x]$ (see \cite[Thm.2.7.2]{bcr}).

We denote $\ol{\Ss}^{\zar}$ the Zariski closure of a semialgebraic set $\Ss\subset\R^n$. If $\Ss\subset\R^m$ and $\Tt\subset\R^n$ are semialgebraic sets, a map $f:\Ss\to\Tt$ is \em semialgebraic \em if its graph is a semialgebraic set. Two relevant types of semialgebraic maps $f:\Ss\to\Tt$ are restrictions to $\Ss$ of {\em polynomial maps} $f:=(f_1,\ldots,f_n):\R^m\to\R^n$ (where each $f_k\in\R[\x]:=\R[\x_1,\ldots,\x_n]$) whose images are contained in $\Tt$ and restrictions to $\Ss$ of {\em rational maps} $f:=(\frac{g_1}{h_1},\ldots,\frac{g_n}{h_n}):\R^m\dasharrow\R^n$ (where each $g_k,h_k\in\R[\x]:=\R[\x_1,\ldots,\x_n]$ and each $h_k\neq0$) whose images are contained in $\Tt$. In case $\Ss\cap\{h_k=0\}=\varnothing$ for each $k=1,\ldots,n$ we say $f|_\Ss$ is a {\em regular map on $\Ss$}.

A {\em Nash map} on an open semialgebraic set $U\subset\R^n$ is a smooth semialgebraic map $f:U\to\R^m$. Along this article {\em smooth} means $\Cont^\infty$. Given a semialgebraic set $\Ss\subset\R^n$, a \em Nash map on $\Ss$ \em is the restriction to $\Ss$ of a Nash map $F:U\to\R^m$ on an open semialgebraic neighborhood $U\subset\R^n$ of $\Ss$. We denote with ${\mathcal N}(\Ss)$ the ring of Nash functions on $\Ss$ and following \cite{fg1} we say that the semialgebraic set $\Ss$ is {\em irreducible} if ${\mathcal N}(\Ss)$ is an integral domain. In \cite[\S4]{fg1} we prove that each semialgebraic set $\Ss$ can be decomposed uniquely as a finite union of irreducible semialgebraic sets $\Ss_1,\ldots,\Ss_r$ such that each $\Ss_i$ is a maximal irreducible semialgebraic subset of $\Ss$ with respect to the inclusion. The semialgebraic sets $\Ss_1,\ldots,\Ss_r$ are called the {\em irreducible components of $\Ss$}. The irreducible components $\Ss_1,\ldots,\Ss_r$ are closed subsets of $\Ss$. Observe that if $\Ss$ is irreducible (as a semialgebraic set), its Zariski closure $\ol{\Ss}^{\zar}$ is irreducible (as a algebraic set), while the converse does not hold in general.

A semialgebraic set $\Ss\subset\R^n$ is {\em connected by analytic paths} if for each pair of points $x,y\in\Ss$ there exists an analytic path $\alpha:[0,1]\to\Ss$ such that $\alpha(0)=x$ and $\alpha(0)=y$. In \cite[Thm.9.2]{fe2} we prove that each semialgebraic set $\Ss$ can be decomposed uniquely as a finite union of semialgebraic sets $\Tt_1,\ldots,\Tt_\ell$ connected by analytic paths such that each $\Tt_i$ is a maximal semialgebraic subset of $\Ss$ connected by analytic paths with respect to the inclusion. The semialgebraic sets $\Tt_1,\ldots,\Tt_\ell$ are called the {\em components of $\Ss$ connected by analytic paths}. The components $\Tt_1,\ldots,\Tt_\ell$ of $\Ss$ connected by analytic paths are closed subsets of $\Ss$. A component $\Tt_j$ connected by analytic paths of $\Ss$ is irreducible (as a semialgebraic set), so it is contained in at least one of the irreducible component $\Ss_i$ of $\Ss$ (maybe of larger dimension). Each irreducible component $\Ss_i$ of $\Ss$ is a (finite) union of components connected by analytic paths of $\Ss$, that is, the components connected by analytic paths of $\Ss_i$ are also components connected by analytic paths of $\Ss$.

A {\em Nash subset} $X\subset M$ of a {\em Nash manifold} $M\subset\R^n$ (that is, a semialgebraic set that is a smooth submanifold of $\R^n$) is the zero set of a Nash function $f:M\to\R$, whereas the {\em Nash closure in $M$} of a semialgebraic set $\Ss\subset M$ is the smallest Nash subset $X$ of $M$ that contains $\Ss$. The Nash subset $X$ of $M$ is irreducible if it cannot be written as the union of two Nash subsets of $M$ strictly contained in $X$. By \cite[Cor.8.6.8]{bcr} each Nash set admits a unique decomposition as the union of its Nash irreducible components. The Nash irreducible components of $X$ coincide with its irreducible components as a semialgebraic set \cite[3.1(v)]{fg1}. In this setting, a {\em Nash normal-crossings divisor} of a Nash manifold $M$ is a Nash set $X\subset M$ whose Nash irreducible components are Nash submanifolds of codimension $1$ of $M$ and constitute a transversal family. 

\subsubsection{Whitney's semialgebraic topology}
The {\em(Whitney's) ${\mathcal S}^0$ topology} of the space ${\mathcal S}^0(\Ss,\Tt)$ of continuous semialgebraic maps between the semialgebraic sets $\Ss\subset\R^m$ and $\Tt\subset\R^n$ is determined by the open neighborhoods: ${\mathcal U}_f:=\{g\in{\mathcal S}^0(\Ss,\Tt):\ \|f-g\|<\veps\}$ where $f\in{\mathcal S}^0(\Ss,\Tt)$ and $\veps:\Ss\to\R$ is a strictly positive continuous semialgebraic function on $\Ss$. If $\mu\geq1$ the {\em (Whitney's) ${\mathcal S}^\mu$ topology} \cite[\S II.1, p.79--80]{sh} of the space ${\mathcal S}^\mu(M,N)$ of ${\mathcal S}^\mu$ maps between two Nash manifolds $M\subset\R^m$ and $N\subset\R^n$ is defined as follows. Let $\xi_1,\dots,\xi_s$ be semialgebraic ${\mathcal S}^{\mu-1}$ tangent fields on $M$ that span the tangent bundle of $M$. For every strictly positive continuous semialgebraic function $\veps:M\to\R$ and each $f\in{\mathcal S}^\mu(M,N)$ we denote the set of all $g\in{\mathcal S}^\mu(M,N)$ such that $\|f-g\|<\veps$ and $\|\xi_{i_1}\cdots\xi_{i_\ell}(f-g)\|<\veps$ for $1\leq i_1,\dots,i_\ell\leq s$ and $1\leq\ell\leq r$ with ${\mathcal U}_\epsilon$. These sets ${\mathcal U}_\epsilon$ form a basis of neighborhoods of $f$ that does not depend on the choice of the tangent fields. Let $\Ss\subset M$ be a closed semialgebraic subset of $M$ and $\Tt\subset N$ a semialgebraic subset of $N$. The ${\mathcal S}^\mu$ topology of ${\mathcal S}^\mu(M,N)$ induces naturally a \textit{(Whitney's) ${\mathcal S}^\mu$ topology} on the subspace ${\mathcal S}^\mu(\Ss,\Tt)$ of ${\mathcal S}^\mu$ maps between $\Ss$ and $\Tt$. We refer the reader to \cite[\S2.D]{bfr} and \cite[\S2.2]{cf3} for further details.

\subsection{Folding Nash manifolds to build Nash manifolds with corners}
A {\em Nash manifold with corners} is a semialgebraic set that is a smooth submanifold with corners of $\R^n$. A Nash manifold with corners $\Qq\subset\R^n$ is contained, as a closed (semialgebraic) subset, in a Nash manifold $M\subset\R^n$ of its same dimension \cite[Prop.1.2]{fgr}. In fact, we restrict our scope to Nash manifolds $\Qq\subset\R^n$ with corners such that the Nash closure in $M$ of the boundary $\partial\Qq$ is a (Nash) normal-crossings divisor of $M$ (maybe after shrinking $M$). Our main result shows that the previous Nash manifold $M$ can be `folded' to reconstruct the Nash manifold with corners $\Qq$. Namely,

\begin{thm}[Folding Nash manifolds]\label{fold}
Let $\Qq\subset\R^n$ be a $d$-dimensional Nash manifold with corners that is a closed (semialgebraic) subset of $\R^n$. Then there exist:
\begin{itemize}
\item[(i)] A $d$-dimensional Nash manifold $M\subset\R^n$ that contains $\Qq$ as a closed subset.
\item[(ii)] A Nash normal-crossings divisor $Y\subset M$ that is the Nash closure of $\partial\Qq$ in $M$ and satisfies $\Qq\cap Y=\partial\Qq$.
\item[(iii)] An open semialgebraic neighborhood $W\subset M$ of $\Qq$ such that $\cl(W)\subset M$ and a proper Nash map $f:\cl(W)\to\Qq$ such that $f(\Qq)=\Qq$, $f(\Int(\Qq))=\Int(\Qq)$, $f|_\Qq:\Qq\to\Qq$ is a semialgebraic homeomorphism close to the identity map (with respect to the ${\mathcal S}^0$ topology) and $f|_{\Int(\Qq)}:\Int(\Qq)\to\Int(\Qq)$ is a Nash diffeomorphism.
\end{itemize}
In addition, for each $x\in\partial\Qq$ there exist open semialgebraic neighborhoods $U,V\subset M$ of $x$ equipped with Nash diffeomorphisms $\varphi:U\to\R^d$ and $\psi:V\to\R^d$ and $1\leq s\leq d$ such that 
$$
\psi\circ f\circ\varphi^{-1}:\R^d\to\R^d,\ (x_1,\ldots,x_d)\mapsto(x_1^2,\ldots,x_s^2,x_{s+1},\ldots,x_d).
$$
\end{thm}
\begin{remark}
By Mostowski's trick (Proposition \ref{Mos}) we can always embedded a Nash manifold with corners as a closed (semialgebraic) subset of an affine space.
\end{remark}

In addition, we construct for Nash manifolds with corners $\Qq\subset\R^n$ the counterpart of the Nash double of a Nash manifold with (smooth) boundary and such Nash double $D(\Qq)$ is Nash diffeomorphic around $\Qq_+$ (which is a suitable Nash manifold with corners contained in $D(\Qq)$ and Nash diffeomorphic to $\Qq$) to the Nash manifold $M$ provided by Theorem \ref{fold} (Theorem \ref{ndnmwc}). As the presentation is rather technical and involves quite notations, we pospone the precise statement until Section \ref{s4}.

\subsection{Applications}
We next propose some further applications of folding of Nash manifolds to construct Nash manifolds with corners.

\subsubsection{Nash ramified covering of closed semialgebraic sets}
As a straightforward consequence of \cite[Cor.1.10]{cf2}, Lemma \ref{cnma} and Theorem \ref{ndnmwc} the reader can prove the following:

\begin{cor}[Nash ramified covering map]\label{red5}
Let $\Ss\subset\R^m$ be a closed semialgebraic set and let $\Ss_1,\ldots,\Ss_\ell$ be the components of $\Ss$ connected by analytic paths. Then there exist: 
\begin{itemize}
\item[{\rm(i)}] A finite union $X\subset\R^n$ of pairwise disjoint irreducible non-singular real algebraic sets $X_i\subset\R^n$ (maybe of different dimensions) for $i=1,\ldots,\ell$. 
\item[\rm{(ii)}] A polynomial map $f:\R^n\to\R^m$ such that $f(X_i)=\Ss_i$ and the restriction $f|_{X_i}:X_i\to\Ss_i$ is proper for each $i=1,\ldots,\ell$. 
\item[\rm{(iii)}] Closed semialgebraic sets $\Rr_i\subset\Ss_i$ of dimension strictly smaller than the dimension of $\Ss_i$ for $i=1,\ldots,\ell$ such that $\Ss_i\setminus\Rr_i$ and $X\setminus f^{-1}(\Rr_i)$ are Nash manifolds of the same dimension as $\Ss_i$ and the restriction $f|_{X_i\setminus f^{-1}(\Rr_i)}:X_i\setminus f^{-1}(\Rr_i)\to\Ss_i\setminus\Rr_i$ is a Nash covering map whose fibers are finite and have constant cardinality for $i=1,\ldots,\ell$. 
\end{itemize}
\end{cor}

\subsubsection{Weak Nash uniformization of closed semialgebraic sets by Nash manifolds with boundary.}\label{resboundary}

A Nash uniformization result for semialgebraic sets similar to \cite[Cor.1.10]{cf2} changing the Nash manifold with corners $\Qq$ by a Nash manifold with boundary $\Hh$ (that is, a semialgebraic set that is a smooth submanifold with (smooth) boundary of $\R^n$) seems difficult to be obtained. We propose the following statement, whose proof is based on the folding of Nash manifolds (Theorem \ref{fold}):

\begin{thm}[Weak Nash uniformization of closed semialgebraic sets]\label{red3}
Let $\Ss\subset\R^m$ be a $d$-dimensional closed semialgebraic set, let $\Ss_1,\ldots,\Ss_\ell$ be the components of $\Ss$ connected by analytic paths and let $\veps_i:\Ss_i\to\R$ be a strictly positive continuous semialgebraic function on $\Ss_i$ for $i=1,\ldots,\ell$. Then there exist:
\begin{itemize}
\item[{\rm(i)}] A finite union $X\subset\R^n$ of pairwise disjoint irreducible non-singular real algebraic sets $X_i\subset\R^n$ (maybe of different dimensions) for $i=1,\ldots,\ell$. 
\item[{\rm(ii)}] Nash manifolds with boundary $\Hh_{i,\veps}\subset\R^m$ such that the Zariski closure $Z_{i,\veps}$ of $\partial\Hh_{i,\veps}$ is a non-singular real algebraic set $Z_{i,\veps}\subset X_i$ of dimension $\dim(\Ss_i)-1$ and the interior $\Int(\Hh_{i,\veps}):=\Hh_{i,\veps}\setminus\partial\Hh_{i,\veps}$ is a connected component of $X_i\setminus Z_{i,\veps}$ for $i=1,\ldots,\ell$.
\item[{\rm(iii)}] A proper Nash map $f:\bigsqcup_{i=1}^\ell\Hh_{i,\veps}\to\Ss$ such that $f(\Hh_{i,\veps})=\Ss_i$ for $i=1,\ldots,\ell$.
\item[{\rm(iv)}] A closed semialgebraic subset $\Rr_i\subset\Ss_i$ of dimension strictly smaller than $\dim(\Ss_i)$ such that if $\Tt_{i,\veps}:=\{x\in\Ss_i:\ {\rm dist}(x,\Rr_i)\leq\veps_i(x)\}$, the restriction map $f|_{\Hh_{i,\veps}\setminus f^{-1}(\Tt_{i,\veps})}:\Hh_{i,\veps}\setminus f^{-1}(\Tt_{i,\veps})\to\Ss_i\setminus\Tt_{i,\veps}$ is a Nash diffeomorphism.
\end{itemize}
If $\Ss$ is compact, we may assume in addition: $\veps_i$ is constant, $X_i$ and $\Hh_{i,\veps}$ are also compact and $X_i$ is connected for $i=1,\ldots,\ell$.
\end{thm}

Related to the previous result in \cite[Proof of Thm.1.4, \S8.C.12]{fe3} it is shown that if $\Ss\subset\R^n$ is a semialgebraic set connected by analytic paths, then there exists a Nash manifold $\Hh$ with smooth boundary and a surjective Nash map $f:\Hh\to\Ss$. However, the Nash map provided there is far from being proper even if $\Ss$ is a closed semialgebraic set and in general far from being generically injective.

\subsubsection{Nash images of the closed unit ball}
In Section \ref{altcontru} we provide an alternative proof of the following result proposed in \cite{cf1}. We denote the closed ball of $\R^n$ of center $p\in\R^n$ and radius $r>0$ with $\ol{\Bb}_n(p,r)$ and for the sake of simplicity we write $\ol{\Bb}_n:=\ol{\Bb}_n(0,1)$.

\begin{thm}[Compact Nash images]\label{main1}
Let $\Ss\subset\R^n$ be a $d$-dimensional compact semialgebraic set. The following assertions are equivalent:
\begin{itemize}
\item[(i)] There exists a Nash map $f:\R^d\to\R^n$ such that $f(\ol{\Bb}_d)=\Ss$.
\item[(ii)] $\Ss$ is connected by analytic paths.
\end{itemize}
\end{thm}

We refer the reader to an unconventional citation \cite{whi} in `AMS Feature Column' concerning some related result \cite{cf1} known as Teddy-Lambkin's Theorem (B\"archen-Sch\"afchen's Theorem). 

\subsubsection{Explicit Nash models for orientable compact smooth surfaces}

Theorem \ref{fold} provides an explicit procedure to construct Nash models for orientable compact smooth surfaces, see \S\ref{surfaces}.

\subsubsection{Approximation of continuous semialgebraic maps by Nash maps}

Approximation of maps of certain type by maps of a better subtype is an important tool in Mathematics. A celebrated example is Whitney's approximation theorem \cite{wh}: {\em a continuous map (defined on a locally compact subset of an affine space) whose target space is a ${\mathcal C}^r$ submanifold $M$ of $\R^n$ for $r\in\N\cup\{\infty,\omega\}$ can be approximated by a ${\mathcal C}^r$ map whose target space is $M$}. Whitney approximation theorem has been extended in many directions, like the case of manifolds with boundary (using partitions of unit and collars, \cite[Ch. III, Thm. 6.1]{orr}). New results on approximation in Whitney's style have been proved in \cite{fgh2} when the target space has `singularities' and admits `nice' triangulations. As a consequence: {\em every continuous map between a locally compact subset $X\subset\R^m$ and a smooth manifold with corners $\Qq\subset\R^n$ can be approximated by a smooth map with respect to the strong Whitney's $\Cont^0$ topology \cite[Cor.1.10]{fgh2}}.

There are also relevant results concerning approximation in the semialgebraic setting. Efroymson \cite{ef} showed that every continuous semialgebraic function on a Nash manifold can be approximated by Nash functions. Shiota \cite{sh} improved this result in many directions. He proved relative and absolute versions with (strong) control on finitely many derivatives of the approximation map. Approximation results also appear when the target space has singularities. In \cite[Thm.1.7]{bfr} the target space is a Nash set with monomial singularities and approximation of continuous semialgebraic maps by Nash maps is proved under certain restrictions. In \cite{fgh} it is presented differentiable approximation of continuous semialgebraic maps when the target space admits `nice' triangulations. In \cite{ca} it is shown that the approximation can be made without changing the image. As a consequence: {\em Every continuous semialgebraic map $f$ between a compact semialgebraic set $\Ss\subset\R^m$ and a Nash manifold with corners $\Qq\subset\R^n$, can be approximated by an ${\mathcal S}^\mu$ map $g$ with respect to ${\mathcal S}^0$ topology such that $g(\Ss)=f(\Ss)$ \cite[Thm.1.3 \& Thm.1.4]{fgh} and \cite[Thm.1.4]{ca}}. To apply \cite[Thm.1.4]{fgh} we need in addition that the target space is compact. But this is not limiting, because as $X$ is compact, its image $f(\Ss)$ under a continuous semialgebraic map $f:\Ss\to\Qq$ is compact. Thus, we can substitute $\Qq$ by a compact Nash manifold with corners of the type $\Qq_r:=\Qq\cap\Bb(0,r)\subset\Qq$ for a suitable radius $r>0$ that in addition contains $f(\Ss)$. The suitable radius $r>0$ can be found as an application of Sard's theorem. Recall that Paw\l{}ucki has shown that (compact) semialgebraic sets always admit `nice triangulations' \cite{p1,p2}, so the results of \cite{fgh} and \cite{ca} always hold for continuous semialgebraic maps defined on a compact semialgebraic set.

The techniques developed in \cite{fgh} make an essential use of ${\mathcal S}^\mu$ partitions of unity, so the proofs proposed there do not extend to Nash approximation. However, as an additional application of Theorem \ref{fold}, we prove approximation of continuous semialgebraic maps by Nash maps when the target space is a Nash manifold with corners.

\begin{thm}[Nash approximation I]\label{approxn}
Let $\Ss\subset\R^m$ be a locally compact semialgebraic set, $\Qq\subset\R^n$ a Nash manifold with corners and $f:\Ss\to\Qq$ a continuous semialgebraic map. Then there exist Nash maps $g:\Ss\to\Qq$ arbitrarily close to $f$ (with respect to the ${\mathcal S}^0$ topology).
\end{thm}

The previous result has been improved in \cite{cf3} where we have shown, using more sophisticated techniques, the following result.

\begin{thm}[Nash approximation II]
Let $\Ss\subset\R^m$ be a locally compact semialgebraic set, let $\Qq\subset\R^n$ be a Nash manifold with corners and let $f:\Ss\to\Qq$ be the restriction of an ${\mathcal S}^\mu$ map $F:\Omega\to\R^n$, where $\Omega$ is an open semialgebraic neighborhood of $\Ss$ in $\R^n$. Then there exist Nash maps $g:\Ss\to\Int(\Qq)$ arbitrarily close to $f$ with respect to the ${\mathcal S}^\mu$ topology.
\end{thm}

\subsection{Structure of the article}
The article is organized as follows. In Section \ref{s2} we present some preliminary results: (1) Nash retractions (of open semialgebraic subsets of a Nash manifold onto a Nash hypersurface) compatible with a Nash normal-crossings divisor \cite[Prop.4.1]{fgh}, (2) Compatible Nash collars (Lemma \ref{style}) and Nash doubles for Nash manifolds with (smooth) boundary \cite[\S4.B.1, \S4.B.2]{fe3}, (3) Basic properties on Nash manifolds with corners and Nash equations of the irreducible components of their boundaries, (4) Non-singular real algebraic models of Nash manifolds via Artin-Mazur's theorem (Lemma \ref{cnma}). In Section \ref{s3} we prove Theorem \ref{fold} and in Section \ref{s4} we construct the Nash double of a Nash manifold with corners (Theorem \ref{ndnmwc}). Section \ref{s5} is devoted to prove Theorems \ref{red3}, \ref{main1} and \ref{approxn}. We also present an explicit procedure to construct Nash models for orientable compact smooth surfaces.

\section{Preliminaries}\label{s2}

We present in this section several tools that will be used in the subsequent sections.

\subsection{Mostowski's trick}

We begin with a well-known trick of Mostowski, because it allows to embedded via a Nash map locally compact semialgebraic sets of an affine space as closed semialgebraic sets of a (maybe larger) affine space. Given a semialgebraic set $\Ss\subset\R^m$, we define its {\em exterior boundary} as $\delta^\bullet\Ss:=\cl(\Ss)\setminus\Ss$.

\begin{prop}[Mostowski's trick]\label{Mos}
Let $\Ss\subset\R^m$ be a locally compact semialgebraic set. Then there exists a Nash map $H:\R^m\setminus\delta^\bullet\Ss\to\R^{m+1}$ such that the image $H(\Ss)$ is a closed semialgebraic subset of $\R^{m+1}$ Nash diffeomorphic to $\Ss$. 
\end{prop}
\begin{proof}
As $\Ss$ is locally compact in $\R^m$, its exterior boundary $\delta^\bullet\Ss$ is a closed semialgebraic subset of $\R^m$. By \cite[Lem.6]{mo} there exists a continuous semialgebraic function $h:\R^m\to\R$ such that $\delta^\bullet\Ss=\{h=0\}$ and $h$ is Nash on the open semialgebraic set $\R^m\setminus\delta^\bullet\Ss$. Consider the Nash diffeomorphism 
$$
H:\R^m\setminus\delta^\bullet\Ss\to M:=\{(x,t)\in\R^m\times\R:\ th(x)=1\}\subset\R^{m+1},\ x\mapsto\Big(x,\frac{1}{h(x)}\Big),
$$
whose inverse is the restriction to the Nash manifold $M$ (which is a closed subset of $\R^{m+1}$) of the projection $\pi:\R^m\times\R\to\R^m$ onto the first factor. Observe that $H(\Ss)$ is the closed semialgebraic set $\Tt:=M\cap(\cl(\Ss)\times\R)$, as required.
\end{proof}

The following semialgebraic separation result will be useful in the sequel also combined with Mostowski's trick.

\begin{lem}\label{co}
Let $\Ss\subset\R^n$ be a locally compact semialgebraic set and let $\Tt\subset\Ss$ be a semialgebraic subset that is a closed semialgebraic subset of $\R^n$. Then there exists an open semialgebraic neighborhood $U\subset\Ss$ of $\Tt$ such that $\cl(U)\subset\Ss$. 
\end{lem}
\begin{proof}
As $\Ss$ is locally compact, the semialgebraic set $\cl(\Ss)\setminus\Ss$ is a closed subset of $\R^n$. As $\Tt\subset\Ss$, the closed semialgebraic sets $\Tt$ and $\cl(\Ss)\setminus\Ss$ are disjoint. Thus, there exists an open semialgebraic set $V\subset\R^n$ such that $\Tt\subset V\subset\cl(V)\subset\R^n\setminus(\cl(\Ss)\setminus\Ss)$. Define $U:=\Ss\cap V$, which is an open semialgebraic neighborhood of $\Tt$ in $\Ss$, and observe that $\cl(U)\subset\cl(\Ss)\cap\cl(V)\subset\cl(\Ss)\cap(\R^n\setminus(\cl(\Ss)\setminus\Ss))=\Ss$, as required.
\end{proof}

\subsection{Nash manifolds with boundary}
A Nash manifold $M\subset\R^n$ with (smooth) boundary is a semialgebraic set that is smooth submanifold with smooth boundary. We first present the construction of the Nash double of a Nash manifold with boundary and next Nash collars of the boundary of a Nash mnaifold with boundary compatible with a Nash normal-crossings divisor.

\subsubsection{Nash doubles of Nash manifolds with boundary.} 
Doubling a smooth manifold with boundary is a standard tool in differential topology. The Nash construction has been treated in \cite[Thm.VI.2.1]{sh} (compact case) and in \cite[\S4.B.1, \S4.B.2]{fe3} (general case). We recall next the main statements and refer the reader to \cite[\S4.B.1, \S4.B.2]{fe3} for the precise details (Figure \ref{fig11}).

\begin{lem}[Nash double, {\cite[\S4.B.1]{fe3}}]\label{double} 
Let $\Hh\subset\R^n$ be a $d$-dimensional Nash manifold with boundary $\partial\Hh$ and let $h:M\to\R$ be a Nash equation for $\partial\Hh$ such that $\partial\Hh=\{h=0\}$, $\Int(\Hh)=\{h>0\}$ and $d_xh:T_x\Hh\to\R$ is surjective for all $x\in\partial\Hh$. Then, the semialgebraic set
$$
D(\Hh):=\{(x,t)\in\Hh\times\R:\ t^2-h(x)=0\}\subset\R^{n+1}
$$ 
is a $d$-dimensional Nash manifold that contains $\partial\Hh\times\{0\}$ as the Nash subset $\{t=0\}$.
\end{lem}

\begin{figure}[!ht]
\begin{center}
\begin{tikzpicture}[scale=0.75]

\draw[line width=1pt,rotate=-90,dashed] (-3.5,0) parabola bend (-5.5,4) (-7.5,0);
\draw[rotate=-90,draw=none,fill=gray!100] (-4.5,6) parabola bend (-4.5,6) (-6.5,2)--
(-7.5,0) parabola bend (-5.5,4) (-6.5,3)--(-5.5,5);
\draw[rotate=-90,draw=none,fill=gray!20] (-2.5,2) parabola bend (-4.5,6) (-6,3.75) parabola bend (-5.5,4) (-3.5,0);

\draw[line width=0.5pt,rotate=-90,dashed] (-2.5,2) parabola bend (-4.5,6) (-6.5,2);
\draw[line width=0.5pt,dashed](2,2.5)--(0,3.5);
\draw[line width=1.5pt,draw](6,4.5)--(4,5.5);
\draw[line width=0.5pt,dashed](2,6.5)--(0,7.5);
\draw[line width=0.5pt,dashed](5,5.5)--(3,6.5);

\draw[draw=none,fill=gray!20] (2,0.5) -- (6,0.5) -- (4,1.5) -- (0,1.5) -- (2,0.5);
\draw[line width=0.5pt,dashed](2,0.5)--(6,0.5);
\draw[line width=1.5pt,draw](6,0.5)--(4,1.5);
\draw[line width=0.5pt,dashed](4,1.5)--(0,1.5);
\draw[line width=0.5pt,dashed](0,1.5)--(2,0.5);

\draw[<->, line width=0.5pt] (1.5,5) -- (6.5,5);
\draw[<->, line width=0.5pt] (3.5,5.75) -- (6.5,4.25);
\draw[<->, line width=0.5pt] (5,3) -- (5,7);

\draw[->, line width=1pt] (5.5,3) -- (5.5,1.5);

\draw (2,5.25) node{\small$x_1$};
\draw (6.85,4.25) node{\small$x_d$};
\draw (5.25,6.75) node{\small$t$};
\draw (0.6,6.85) node{\small$\Hh_+$};
\draw (0.6,5.5) node{\small$D(\Hh)$};
\draw (0.6,4) node{\small$\Hh_-$};
\draw (6.1,5.25) node{\small$\partial\Hh\times\{0\}$};
\draw (2.25,1) node{\small$\Hh$};
\draw (6,1) node{\small$\partial\Hh$};
\draw (5.75,2.25) node{\small$\pi$};

\end{tikzpicture}
\end{center}
\caption{\small{Nash double of $\Hh$ and the projection $\pi$ (figure borrowed from \cite[Fig.3]{fe3}).}\label{fig11}}
\label{doppio}
\end{figure}
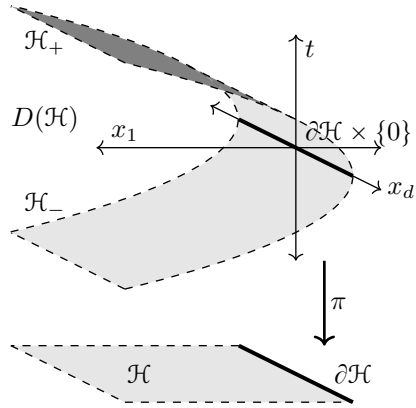

\begin{lem}[{\cite[4.B.2]{fe3}}]\label{doubleii}
The projection $\pi: D(\Hh)\to\Hh,\, (x,t)\mapsto x$ is a surjective Nash map. Fix $\epsilon=\pm$ and denote $\Hh_{\epsilon}:=D(\Hh)\cap\{\epsilon t\geq0\}$. We have the following: 
\begin{itemize}
\item[{\rm(i)}] The restriction $\pi_{\epsilon}:=\pi|_{\Hh_{\epsilon}}:\Hh_{\epsilon}\to\Hh$ is a semialgebraic homeomorphism and the restriction $\pi|_{D(\Hh)\cap\{\epsilon t>0\}}:D(\Hh)\cap\{\epsilon t>0\}\to\Int(\Hh)$ is a Nash diffeomorphism. 
\item[{\rm(ii)}] $\pi(x,0)=x$ for all $(x,0)\in\partial\Hh\times\{0\}=D(\Hh)\cap\{t=0\}$. 
\item[{\rm(iii)}] $\pi$ has local representations $(y_1,\ldots,y_d)\mapsto(y_1^2,y_2,\ldots,y_d)$ at each point of $D(\Hh)\cap\{t=0\}$.
\item[{\rm(iv)}] $\pi$ is open and proper.
\end{itemize}
\end{lem}
\begin{proof}
Properties (i), (ii) and (iii) are proved in \cite[4.B.2]{fe3}. Let us show (iv). We prove first that $\pi$ is open. Let $A\subset D(\Hh)$ be an open set and let $A_{\epsilon}:=A\cap \Hh_{\epsilon}$, which is an open subset of $\Hh_{\epsilon}$. As $\pi_{\epsilon}:\Hh_{\epsilon}\to\Hh$ is a semialgebraic homeomorphism, $\pi(A_{\epsilon})$ is an open subset of $\Hh$. Thus, $\pi(A)=\pi(A_+)\cup\pi(A_-)$ is an open subset of $\Hh$, so $\pi$ is open.

To show that $\pi$ is proper, pick $K\subset\Hh$ compact and observe that $\pi^{-1}(K)=(\pi_+)^{-1}(K)\cup(\pi_-)^{-1}(K)$. As each $\pi_{\epsilon}:\Hh_{\epsilon}\to\Hh$ is a semialgebraic homeomorphism, $(\pi_{\epsilon})^{-1}(K)$ is compact for $\epsilon=\pm$, so $\pi^{-1}(K)$ is a compact subset of $D(\Hh)$, so $\pi$ is proper, as required. 
\end{proof}

\subsubsection{Compatible Nash retractions.}

Let $M\subset\R^n$ be a connected $d$-dimensional Nash manifold and $Y\subset M$ a Nash normal-crossings divisor. If $Y_1,\ldots,Y_r$ are the Nash irreducible components of $Y$, we define $\Sing(Y):=\bigcup_{i\neq j}(Y_i\cap Y_j)$. Construct inductively 
\begin{align*}
&\Sing_1(Y):=\Sing(Y),\\
&\Sing_\ell(Y):=\Sing_{\ell-1}(\Sing(Y))\quad\text{for $\ell\geq 2$}.
\end{align*}
In order to lighten the exposition, we write $\Sing_0(Y):=Y$. As $Y\subset M$ is a Nash normal-crossings divisor, $\Sing_\ell(Y)$ is either the empty set or a transversal union of finitely many Nash manifolds for each $\ell\geq0$ (see \cite[Lem.5.1]{bfr}). Thus, the Nash irreducible components of $\Sing_\ell(Y)$ are Nash manifolds for each $\ell\geq0$ such that $\Sing_\ell(Y)\neq\varnothing$. In fact, if $Y_{\ell,1},\ldots,Y_{\ell, s_\ell}$ are the irreducible components of $\Sing_\ell(Y)$, then $\Sing_{\ell+1}(Y)=\bigcup_{i\neq j}(Y_{\ell,i}\cap Y_{\ell, j})$. If $\Sing_\ell(Y)\neq\varnothing$, then $\dim(\Sing_\ell(Y))=d-\ell-1$. 

For each $\ell\geq 1$ we have $\Sing_{\ell+1}(Y)\subset\Sing_\ell(Y)$ and $\dim(\Sing_{\ell+1}(Y))\leq\dim(\Sing_\ell(Y))-1$, so there exists an $r\geq0$ such that $\Sing_k(Y)\neq\varnothing$ for $0\leq k\leq r$ and $\Sing_k(Y)=\varnothing$ if $k\geq r+1$. 

\begin{defn}
Let $Z$ be a Nash irreducible component of $\Sing_k(Y)$ for some $0\leq k\leq r$. A Nash retraction $\rho:W\to Z$, where $W\subset M$ is an open semialgebraic neighborhood of $Z$, is {\em compatible with} $Y$ if $\rho(Y_i\cap W)=Y_i\cap Z$ for each irreducible component $Y_i$ of $Y$ such that $Y_i\cap Z\neq\varnothing$. 
\end{defn}

In \cite[Prop.4.1]{fgh} the following result is proved, which is a powerful tool to make constructions compatible with an assigned Nash normal-crossings divisor.

\begin{prop}[Compatible Nash retractions, {\cite[Prop. 4.1]{fgh}}]\label{compretraction}
There exist an open semialgebraic neighborhood $W\subset M$ of $Z$ and a Nash retraction $\rho:W\to Z$ that is compatible with $Y$. In addition, $\rho(X\cap W)=X\cap Z$ for each irreducible component $X$ of $\Sing_\ell(Y)$ such that $X\cap Z\neq\varnothing$ and $\ell\geq0$. 
\end{prop}

\subsubsection{Compatible Nash collars for Nash manifolds with boundary.}

Nash collars for a Nash manifold $\Hh\subset\R^n$ with boundary $\partial\Hh$ have been constructed in \cite[Lem.VI.1.6]{sh} (compact case) and in \cite[Lem.4.2]{fe3} (general case). We adapt these constructions to build Nash collars compatible with a Nash normal-crossings divisor $Y$ that contains $\partial\Hh$.

\begin{lem}[Compatible Nash collars]\label{style}
Let $M\subset\R^n$ be a Nash manifold and $Y\subset M$ a Nash normal-crossings divisor. Let $Y_1$ be a Nash irreducible component of $Y$, $W\subset M$ an open semialgebraic neighborhood of $Y_1$ and $\rho:W\to Y_1$ a Nash retraction compatible with $Y_1$ such that $\rho(X\cap W)=X\cap Y_1$ for each irreducible component $X$ of $\Sing_\ell(Y)$ such that $X\cap Y_1\neq\varnothing$ and $\ell\geq0$. Let $h$ be a Nash function on $W$ such that $\{h=0\}=Y_1$ and $d_xh:T_xM\to\R$ is surjective for all $x\in Y_1$. Then there exist an open semialgebraic neighborhood $V\subset W$ of $Y_1$ and a strictly positive Nash function $\veps:Y_1\to\R$ such that: 
\begin{itemize}
\item[(i)] The Nash map $\varphi:=(\rho,\frac{h}{\veps\circ\rho}):V\to Y_1\times(-1,1)$ is a Nash diffeomorphism. 
\item[(ii)] $\varphi(Z\cap V)=(Z\cap Y_1)\times(-1,1)$ for each irreducible component $Z$ of $\Sing_\ell(Y)$ such that $Z\not\subset Y_1$ and $Z\cap Y_1\neq\varnothing$ and $\ell\geq0$.
\end{itemize}
\end{lem}
\begin{proof}
Define $\phi:=(\rho,h):W\to Y_1\times\R$. We show first: {\em The derivative $d_x\phi=(d_x\rho,d_xh):T_xM\to T_xY_1\times\R$ is an isomorphism for all $x\in Y_1$}. As $\dim(T_xM)=\dim(T_xY_1\times\R)$, it is enough to prove: \em $d_x\phi$ is surjective\em. 

As $\phi|_{Y_1}=(\id_{Y_1},0)$, we have $d_x\phi|_{T_xY_1}=(\id_{T_xY_1},0)$, so $T_xY_1\times\{0\}\subset d_x\phi(T_xM)$. In addition $d_xh:T_xM\to\R$ is surjective, so there exists $v\in T_xM$ such that $d_xh(v)=1$. Thus, $d_x\phi(v)=(d_x\rho(v),1)$, so $d_x\phi$ is surjective.

Let $W':=\{x\in W:\ d_x\phi\ \text{is an isomorphism}\}$, which is an open semialgebraic neighborhood of $Y_1$. Thus, $\phi|_{W'}:W'\to Y_1\times\R$ is an open map and $\phi(W')$ is an open semialgebraic neighborhood of $Y_1\times\{0\}$ in $Y_1\times\R$. As $\phi|_{W'}:W'\to\phi(W')$ is a local homeomorphism and $\phi|_{Y_1}=(\id_{Y_1},0)$ is a homeomorphism (onto its image), there exist by \cite[Lem.9.2]{bfr} open semialgebraic neighborhoods $W''\subset W'$ of $Y_1$ and $U\subset Y_1\times\R$ of $Y_1\times\{0\}$ such that $\phi|_{W''}:W''\to U$ is a semialgebraic homeomorphism.

Consider the strictly positive, continuous semialgebraic map 
$$
\delta:Y_1\to(0,+\infty),\ x\mapsto\dist((x,0),(Y_1\times\R)\setminus U).
$$ 
By \cite[Thm.II.4.1]{sh} there exists a strictly positive Nash function $\veps$ on $Y_1$ such that $|\veps-\frac{3}{4}\delta|<\frac{1}{4}\delta$, that is, $\frac{1}{2}\delta<\veps<\delta$. Consider the open semialgebraic neighborhood 
$$
U':=\{(y,t)\in Y_1\times\R:\ |t|<\veps(y)\}\subset U
$$ 
of $Y_1\times\{0\}$ and define $V:=(\phi|_{W''})^{-1}(U')$. The restriction $\phi|_V:V\to U'$ is a Nash diffeomorphism (because it is a Nash homeomorphism such that $d_x(\phi|_V)$ is an isomorphism for each $x\in V$). Consequently, 
$$
\varphi:V\to Y_1\times(-1,1),\ x\mapsto\Big(\rho(x),\frac{h(x)}{\veps(\rho(x))}\Big)
$$
is a Nash diffeomorphism, because it is the composition of $\phi|_V$ with the Nash diffeomorphism
$$
\psi:U'\to Y_1\times(-1,1),\ (y,t)\mapsto\Big(y,\frac{t}{\veps(y)}\Big).
$$ 

Let $\ell\geq0$ and let $Z$ be a Nash irreducible component of $\Sing_\ell(Y)$ such that $Z\not\subset Y_1$ and $Z\cap Y_1\neq\varnothing$. As $Y$ is a Nash normal-crossings divisor, $Z$ is a Nash manifold \cite[Lem.5.1]{bfr}. Observe that $Z\cap V$ is a Nash manifold, say of dimension $e$. By Proposition \ref{compretraction} $\rho(Z\cap V)=Z\cap Y_1$, so $\varphi(Z\cap V)\subset(Z\cap Y_1)\times(-1,1)$. Let us prove: {\em $\varphi|_{Z\cap V}:Z\cap V\to(Z\cap Y_1)\times(-1,1)$ is surjective}.

As $Y$ is a Nash normal-crossings divisor, $Z\cap Y_1$ is a Nash manifold, so its connected components $C_i$ are Nash manifolds of dimension $e-1$ where $e=\dim(Z)$, because $Z\not\subset Y_1$. Thus, each $C_i\times(-1,1)$ is a closed connected Nash submanifold of $Y_1\times(-1,1)$ of dimension $e$ and, in particular, $C_i\times(-1,1)$ is an irreducible Nash subset of $Y_1\times(-1,1)$ of dimension $e$. 

Let $Z_i$ be a connected component of the Nash manifold $Z\cap V$ such that $C_i\subset Z_i\cap Y_1$. As $Z_i$ is connected, also $\rho(Z_i)\subset Z\cap Y_1$ is connected, so it is contained in one of the connected components of $Z\cap Y_1$. As $C_i\subset Z_i\cap Y_1\subset\rho(Z_i)$ is a connected component of $Z\cap Y_1$, we deduce $\rho(Z_i)=Z_i\cap Y_1=C_i$. Then $\varphi(Z_i)\subset C_i\times(-1,1)$ and $\varphi(Z_i)$ is a Nash subset of $Y_1\times(-1,1)$ of dimension $e$ (because $\varphi$ is a Nash diffeomorphism and $Z_i$ is a connected component of the $e$-dimensional Nash manifold $Z\cap V$, which is a closed subset of $V$). By the identity principle, we conclude $\varphi(Z_i)=C_i\times(-1,1)$ because $C_i\times(-1,1)$ is an irreducible Nash subset of $Y_1\times(-1,1)$ of dimension $e$. Consequently, 
$$
(Z\cap Y_1)\times(-1,1)=\bigcup_i(C_i\times(-1,1))=\bigcup_i\varphi(Z_i)\subset\varphi(Z\cap V)\subset(Z\cap Y_1)\times(-1,1),
$$
so $\varphi|_{Z\cap V}:Z\cap V\to(Z\cap Y_1)\times(-1,1)$ is surjective, as required.
\end{proof}

\subsection{Nash manifolds with corners.}\label{angoliPre}
We next study some properties of Nash manifolds with corners. Let $\Qq\subset\R^n$ be a Nash manifold with corners. The set of {\em internal points} of $\Qq$ is $\Int(\Qq):=\Sth(\Qq)$, that is, the set of {\em smooth points of $\Qq$}, which correspond to the points $x\in\Qq$ for which there exists an open semialgebraic neighborhood $W^x\subset\R^n$ of $x$ such that the intersection $W^x\cap\Qq$ is a Nash manifold. The {\em boundary} $\partial\Qq$ of $\Qq$ is $\partial\Qq:=\Qq\setminus\Int(\Qq)=\Qq\setminus\Sth(\Qq)$. The set $\Sth(\Qq)$ of smooth points of $\Qq$ is by \cite{st} a semialgebraic subset of $\R^n$. Thus, both $\Int(\Qq)=\Sth(\Qq)$ and $\partial \Qq=\Qq\setminus\Sth(\Qq)$ are semialgebraic subsets of $\R^n$.

A Nash manifold with corners $\Qq\subset\R^n$ is locally closed. Consequently, {\em $\Qq$ is a closed Nash submanifold with corners of the Nash manifold $\R^n\setminus(\cl(\Qq)\setminus\Qq)$.} In \cite[Thm.1.11]{fgr} it is shown that $\Qq$ is a closed subset of an affine Nash manifold of its same dimension. Recall that a Nash subset $Y$ of a Nash manifold $M\subset\R^n$ {\em has only Nash normal-crossings in $M$} if for each point $y\in Y$ there exists an open semialgebraic neighborhood $U\subset M$ such that $Y\cap U$ is a Nash normal-crossings divisor of $U$.

\begin{thm}[{\cite[Thm.1.11]{fgr}}]\label{corners0}
Let $\Qq\subset\R^n$ be a $d$-dimensional Nash manifold with corners. There exists a $d$-dimensional Nash manifold $M\subset\R^n$ that contains $\Qq$ as a closed subset and satisfies:
\begin{itemize}
\item[\rm{(i)}] The Nash closure $Y$ of $\partial\Qq$ in $M$ has only Nash normal-crossings in $M$ and $\Qq\cap Y=\partial\Qq$.
\item[\rm{(ii)}] For every $x\in\partial\Qq$ the smallest analytic germ that contains the germ $\partial\Qq_x$ is $Y_x$.
\item[\rm{(iii)}] $M$ can be covered by finitely many open semialgebraic subsets $U_i$ (for $i=1,\ldots,r$) equipped with Nash diffeomorphisms $u_i:=(u_{i1},\dots,u_{id}):U_i\to\R^d$ such that:
$$\hspace{-3mm}
\begin{cases}
\text{$U_i\subset\Int(\Qq)$ or $U_i\cap\Qq=\varnothing$},&\text{\!if $U_i$ does not meet $\partial\Qq$,}\\ 
U_i\cap\Qq=\{u_{i1}\ge0,\dots,u_{ik_i}\ge0\},&\text{\!if $U_i$ meets $\partial\Qq$ (for a suitable $k_i\ge1$).} 
\end{cases}
$$
\end{itemize}
\end{thm} 

The Nash manifold $M$ is called a {\em Nash envelope of $\Qq$}. In general, it is not guaranteed that the Nash closure $Y$ of $\partial\Qq$ in $M$ is a Nash normal-crossings divisor of $M$ as we show next.

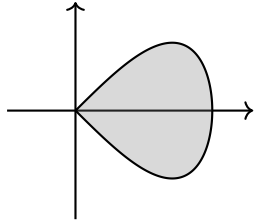
\begin{figure}[!ht]
\begin{center}
\begin{tikzpicture}[scale=1.8]
\draw[thick,->] (-0.5, 0) -- (1.3, 0) ;
\draw[thick,->] (0, -0.8) -- (0, 0.8) ;
\draw [thick,scale=1,domain=0:1, draw=black, fill=gray!100, fill opacity=0.3,smooth,samples=1000]
(0, 0)
-- plot ({\x}, {sqrt(\x*\x-\x*\x*\x*\x)})
-- (1, 0);
\draw [thick,scale=1,domain=0:1, draw=black, fill=gray!100, fill opacity=0.3,samples=1000]
(0, 0)
-- plot ({\x},{ -sqrt(\x*\x-\x*\x*\x*\x)})
-- (1, 0);
\end{tikzpicture}
\end{center}
\caption{\small{The teardrop.}}
\label{teardrop}
\end{figure}

\begin{example}
The {\em teardrop} $\Qq:=\{\x\geq0,\y^2\leq\x^2-\x^4\}\subset\R^2$ is a Nash manifold with corners (Figure \ref{teardrop}). Given any open semialgebraic neighborhood $M$ of $\Qq$ in $\R^2$ the Nash closure of $\partial\Qq$ in $M$ is not a Nash normal-crossings divisor. 
\hfill$\sqbullet$
\end{example} 

We define now Nash manifolds with divisorial corners. 

\begin{defn}
A Nash manifold with corners $\Qq\subset\R^n$ is a Nash manifold with {\em divisorial corners} if there exists a Nash envelope $M\subset\R^n$ such that the Nash closure of $\partial\Qq$ in $M$ is a Nash normal-crossings divisor.
\end{defn}

A {\em facet} of a Nash manifold with corners $\Qq\subset\R^n$ is the (topological) closure in $\Qq$ of a connected component of $\Sth(\partial\Qq)$. As $\partial\Qq=\Qq\setminus\Sth(\Qq)$ is semialgebraic, the facets are semialgebraic and finitely many. The non-empty intersections of facets of $\Qq$ are the {\em faces} of $\Qq$. In \cite{fgr} the following characterization for Nash manifolds with divisorial corners is shown:

\begin{thm}[{\cite[Thm.1.12, Cor.6.5]{fgr}}]\label{divisorialPre}
Let $\Qq\subset\R^n$ be a $d$-dimensional Nash manifold with corners. The following assertions are equivalent:
\begin{itemize}
\item[\rm{(i)}] There exists a Nash envelope $M\subset\R^n$ where the Nash closure of $\partial\Qq$ is a Nash normal-crossings divisor.
\item[\rm{(ii)}] Every facet $\Ff$ of $\Qq$ is contained in a Nash manifold $X\subset\R^n$ of dimension $d-1$.
\item[\rm{(iii)}] The number of facets of $\Qq$ that contain every given point $x\in\partial\Qq$ coincides with the number of connected components of the germ $\Sth(\partial\Qq)_x$.
\item[\rm{(iv)}] All the facets of $\Qq$ are Nash manifold with divisorial corners. 
\end{itemize}
If that is the case, the Nash manifold $M$ in {\rm(i)} can be chosen such that the Nash closure in $M$ of every facet $\Ff$ of $\Qq$ meets $\Qq$ exactly along $\Ff$.
\end{thm}

Note that properties (ii), (iii) and (iv) are intrinsic properties of $\Qq$ and do not depend on the Nash envelope $M$. The faces of a Nash manifold with divisorial corners are again Nash manifolds with divisorial corners. If a Nash envelope $M\subset\R^n$ of $\Qq$ satisfies (one of the equivalent) conditions of Theorem \ref{divisorialPre}, every open semialgebraic neighborhood $M'\subset M$ of $\Qq$ satisfies such conditions. For the rest of this article and the statements in the Introduction, we make the following assumption: 

\begin{assumption}
A Nash manifold with corners means a Nash manifold with divisorial corners.
\end{assumption}

\subsubsection{Nash equations for the facets of a Nash manifold with corners.}\label{envelope}
Let $\Qq\subset\R^n$ be a Nash manifold with corners. The following result provides suitable Nash equations for the facets of $\Qq$. The proof is strongly inspired by the proof of \cite[Lem.4.3]{fe3}, however it needs some modifications and we include it in full detail for the sake of the reader. 

\begin{lem}[Nash equations for the facets]\label{equation}
Let $M$ be a Nash envelope of the Nash manifold with corners $\Qq$ such that the Nash closure $Y$ of $\partial\Qq$ is a Nash normal-crossings divisor and $\Qq\cap Y=\partial\Qq$. Let $Y_1$ be a Nash irreducible component of $Y$. Then, after shrinking $M$ if necessary, there exists a Nash function $h_1:M\to\R$ such that: 
\begin{itemize}
\item $Y_1=\{h_1=0\}
$ and $d_xh_1:T_xM\to\R$ is surjective for each $x\in Y_1$. 
\item $\Hh_1:=h_1^{-1}([0,+\infty))$ is a Nash manifold with boundary that contains $\Qq$ as a closed subset. In addition, $\partial\Hh_1=Y_1$ and $\Int(\Hh_1)=h_1^{-1}((0,+\infty))$.
\end{itemize}
\end{lem}
\begin{proof}
The proof is conducted in several steps:

\noindent{\sc Step 1.} We construct first an ${\mathcal S}^2$ function $h_1^*$ on $M$ such that $Y_1\subset\{h_1^*=0\}$ and $d_xh_1^*(v)>0$ for each $x\in Y_1\cap\Qq$ and each non-zero vector $v\in T_xM$ pointing `inside $\Qq$'.

By Theorem \ref{corners0} there exist finitely many open semialgebraic subsets $U_i$ (for $i=1,\ldots,r$) equipped with Nash diffeomorphisms $u_i:=(u_{i1},\dots,u_{id}):U_i\to\R^d$ such that:
$$\hspace{-3mm}
\begin{cases}
\text{$U_i\subset\Int(\Qq)$ or $U_i\cap\Qq=\varnothing$},&\text{\!if $U_i$ does not meet $\partial\Qq$,}\\ 
U_i\cap\Qq=\{u_{i1}\ge0,\dots,u_{ik_i}\ge0\},&\text{\!if $U_i$ meets $\partial\Qq$ (for a suitable $k_i\ge1$).} 
\end{cases}
$$

Assume that $Y_1$ meets $U_i$ exactly for $i=1,\ldots,s$ for some $s\leq r$. We reorder the variables $u_{ij}$ in order to guarantee that $Y_1\cap U_i=\{u_{i1}=0\}$ and $\Qq\cap U_i\subset\{u_{i1}\geq0\}$ for $i=1,\ldots,s$. Let $\{\theta_i\}_{i=1}^{s+1}:M\to[0,1]$ be an ${\mathcal S}^2$ partition of unity subordinated to the finite open semialgebraic covering $\{U_i\}_{i=1}^s\cup\{U_{s+1}:=M\setminus Y_1\}$ of $M$ and consider the ${\mathcal S}^2$ function $h_1^*:=\sum_{i=1}^s\theta_iu_{i1}$. It holds $Y_1\subset\{h_1^*=0\}$. 

Fix $x\in Y_1\cap\Qq$ and let $v\in T_xM$ be a non-zero vector pointing `inside $\Qq$', that is, $d_xu_{i1}(v)>0$ if $x\in U_i$. We have $\cl(\{\theta_i>0\})\subset U_i$, $x\not\in U_{s+1}$ and $u_{i1}(x)=0$ if $x\in U_i$, so
$$
d_xh_1^*=\sum_{x\in U_i}u_{i1}(x)d_x\theta_i+\sum_{x\in U_i}\theta_i(x)d_xu_{i1}=\sum_{x\in U_i}\theta_i(x)d_xu_{i1}\ \leadsto\ d_xh_1^*(v)=\sum_{x\in U_i}\theta_i(x)d_xu_{i1}(v)>0
$$
because $\sum_{x\in U_i}\theta_i(x)=1$, $\theta_i(x)\geq0$ and $d_xu_{i1}(v)>0$ if $x\in U_i$. 

\noindent{\sc Step 2.} By \cite[Prop.8.2]{bfr} there exists a Nash function $h_1'$ on $M$ close to $h_1^*$ in the ${\mathcal S}^2$ topology such that $Y_1\subset\{h_1'=0\}$ and $d_xh_1'(v)>0$ for each $x\in Y_1\cap\Qq$ and each non-zero vector $v\in T_xM$ pointing `inside $\Qq$'. We claim: \em there exists an open semialgebraic neighborhood $W\subset M$ of $ Y_1\cap\Qq$ such that $\Int(\Qq)\cap W\subset\{h_1'>0\}\cap W$ and $\{h_1'=0\}\cap W= Y_1$\em.

Pick a point $x\in Y_1$ and assume $x\in U_1$. As $h_1'$ vanishes identically at $ Y_1$, we may write $h_1'|_{U_1}=u_{11}a_1$, where $a_1$ is a Nash function on $U_1$. Pick $y\in Y_1\cap U_1\cap\Qq$ and observe that $d_yh_1'=a_1(y)d_yu_{11}$. Let $v\in T_yM$ be a non-zero vector pointing `inside $\Qq$'. As $d_yu_{11}(v)>0$ and $d_yh_1'(v)>0$, we deduce $a_1(y)>0$. Define $W_1:=\{a_1>0\}\subset U_1$ and notice that $ Y_1\cap U_1\cap\Qq\subset W_1$, $\Int(\Qq)\cap W_1\subset\{h_1'>0\}\cap W_1$ and $\{h_1'=0\}\cap W_1= Y_1\cap W_1$. Construct analogously $W_2,\dots, W_s$ and observe that $W:=\bigcup_{i=1}^sW_i$ satisfies the required properties.

Substitute $M$ by $M\setminus(Y_1\setminus W)$, which is an open semialgebraic subset of $M$ that contains $\Qq$ as a closed subset. Substitute $Y$ by the Nash closure of $\partial\Qq$ in the new $M$ and $Y_1$ by the irreducible component of the new $Y$ that contains the facet $Y_1\cap\Qq$. By \cite[Prop.II.5.3]{sh} $ Y_1$ is a Nash subset of $M$.

\noindent{\sc Step 3.} Next, we construct $h_1$. If $W=M$, it is enough to set $h_1:=h_1'$. Suppose $W\neq M$. Let $\veps_0$ be a (continuous) semialgebraic function whose value is $1$ on $ Y_1$ and $-1$ on $M\setminus W$. Let $\veps$ be a Nash approximation of $\veps_0$ such that $|\veps-\veps_0|<\frac{1}{2}$. Then
$$
\veps(x)\begin{cases}
>\frac{1}{2}&\text{if $x\in Y_1$,}\\
<-\frac{1}{2}&\text{if $x\in M\setminus W$.}
\end{cases}
$$
Thus, $\{\veps>0\}\subset W$ is an open semialgebraic neighborhood of $ Y_1$ in $M$. Let $f$ be a Nash equation of $ Y_1$ in $M$. Substituting $f$ by $\frac{f^2}{\veps^2+f^2}$, we may assume that $f$ is non-negative and $f(x)=1$ if $\veps(x)=0$. Consider the (continuous) semialgebraic function on $M$ given by
$$
\delta(x):=
\begin{cases}
1&\text{if $\veps(x)>0$,}\\
\frac{1}{f(x)}&\text{if $\veps(x)\leq0$}.
\end{cases}
$$
Let $g$ be a Nash function on $M$ such that $\delta<g$ (see \cite[Prop.2.6.2]{bcr}, after embedding $M$ in $\R^{n+1}$ as a closed subset). Consider the Nash function $h_1:=h_1'+f^2g^2(h_1'^2+1)$ and let us prove that it satisfies the required conditions. 

\noindent{\sc Step 4.} We claim: {\em $h_1$ is positive on $\Int(\Qq)$\em. }

Let $x\in\Int(\Qq)$. If $h_1'(x)>0$, then $h_1(x)>0$. If $h_1'(x)\leq0$, then $x\not\in W$ (because $\Int(\Qq)\cap W\subset\{h_1'>0\}\cap W$). Thus, $\veps(x)\leq0$ and
\begin{multline*}
h_1(x)=h_1'(x)+g^2(x)f^2(x)(h_1'^2(x)+1)\\
>h_1'(x)+\frac{1}{f^2(x)}f^2(x)(h_1'^2(x)+1)=h_1'^2(x)+h_1'(x)+1>0.
\end{multline*}

\noindent{\sc Step 5.} It holds: {\em There exists an open semialgebraic neighborhood $W'\subset W$ of $Y_1$ such that $\{h_1=0\}\cap W'= Y_1$, $\Int(\Qq)\cap W'\subset\{h_1>0\}\cap W'$ and the differential $d_xh_1:T_xM\to\R$ is surjective for all $x\in Y_1$\em.}

Recall that $W=\bigcup_{i=1}^sW_i$. We have seen in {\sc Step 2} that there exists a Nash function $a_1$ on $U_1$ such that $h_1'|_{U_1}=u_{11}a_1$ and $ Y_1\cap U_1\subset W_1:=\{a_1>0\}$. As $f$ vanishes identically at $ Y_1$, we deduce $f|_{U_1}=u_{11}b_1$ where $b_1$ is a Nash function on $U_1$. Consequently, 
$$
h_1|_{U_1}=u_{11}a_1+g^2|_{U_1}u_{11}^2b_1^2(u_{11}^2a_1^2+1)=u_{11}(a_1+g^2|_{U_1}u_{11}b_1^2(u_{11}^2a_1^2+1))
$$
and $d_xh_1=a_1(x)d_xu_{11}=d_xh_1'$ for $x\in Y_1\cap U_1$. 
Define 
$$W_1':=\{a_1+g^2|_{U_1}u_{11}b_1^2(u_{11}^2a_1^2+1)>0\}\cap W_1,
$$
which is an open semialgebraic subset of $M$. We have $ Y_1\cap U_1\subset W_1'$ and 
$$
\Int(\Qq)\cap W_1'\subset\{h_1>0\}\cap W_1'.
$$
We construct analogously $W_2',\ldots,W_s'$ and observe that the open semialgebraic subset $W':=\bigcup_{i=1}^sW_i'\subset M$ is an open neighborhood of $ Y_1$ that satisfies $\{h_1=0\}\cap W'= Y_1$, $\Int(\Qq)\cap W'\subset\{h_1>0\}\cap W'$ and $d_xh:T_xM\to\R$ is surjective for all $x\in Y_1$. 

Observe that $\Hh_1:=h_1^{-1}([0,+\infty))$ is a Nash manifold with boundary, because it is a semialgebraic set and a smooth manifold with boundary (the latter holds by the implicit function theorem, because $0$ is not a critical value of $h_1$). Consequently, $M':=\{h_1>0\}\cup W'$ and $h_1|_{M'}$ satisfy all the required conditions. 
\end{proof}

\begin{cor}\label{Q}
Let $Y_1,\ldots,Y_\ell$ be the irreducible components of $Y$. After shrinking the manifold $M$, there exist Nash functions $h_i:M\to\R$ such that $Y_i=\{h_i=0\}$, $d_xh_i:T_xM\to\R$ is surjective for each $x\in Y_i$, $\Qq=\{h_1\geq0,\ldots,h_\ell\geq0\}$ and $\Int(\Qq)=\{h_1>0,\ldots,h_\ell>0\}$.
\end{cor}
\begin{proof}
We keep the notations of Lemma \ref{equation}. We may assume that $M$ is a closed subset of $\R^n$. Only the equality $\Ss:=\{h_1\geq0,\ldots,h_\ell\geq0\}=\Qq$ (after shrinking $M$ is necessary) requires some comment. Let $\Cc_1,\ldots,\Cc_k$ be the connected components of $\Ss$ that do not meet $\Qq$. After substituting $M$ by $M\setminus\bigcup_{j=1}^k\Cc_k$ we may assume that all the connected components of $\Ss$ meet $\Qq$. As $\Qq\setminus Y=\Int(\Qq)$ and $\Qq$ is a closed subset of $M$, we deduce that $\Qq\setminus Y$ is an open and closed subset of $M\setminus Y$. Analogously, $\Ss\setminus Y=\{h_1>0,\ldots,h_\ell>0\}$ is an open and closed subset of $M\setminus Y$. Thus, both $\Qq\setminus Y$ and $\Ss\setminus Y$ are unions of connected components of $M\setminus Y$. 

Let $C$ be the connected component of $\Ss\setminus Y$. If $\cl(C)\cap\Qq\subset\partial\Qq\subset Y$, there exists an index $i=1,\ldots,\ell$ such that $\cl(C)\subset\{h_i\geq0\}$, whereas $\Qq\subset\{h_i\leq0\}$, which is a contradiction. Thus, $C\cap\Int(\Qq)=(\cl(C)\cap\Qq)\setminus Y\neq\varnothing$. Consequently, $C\setminus Y$ is a connected component of $\Qq\setminus Y$, so $\Ss\setminus Y\subset\Qq\subset\Ss$. We conclude $\Ss=\cl(\Ss\setminus Y)\subset\Qq\subset\Ss$ and $\Ss=\Qq$, as required.
\end{proof}

The following result is the counterpart of the previous ones.

\begin{lem}[Finite intersections of Nash manifolds with boundary]\label{inthi}
Let $M\subset\R^n$ be a Nash manifold and $Y\subset M$ a Nash normal-crossings divisor. Let $Y_1,\ldots,Y_\ell$ be the irreducible components of $Y$ and $h_i:M\to\R$ a Nash equation of $Y_i$ such that $d_xh_i:T_xM\to\R$ is surjective for each $x\in Y_i$ and each $i=1,\ldots,\ell$. Assume that $\Qq:=\{h_1\geq0,\ldots,h_\ell\geq0\}$ is non-empty. Then $\Qq$ is a Nash manifold with corners such that $\partial\Qq=Y\cap\Qq$, the Nash closure of its boundary $\partial\Qq$ is $Y$ and $\Int(\Qq)=\{h_1>0,\ldots,h_\ell>0\}$. In particular, $\Qq$ is the intersection of the Nash manifolds with boundary $\Hh_i:=\{h_i\geq0\}$ for $i=1,\ldots,\ell$.
\end{lem}
\begin{proof}
Observe first that $\Hh_i:=\{h_i\geq0\}\subset M$ is a Nash manifold with boundary, $\partial\Hh_i=Y_i$ for $i=1,\ldots,\ell$ and $\Qq=\bigcap_{i=1}^\ell \Hh_i$. We check next: {\em $\Qq$ is a Nash manifold with corners}.

Pick $x\in\Qq$. If $h_i(x)>0$ for $i=1,\ldots,\ell$, then $x$ belongs to the open semialgebraic subset $\{h_1>0,\ldots,h_\ell>0\}$ of $M$ contained in $\Qq$. Suppose $h_1(x)=0,\ldots,h_s(x)=0,h_{s+1}(x)>0,\ldots,h_\ell(x)>0$ for some $1\leq s\leq\ell$. As $Y$ is a Nash normal-crossings divisor, the linear forms $d_xh_1,\ldots,d_xh_s$ are linearly independent and there exists an open semialgebraic neighborhood $U\subset M$ of $x$ endowed with a Nash diffeomorphism $u:=(h_1,\ldots,h_s,u_{s+1},\ldots,u_d):U\to\R^d$ such that $u(x)=0$ and $Y\cap U=\{h_1\cdots h_s=0\}\cap U$. Observe that $\Qq\cap U=\{h_1\geq0,\ldots,h_s\geq0\}\cap U$. 

The boundary of $\Qq$ is 
$$
\partial\Qq=\bigcup_{i=1}^r\{h_1\geq0,\ldots,h_{i-1}\geq0,h_i=0,h_{i+1}\geq0,\ldots,h_\ell\geq0\}
$$
and its Nash closure in $M$ is $Y$, so $\Qq$ is a Nash manifold with corners, as required.
\end{proof}

\subsection{Connected components of a non-singular real algebraic set}\label{ccnsras}
The following lemma allows the reader to understand connected Nash manifolds as connected components of irreducible non-singular real algebraic sets whose connected components are pairwise diffeomorphic. Recall that a real algebraic set $X\subset\R^n$ is non-singular at a point $x\in X$ if and only if the localized quotient ring $(\R[\x]/{\mathcal I}(X))_{\gtm_x}$, where ${\mathcal I}(X)$ is the ideal of all polynomials of $\R[\x]:=\R[\x_1,\ldots,\x_n]$ that are identically zero on $X$ and $\gtm_x$ is the maximal ideal of all polynomial of $\R[\x]$ that vanishes at $x$, is a local regular ring. 

\begin{lem}\label{cnma}
Let $M\subset\R^n$ be a $d$-dimensional connected Nash manifold. Then there exists a non-singular irreducible real algebraic set $X\subset\R^{n+k}$ whose connected components are Nash diffeomorphic to $M$. In addition,
\begin{itemize}
\item[(i)] Each connected component of $X$ projects onto $M$.
\item[(ii)] Given any two of the connected components of $X$, there exists a linear involution of $\R^{n+k}$ that swaps them and keeps $X$ invariant.
\item[(iii)] If $M$ is compact, we may assume that $X$ has at most two connected components.
\end{itemize}
\end{lem}
\begin{proof}
Consider the Nash map $f:M\hookrightarrow\R^n,\ x\mapsto x$. By Artin-Mazur's description \cite[Thm.8.4.4]{bcr} of Nash maps there exist $s\geq1$ and a non-singular irreducible real algebraic set $Z\subset\R^{2n+s}$ of dimension $d$, a connected component $N$ of $Z$ and a Nash diffeomorphism $g:M\to N$ such that the following diagram is commutative:
$$
\xymatrix{
Z\ar@{^{(}->}[rr]&&\R^n\times\R^n\times\R^s\equiv\R^m\ar[dd]_{\pi_2}\ar[ddll]_{\pi_1}\\
N\ar@{^{(}->}[u]&&\\
M\ar@{<->}[u]_{\cong}^g\ar[rr]^{f}&&\R^n
}
$$
We denote the projection of $\R^n\times\R^n\times\R^s$ onto the first space $\R^n$ with $\pi_1$ and the projection of $\R^n\times\R^n\times\R^s$ onto the second space $\R^n$ with $\pi_2$. Write $m:=2n+s$. By \cite{mo} applied to $N$ and the union of the remaining connected components of $Z$ there exist finitely many polynomials $P_1,\ldots,P_\ell,Q_1,\ldots,Q_\ell\in\R[\x]:=\R[\x_1,\ldots,\x_m]$ such that each $Q_j$ is strictly positive on $\R^m$ and 
$$
N=Z\cap\Big\{\sum_{j=1}^\ell P_j\sqrt{Q_j}>0\Big\}.
$$
If $M$ is compact, also $N$ is compact and by \cite[Lem.6.1]{cf1} there exists a polynomial $P\in\R[\x]$ such that $N=Z\cap\{P>0\}$. 

A general $N$ is the projection of the real algebraic set
$$
Y:=\Big\{(z,y,t)\in Z\times\R^\ell\times\R:\ \Big(\sum_{j=1}^\ell P_jy_j^2\Big)t^2-1=0,\ y_j^4-Q_j=0\ \text{for $j=1,\ldots,\ell$}\Big\}
$$
under $\pi:\R^m\times\R^\ell\times\R\to\R^m,\ (z,y,t)\mapsto z$. If $M$ is compact, $N$ is the projection of the real algebraic set $Y:=\{(z,t)\in Z\times\R:\ P(z)t^2-1=0\}$ under $\pi:\R^m\times\R\to\R^m,\ (z,t)\mapsto z$. Denote
$$
h:=\begin{cases}
P&\text{if $M$ is compact,}\\
\sqrt{\sum_{j=1}^\ell P_j\sqrt{Q_j}}&\text{if $M$ is non-compact.}
\end{cases}
$$

Fix $\epsilon:=(\epsilon_1,\ldots,\epsilon_\ell,\epsilon_{\ell+1})\in\{-1,1\}^{\ell+1}$ (where $\ell=0$ if $N$ is compact) and let $N_\epsilon:=Y\cap\{\epsilon_1y_1>0,\ldots,\epsilon_\ell y_\ell>0,\epsilon_{\ell+1}t>0\}$. Consider the Nash diffeomorphism 
$$
\varphi_\epsilon:N\to N_\epsilon,\ z\mapsto\Big(z,\epsilon_1\sqrt[4]{Q_1(z)},\ldots,\epsilon_\ell\sqrt[4]{Q_\ell(z)},\epsilon_{\ell+1}\frac{1}{h(z)}\Big)
$$
whose inverse map is the restriction of the projection $\pi$ to $N_\epsilon$.

Observe that $\{N_\epsilon\}_{\epsilon\in\{-1,1\}^{\ell+1}}$ is the collection of the connected components of $Y$. As $\pi(N_\epsilon)=N$ and using the diagram above, we deduce
$$
(\pi_2\circ\pi)(N_{\epsilon})=\pi_2(N)=(f\circ\pi_1)(N)=f(M)=M.
$$
In addition, each $N_\epsilon$ is Nash diffeomorphic to $M$ and for $\epsilon\neq\epsilon'$ the linear involution 
$$
\phi_{\epsilon,\epsilon'}:\R^m\times\R^\ell\times\R\to\R^m\times\R^\ell\times\R,\ (x,y,t)\mapsto(x,(\epsilon\cdot\epsilon')\cdot(y,t))
$$
induces an involution of $Y$ such that $\phi_{\epsilon,\epsilon'}(N_\epsilon)=N_{\epsilon'}$. We have denoted
$$
(\epsilon\cdot\epsilon')\cdot(y,t):=(\epsilon_1\epsilon_1'y_1,\ldots,\epsilon_\ell\epsilon_\ell'y_\ell,\epsilon_{\ell+1}\epsilon_{\ell+1}'t).
$$
Let us check: {\em $Y$ is a non-singular real algebraic set}. 

Pick a point $(z,y,t)\in Y$. As $z$ is a non-singular point of the real algebraic set $Z$ of dimension $d$, there exists by the Jacobian criterion \cite[Prop.3.3.10]{bcr} polynomials $f_1,\ldots,f_{m-d}\in{\mathcal I}(Z)$ and an open semialgebraic set $U\subset\R^m$ of $z$ such that 
$$
Z\cap U=\{u\in U:\ f_1(u)=0,\ldots,f_{m-d}(u)=0\}
$$
and the rank of the Jacobian matrix $(\frac{\partial f_i}{\partial\z_k}(z))_{1\leq i\leq m-d,1\leq k\leq m}$ is $m-d$. Define $f_{m+j-d}(\z,\y_j):=\y_j^4-Q_j(\z)$ for $j=1,\ldots,\ell$ and $f_{m+\ell+1-d}(\z,\y,\t):=\big(\sum_{j=1}^\ell P_j(\z)\y_j^2\big)\t^2-1$ and denote Kroneker's $\delta$ for $1\leq j,p\leq\ell$ with $\delta_{jp}$. We have
$$
Y\cap(U\times\R^\ell\times\R)=\{(u,v,w)\in U\times\R^\ell\times\R:\ f_q(u,v,w)=0,\ q=1,\ldots,m+\ell+1-d\}
$$
and the Jacobian matrix of $f:=(f_1,\ldots,f_{m+\ell+1-d})$ at $(z,y,t)$ is
\begin{equation*}
J_f(z,y,t)=\left[\begin{array}{c|c|c}
\big(\frac{\partial f_i}{\partial\z_k}(z)\big)_{\substack{1\leq i\leq m-d,\\1\leq k\leq m}}&0&0\\\hline
\big(\frac{\partial f_i}{\partial\z_k}(z,y))_{\substack{m+1-d\leq i\leq m+\ell-d,\\1\leq k\leq m}}&(4y_j^3\delta_{jp})_{1\leq j,p\leq\ell}&0\\\hline
\big(\frac{\partial f_{m+\ell+1-d}}{\partial\z_k}(z,y,t)\big)_{1\leq k\leq m}&\big(\frac{\partial f_{m+\ell+1-d}}{\partial\y_j}(z,y,t)\big)_{1\leq j\leq\ell}&2t\sum_{j=1}^\ell P_j(z)y_j^2
\end{array}\right]
\end{equation*}
and has rank $m+\ell+1-d$ (because $y_1,\ldots,y_\ell\neq0$ and $t=\pm h(z)\neq0$). By the Jacobian criterion \cite[Prop.3.3.10]{bcr} $(z,y,t)$ is a non-singular point of the real algebraic set $Y$. Consequently, $Y$ is a non-singular real algebraic set.

Let $X$ be the irreducible component of $Y$ that contains $N_{(1,\ldots,1)}$, which has at most two connected components if $M$ is compact. Then $k:=m+\ell+1-n$ and the non-singular real algebraic set $X$ satisfies the requirements in the statement.
\end{proof}

\section{Folding Nash manifolds to build Nash manifolds with corners}\label{s3}

In \cite[Thm.1.12, Cor.6.5]{fgr} it is shown that a Nash manifold $\Qq\subset\R^n$ with corners is contained as a closed subset in a Nash manifold $M\subset\R^n$ of the same dimension and the Nash closure of its boundary $\partial\Qq$ is a Nash normal-crossings divisor of $M$. In this section we show that the Nash manifold $M$ can be `folded' to build the Nash manifold with corners $\Qq$, that is, there exists a surjective Nash map $f:M\to\Qq$ such that its restriction to $\Qq$ is a semialgebraic homeomorphism close to the identity map, its restriction to $\Int(\Qq)$ is a Nash diffeomorphism close to the identity map, and $f$ preserves the natural semialgebraic partition of the boundary $\partial\Qq$ (Theorem \ref{fold}). We use standard tools for smooth manifolds with boundary adapted to the Nash category (equations of the boundary, collars and doubling of smooth manifolds with boundary, etc.). These standard constructions in the Nash setting were mainly developed in \cite[Ch.VI]{sh} (compact case) and in \cite[\S4]{fe3} (general case). However, we need to adapt (taking profit of the results introduced in Section \ref{s2}) some of the statements and their proofs to obtain constructions compatible with an assigned Nash normal-crossings divisor. We also show that our construction is canonical and does not depend on the order chosen to `fold $M$ along the facets of $\Qq$' (see Section \ref{s4}).

\subsection{Folding one component of the boundary.}

Let $\Hh$ be a $d$-dimensional Nash manifold with boundary, let $M\subset\R^n$ be a $d$-dimensional Nash manifold that contains $\Hh$ as a closed subset and assume that the boundary $\partial\Hh$ of $\Hh$ is a Nash subset of $M$. Let $Y$ be a Nash normal-crossings divisor of $M$ such that the Nash subset $\partial\Hh$ of $M$ is a union of irreducible components of $Y$ (maybe after shrinking $M$ if necessary). Let $Y_1,\ldots,Y_r$ be the irreducible components of $Y$ that meet $\partial\Hh$ but are not contained in $\partial\Hh$ and let $Y_{r+1},\ldots,Y_s$ be the irreducible components of $Y$ that do not meet $\partial\Hh$. Let $h:M\to\R$ be a Nash equation of $\partial\Hh$. Observe that $h_i:=h|_{Y_i}$ is a Nash equation of $Y_i\cap\partial\Hh$ such that $Y_i\cap\Int(\Hh)=\{h_i>0\}$ and $d_xh_i:T_x Y_i\to\R$ is surjective for all $x\in Y_i\cap\partial\Hh$. Thus, $Y_i\cap\Hh$ is a Nash manifold with boundary $Y_i\cap\partial\Hh$ that is contained in $Y_i$ as a closed subset. In addition,
\begin{equation}\label{yi}
D(Y_i\cap\Hh)=\{(x,t)\in (Y_i\cap\Hh)\times\R:\, t^2-h_i(x)=0\}=(Y_i\times\R)\cap D(\Hh),
\end{equation}
is the Nash double of $Y_i\cap\Hh$. We show that there exists a Nash embedding of $M$ into $D(\Hh)$ that maps $Y_i$ into $D(Y_i\cap\Hh)=(Y_i\times\R)\cap D(\Hh)$ for $i=1,\ldots,r$ and maps $Y_j$ onto $(Y_j\times\R)\cap D(\Hh)\cap\{t>0\}$ for $j=r+1,\ldots,s$. We begin with the following technical lemma. 

\begin{lem}\label{fold1}
Let $k\geq1$ and $0<a\leq 1$. Consider the ${\mathcal S}^{2k}$ function $f_{a,k}:=(1-(\s/a)^{2k})^{2k}\s+(1-(1-(\s/a)^{2k})^{2k})\sqrt{\s}$. Then: 
\begin{itemize}
\item $f_{a,k}$ is positive semidefinite and strictly increasing on $[0,a]$. 
\item The Taylor polynomial of $f_{a,k}$ at $s=0$ of degree $2k$ is $\s$.
\item The Taylor polynomials of $f_{a,k}$ and $\sqrt{\s}$ of degree $2k-1$ coincide at $s=a$, 
\item $f_{a,k}(s)\leq\sqrt{s}$ on $[0,1]$.
\end{itemize}
\end{lem}

\begin{center}
\begin{figure}[ht]
\begin{minipage}{0.48\textwidth}
\begin{center}
{\scriptsize\begin{tikzpicture}[scale=5]
\draw[->] (-0.05,0)--(0.6,0);
\draw[->](0,-0.05)--(0,0.75);
\draw [domain=0:0.54,red,smooth,line width=0.75pt,samples=500] plot({\x}, {\x*(1-16*\x*\x*\x*\x)*(1-16*\x*\x*\x*\x)*(1-16*\x*\x*\x*\x)*(1-16*\x*\x*\x*\x)+sqrt(\x)*(1-(1-16*\x*\x*\x*\x)*(1-16*\x*\x*\x*\x)*(1-16*\x*\x*\x*\x)*(1-16*\x*\x*\x*\x))});
\draw[densely dotted] (0.5,0)--(0.5,0.70710678);
\draw (0.5,0.70710678) node {\tiny{$\bullet$}};
\draw (0.5,-0.05) node {\small $0.5$};
\end{tikzpicture}}
\end{center}
\end{minipage}
\hfil
\begin{minipage}{0.48\textwidth}
\begin{center}
{\scriptsize\begin{tikzpicture}[scale=5]
\draw[->] (-0.05,0)--(0.6,0);
\draw[->](0,-0.05)--(0,0.75);
\draw [domain=0:0.54,blue, smooth,line width=0.75pt] plot({\x}, {\x*(1-16*\x*\x*\x*\x)*(1-16*\x*\x*\x*\x)*(1-16*\x*\x*\x*\x)*(1-16*\x*\x*\x*\x)*(1-16*\x*\x*\x*\x)*
(1-16*\x*\x*\x*\x)*(1-16*\x*\x*\x*\x)*(1-16*\x*\x*\x*\x)*(1-16*\x*\x*\x*\x)*(1-16*\x*\x*\x*\x)*(1-16*\x*\x*\x*\x)*(1-16*\x*\x*\x*\x)
+sqrt(\x)*(1-(1-16*\x*\x*\x*\x)*(1-16*\x*\x*\x*\x)*(1-16*\x*\x*\x*\x)*(1-16*\x*\x*\x*\x)*(1-16*\x*\x*\x*\x)*(1-16*\x*\x*\x*\x)
*(1-16*\x*\x*\x*\x)*(1-16*\x*\x*\x*\x)*(1-16*\x*\x*\x*\x)*(1-16*\x*\x*\x*\x)*(1-16*\x*\x*\x*\x)*(1-16*\x*\x*\x*\x))});
\draw[densely dotted] (0.5,0)--(0.5,0.70710678);
\draw (0.5,0.70710678) node {\tiny{$\bullet$}};
\draw (0.5,-0.05) node {\small $0.5$};
\end{tikzpicture}}
\end{center}
\end{minipage}
\caption{{\small Graphs of $f_a$ for $a=\frac{1}{2}$ and $k=2$ (left) and $k=6$ (right).}}
\end{figure}
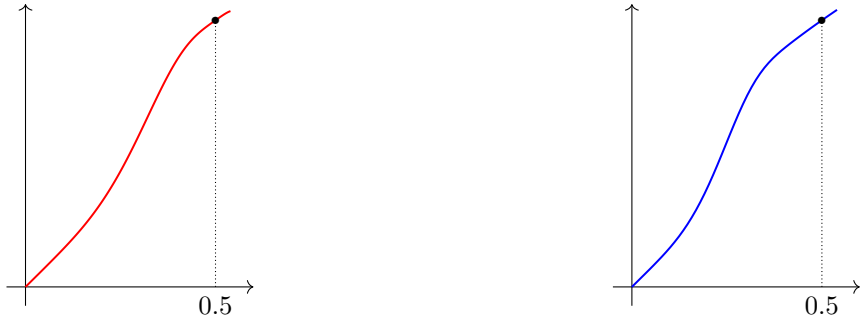
\end{center}

\begin{proof}
Write $f:=f_{a,k}$ to lighten notations. Using Newton's binomial, we have
$$
f=\sum_{\ell=0}^{2k}\binom{2k}\ell(-1)^\ell\Big(\frac{\s}{a}\Big)^{2k\ell}\s-\sum_{\ell=1}^{2k}\binom{2k}\ell(-1)^\ell\Big(\frac{\s}{a}\Big)^{2k\ell}\sqrt{\s},
$$
so $f$ is an ${\mathcal S}^{2k}$ function and the Taylor polynomial of $f$ at $s=0$ of degree $2k$ is $\s$. In addition,
$$
f=(1-(1+(\s/a)-1)^{2k})^{2k}\s+(1-(1-(1+(\s/a)-1)^{2k})^{2k})\sqrt{\s}
$$ 
and using Newton's binomial, we have
$$
f=\sqrt{\s}-\Big(\frac{\s}{a}-1\Big)^{2k}\Big(\sum_{\ell=1}^{2k}\binom{2k}\ell\Big(\frac{\s}{a}-1\Big)^{\ell-1}\Big)^{2k}(\s-\sqrt{\s}),
$$
so the Taylor polynomials of $f$ and $\sqrt{\s}$ of degree $2k-1$ at $s=a$ coincide.

Define $\sigma:=(1-(\s/a)^{2k})^{2k}$ and observe that $\sigma(0)=1$, $\sigma(a)=0$ and $\sigma([0,a])\subset[0,1]$, so both $\sigma$, $1-\sigma$ are positive semidefinite on $[0,a]$. Thus, $f=\sigma\s+(1-\sigma)\sqrt{\s}$ is positive semidefinite on $[0,a]$. As $0<a\leq 1$, it holds $\sqrt{s}-s>0$ for each $s\in(0,a)$, so the derivative
\begin{multline*}
f':=\Big(\s-\sqrt{\s}\Big)\sigma'+(1-\sigma)\frac{1}{2\sqrt{\s}}+\sigma\\
=(\sqrt{\s}-\s)4k^2(1-(\s/a)^{2k})^{2k-1}(\s/a)^{2k-1}(1/a)+(1-\sigma)\frac{1}{2\sqrt{\s}}+\sigma
\end{multline*}
is strictly positive on $(0,a)$ and $f$ is strictly increasing on $[0,a]$. Finally, as $s\leq\sqrt{s}$ on $[0,1]$, we deduce that $f(s)=\sigma(s)s+(1-\sigma(s))\sqrt{s}\leq\sqrt{s}$ for each $s\in[0,1]$, as required.
\end{proof}

We are ready to prove the embedding theorem. 

\begin{thm}[Nash embedding into the Nash double]\label{embedding}
Let $\Hh\subset\R^n$ be a $d$-dimensional Nash manifold with boundary that is a closed (semialgebraic) subset of $\R^n$, let $M\subset\R^n$ be a $d$-dimensional Nash manifold that contains $\Hh$ as a closed subset and assume that the boundary $\partial\Hh$ of $\Hh$ is a Nash subset of $M$. Let $Y$ be a Nash normal-crossings divisor of $M$ such that (after shrinking $M$ if necessary) $\partial\Hh$ in $M$ is a union of irreducible components of $Y$. Let $Y_1,\ldots,Y_r$ be the irreducible components of $Y$ that meet $\partial\Hh$ but are not contained in $\partial\Hh$ and let $Y_{r+1},\ldots,Y_s$ be the irreducible components of $Y$ that do not meet $\partial\Hh$. Let $D(\Hh)\subset\Hh\times\R$ be the Nash double of $\Hh$ and $\pi:D(\Hh)\to\Hh$ the restriction to $D(\Hh)$ of the projection $\Hh\times\R\to\Hh$ onto the first factor.
After shrinking $M$ once more if necessary, there exists a Nash embedding $\phi:M\to D(\Hh)$ that maps $\Hh$ onto $\Hh_+$ and satisfies:
\begin{itemize}
\item $\phi(Y_i)\subset D(Y_i\cap\Hh)$ for each $i=1,\ldots,r$, 
\item $\phi(Y_j)=(Y_j\times\R)\cap\Int(\Hh_+)$ for each $j=r+1,\ldots,s$, 
\item $\phi|_{\partial\Hh}=(\id_{\partial\Hh},0)$ and $\phi|_{\Hh}$ is close to $(\pi|_{\Hh_+})^{-1}$ with respect to the ${\mathcal S}^0$ topology,
\item There exist and open semialgebraic neighborhood $U$ of $\partial\Hh$ and a Nash retraction $\rho:U\to\partial\Hh$ compatible with $Y$ such that $M\setminus\Hh\subset U$ and $\phi|_{M\setminus\Hh}$ is close to $(\rho|_{M\setminus\Hh},0)$ with respect to the ${\mathcal S}^0$ topology.
\item There exists an open semialgebraic neighborhood $W\subset U$ of $\partial\Hh$ such that $\cl(W)\subset M$ and $\phi|_{\cl(W)\cup\Hh}:\cl(W)\cup\Hh\to D(\Hh)$ is proper.
\end{itemize}
\end{thm}

\begin{proof}
The proof is conducted in several steps:

\noindent{\sc Step 1.} We first construct suitable semialgebraic neighborhoods $U\subset M$ of $\partial\Hh$ and $V\subset D(\Hh)$ of $\partial\Hh\times\{0\}$. 

As $\partial\Hh$ is a Nash submanifold of $M$, its connected components coincide with its Nash irreducible components. In particular, they are pairwise disjoint Nash irreducible components $Y_{s+1},\ldots,Y_p$ of $Y$. By Proposition \ref{compretraction} (applied separately to each $Y_{s+1},\ldots,Y_p$) we find an open semialgebraic neighborhood $U\subset M$ of $\partial\Hh$ equipped with a Nash retraction $\rho :U\to\partial\Hh$ compatible with $Y$. We may assume that $U$ does not meet the irreducible components of $Y$ that do not meet $\partial\Hh$. Let $h:M\to\R$ be a Nash equation of $\partial\Hh$ such that $\Int(\Hh)=\{h>0\}$ and $d_xh:T_x\Hh\to\R$ is surjective for all $x\in\partial\Hh$. Define the Nash map $\varphi:=(\rho, h):U\to\partial\Hh\times\R$. By Lemma \ref{style} there exists, after shrinking $U$ if necessary, a strictly positive Nash function $\varepsilon:\partial\Hh\to (0,1)$ such that 
\begin{align*}
\varphi(U)&=\{(y,s)\in\partial\Hh\times\R:\ |s|<\varepsilon(y)\},\\
\varphi(U\cap\Hh)&=\{(y,s)\in\partial\Hh\times\R:\, 0\leq s<\varepsilon(y)\},
\end{align*}
and $\varphi:U\to\varphi(U)$ is a Nash diffeomorphism such that
$$
\varphi(Z\cap U)=\{(y,s)\in (Z\cap\partial\Hh)\times\R:\, |s|<\varepsilon(y)\}
$$
for each irreducible component $Z$ of $\Sing_\ell(Y)$ such that $Z\cap\partial\Hh\neq\varnothing$, $Z\not\subset\partial\Hh$ and $\ell\geq0$. Define $V:=\pi^{-1}(U\cap\Hh)=\{(x,t)\in(U\cap\Hh)\times\R:\ t^2-h(x)=0\}$.

\noindent{\sc Step 2.} Let $V':=\{(y,t)\in\partial\Hh\times\R:\, |t|<\sqrt{\varepsilon(y)}\}$. We want to prove: {\em The Nash map
$$
\psi:V\to\partial\Hh\times\R,\ (x,t)\mapsto (\rho(x),t)
$$
is a Nash diffeomorphism onto its image $V'$, such that 
$$
\psi(Z\cap U\cap\Hh)=\{(y,t)\in (Z\cap\partial\Hh)\times\R:\, 0\leq t<\sqrt{\varepsilon(y)}\}
$$
for each irreducible component $Z$ of $\Sing_\ell(Y)$ such that $Z\cap\partial\Hh\neq\varnothing$, $Z\not\subset\partial\Hh$ and $\ell\geq0$. }

(1) {\em $\psi$ is injective.} If $(x_1,t_1),(x_2,t_2)\in V$ satisfy $\psi(x_1,t_1)=\psi(x_2,t_2)$, then $\rho(x_1)=\rho(x_2)$ and $h(x_1)=t_1^2=t_2^2=h(x_2)$. Then we have $\varphi(x_1)=\varphi(x_2)$, so $x_1=x_2$. Thus, $(x_1,t_1)=(x_2,t_2)$. 

(2) $\psi(V)=V'$. Fix a point $(x,t)\in V$. Then, $x=\pi(x,t)\in U$ and $\varphi(x)=(\rho(x),h(x))\in\varphi(U)$, so $t^2=h(x)<\veps(\rho(x))$ and $\psi(x,t)\in V'$. Conversely, fix a point $(y,t)\in V'$. As $(y,t^2)\in\varphi(U)$, there exists a point $x\in U$ such that $\varphi(x)=(\rho(x),h(x))=(y,t^2)$. As $x\in U$ and $h(x)=t^2\geq0$, we have $x\in U\cap\Hh$, so $(x,t)\in V$ and $(y,t)=(\rho(x),t)=\psi(x,t)\in\psi(V)$. 

(3) {\em The differential $d_z\psi:T_zD(\Hh)\to T_{\rho(z)}\Hh\times\R$ is an isomorphism for each $z\in V$.} Write $z:=(x,t)$ and notice that
$$
T_zD(\Hh)=\{(v,r)\in T_x\Hh\times\R:\, d_xh(v)-2tr=0\}
$$
and $d_z\psi(v,r)=(d_x\rho(v),r)$. If $t\neq 0$, then $
d_z\psi(v,r)=(d_x\rho(v),\frac{1}{2t}d_xh(v))$. As $d_x\varphi=(d_x\rho,d_xh)$ is an isomorphism, also $d_z\psi$ is an isomorphism. If $t=0$, that is, $z=(x,0)\in\partial\Hh\times\R$, then 
$$
T_zD(\Hh)=\{(v,r)\in T_x\Hh\times\R:\, d_xh(v)=0\}=T_x\partial\Hh\times\R
$$
and $d_z\psi(v,r)=(v,r)$ because $\rho|_{\partial\Hh}=\id_{\partial\Hh}$. Hence, $d_z\psi$ is again an isomorphism. 

(4) {\em Let $Z$ be a Nash irreducible component of $\Sing_\ell(Y)$ such that $Z\cap\partial\Hh\neq\varnothing$, $Z\not\subset\partial\Hh$ and $\ell\geq0$. Then
$$
\psi(Z\cap U\cap\Hh)=\{(y,t)\in (Z\cap\partial\Hh)\times\R:\, 0\leq t<\sqrt{\varepsilon(y)}\}.
$$}
By Proposition \ref{compretraction} we have $\rho(Z\cap U)=Z\cap\partial\Hh$. Thus, 
$$
\psi(Z\cap U\cap\Hh)\subset\{(y,t)\in (Z\cap\partial\Hh)\times\R:\, 0\leq t<\sqrt{\varepsilon(y)}\}.
$$
To prove that the previous inclusion is in fact an equality, it is enough to proceed similarly to the end of the proof of Lemma \ref{style}.

\noindent{\sc Step 3.} Let $a:\partial\Hh\to\R$ be a strictly positive Nash function such that $a<\varepsilon$. Recall that $\Hh_\epsilon=D(\Hh)\cap\{\epsilon t\geq0\}$. For $\epsilon=\pm$, define
\begin{align*}
&\Hh^{\bullet}:=\Hh\setminus\varphi^{-1}\Big(\Big\{(y,s)\in\partial\Hh\times\R:\, |s|<\frac{a(y)}{4}\Big\}\Big)\supset\Hh\setminus U,\\
&\Hh^{\bullet}_{\epsilon}:=\Hh_{\epsilon}\setminus\psi^{-1}\Big(\Big\{(y,s)\in\partial\Hh\times\R:\, |s|<\frac{\sqrt{a(y)}}{2}\Big\}\Big)\supset \Hh_{\epsilon}\setminus V.
\end{align*}
We claim: {\em The restriction $\varpi_{a\epsilon}:=\pi|_{\Hh^{\bullet}_{\epsilon}}:\Hh^{\bullet}_{\epsilon}\to\Hh^{\bullet},\ (x,t)\mapsto x$ is a Nash diffeomorphism, whose inverse map is $\varpi_{a\epsilon}^{-1}:\Hh^{\bullet}\to\Hh^{\bullet}_{\epsilon},\ x\mapsto(x,\epsilon\sqrt{h(x)})$ for $\epsilon=\pm$.}

As $\varpi_{a\epsilon}$ is injective, let us check: {\em it is surjective}. 

Let $x\in\Hh^{\bullet}$. As $x\in\Int(\Hh)$, we have $h(x)>0$ and we write $t:=\epsilon\sqrt{h(x)}$. It holds that $(x,t)\in D(\Hh)$ and $\pi(x,t)=x$. We want to check that $(x,t)\in\Hh^{\bullet}_{\epsilon}$. If $x\not\in U$, then $(x,t)\in \Hh_{\epsilon}\setminus V\subset\Hh^{\bullet}_{\epsilon}$. If $x\in U$, then $\psi(x,t)=(\rho(x),t)$. As $x\in\Hh^{\bullet}$, it holds $
\frac{a(\rho(y))}{4}\leq h(x)$, so $\frac{\sqrt{a(\rho(y))}}{2}\leq\sqrt{h(x)}=\epsilon t$. Consequently, $(x,t)\in\Hh^{\bullet}_{\epsilon}$, so $\varpi_{a\epsilon}$ is surjective. 

In addition, by Lemma \ref{doubleii}(i) $d_z\varpi_{a\epsilon}=d_z\pi$ is an isomorphism for each $z\in\Hh^{\bullet}_{\epsilon}\subset D(\Hh)\cap\{\epsilon t>0\}$. Consequently $\varpi_{a\epsilon}$ is a Nash diffeomorphism, as claimed. 

\noindent{\sc Step 4.} We now construct an ${\mathcal S}^{2k-1}$ embedding $\phi_{a,k}:M\to D(\Hh)$ for $k\geq 1$ arbitrarily large. Substitute $M$ by $\Hh\cup U$ and define
$$
F_{a,k}:\partial\Hh\times\R\to\partial\Hh\times\R,\ (y,s)\mapsto\begin{cases}
(y,s)&\text{if $s<0$,}\\
(y,f_{a(y),k}(s))&\text{if $0\leq s\leq a(y)$,}\\
(y,\sqrt{s})&\text{if $a(y)<s$,}
\end{cases}
$$
where $f_{a(y),k}$ is the ${\mathcal S}^{2k}$ function introduced in Lemma \ref{fold1} and we choose $k$ large enough. Denote 
\begin{align*}
& U':=\varphi(U)=\{(y,s)\in\partial\Hh\times\R:\ |s|<\veps(y)\},\\
& V'=\psi(V)=\{(y,t)\in\partial\Hh\times\R:\, |t|<\sqrt{\varepsilon(y)}\}\text{\ (already introduced in {\sc Step 2})}.
\end{align*}
The open semialgebraic set $F_{a,k}(U')$ is contained in $V'$, because if $-\veps(y)<-s<0$, then $-\sqrt{\veps(y)}<-\sqrt{s}<-s<0$ (recall that $0<\veps(y)<1$) and if $0\leq s\leq a(y)\leq 1$, then $0\leq f_{a(y),k}(s)\leq\sqrt{s}$.

The map $F_{a,k}$ is an ${\mathcal S}^{2k-1}$ diffeomorphism, because $f_{a(y),k}:[0,a(y)]\to[0,\sqrt{a(y)}]$ is by Lemma \ref{fold1} an ${\mathcal S}^{2k-1}$ diffeomorphism such that the Taylor polynomial of degree $2k$ at $s=0$ is $\s$ and its Taylor polynomial of degree $2k-1$ at $s=a(y)$ coincides with the one of $\sqrt{\s}$.

Define
$$
\phi_{a,k}:M\to D(\Hh),\ x\mapsto
\begin{cases}
\varpi_{a+}^{-1}(x)&\text{if $x\in M\setminus U=\Hh\setminus U\subset\Hh^{\bullet}$,}\\
(\psi^{-1}\circ F_{a,k}\circ\varphi)(x)&\text{if $x\in U$.}
\end{cases}
$$
It holds that $\phi_{a,k}$ is an ${\mathcal S}^{2k-1}$ diffeomorphism onto its image $\phi_{a,k}(M)$, whose ${\mathcal S}^{2k-1}$ inverse is
$$
\phi_{a,k}^{-1}:\phi_{a,k}(M)\to M,\ x\mapsto
\begin{cases}
\varpi_{a+}(y)&\text{if $y\in\phi_{a,k}(M)\setminus V\subset\Hh^{\bullet}_+$,}\\
(\varphi^{-1}\circ F_{a,k}^{-1}\circ\psi)(y)&\text{if $y\in V$.}
\end{cases}
$$
In particular, $\phi_{a,k}(M)$ is an open semialgebraic subset of $D(\Hh)$, so it is a Nash manifold. In addition, $\phi_{a,k}$ satisfies:
\begin{itemize}
\item $\phi_{a,k}(\Hh)=\Hh_+$ and $\phi_{a,k}|_{\partial\Hh}=(\id_{\partial\Hh},0)$.
\item $\phi_{a,k}(Y_j)=(Y_j\times\R)\cap\Int(\Hh_+)$ for each $j=r+1,\ldots,s$, because $Y_j\subset\Int(\Hh)$ and $\phi_{a,k}(\Int(\Hh))=\Int(\Hh_+)$.
\item $\phi_{a,k}(M)=\phi_{a,k}(\Int(\Hh))\cup\phi_{a,k}(U)$ is an open semialgebraic subset of $D(\Hh)$, because $\phi_{a,k}(\Int(\Hh))=\Int(\Hh_+)$ is an open semialgebraic subset of $D(\Hh)$, $\varphi|_U:U\to U'$, $F_{a,k}:\partial\Hh\times\R\to\partial\Hh\times\R$ and $\psi|_V:V\to V'$ are ${\mathcal S}^{2k-1}$ diffeomorphisms and $F_{a,k}(U')\subset V'$ is an open semialgebraic set.
\item $\phi_{a,k}(Y_i)\subset D( Y_i\cap\Hh)=(Y_i\times\R)\cap D(\Hh)$. This inclusion follows from the equality \eqref{yi} and because the retraction $\rho$ is compatible with $Y$.
\end{itemize}
It only remains to show that if $a$ is small enough, then $\phi_{a,k}|_{\Hh}$ is close to $(\pi|_{\Hh_+})^{-1}$ and $\phi_{a,k}$ is close to $(\rho,0)$ on $M\setminus\Hh$ (with respect to the ${\mathcal S}^0$ topology).

Let us check: $\xi(y,s):=\psi\circ(\pi|_{\Hh_+})^{-1}\circ\varphi^{-1}(y,s)=(y,\sqrt{s})$. 

Let $x\in U\cap\Hh$ be such that $(y,s)=\varphi(x)=(\rho(x),h(x))$, so $(\pi|_{\Hh_+})^{-1}(x)=(x,\sqrt{h(x)})$ and $\psi(x,\sqrt{h(x)})=(\rho(x),\sqrt{h(x)})=(y,\sqrt{s})$, so $\xi(y,s)=(y,\sqrt{s})$.

Let $\eta:M\to\R$ be a strictly positive continuous semialgebraic function and let us choose $a$ to guarantee that $\|\phi_{a,k}|_{\Hh}-(\pi|_{\Hh_+})^{-1}\|<\eta|_{\Hh}$. In fact, as $\psi,\varphi$ are Nash diffeomorphisms, it is enough to check: \em If $b:\varphi(U)\to\R$ is a strictly positive continuous semialgebraic function, there exists a strictly positive continuous semialgebraic function $a:\partial\Hh\to\R$ such that $a<\veps$ and $\|F_{a,k}|_{\varphi(U\cap\Hh)}-\xi\|<b$\em.

Let $a:\partial\Hh\to\R$ be a continuous semialgebraic function such that $0<a<\tfrac{3}{4}\veps<1$. As $F_{a,k}$ coincides with $\xi$ outside $\{(y,s)\in\partial\Hh\times\R:\ s\in[0,a(y))\}$, we will work on the compact semialgebraic set $K:=\{(y,s)\in\partial\Hh\times\R:\ s\in[0,\frac{3}{4}\veps(y)]\}$. We claim: {\em The projection $\pi_1:K\to\partial\Hh$ is open, proper and surjective}. 

The semialgebraic map $\pi_1$ is surjective, because $\partial\Hh\times\{0\}\subset K$. To prove that $\pi_1$ is open, let $A$ be an open subset of $K$ and let $y_0\in\pi_1(A)$. Let $s_0\in [0,\frac{3}{4}\veps(y_0)]$ be such that $(y_0,s_0)\in A$. As $A$ is open, there exists $\zeta>0$ and $B\subset\partial\Hh$ open such that $y_0\in B$ and $(B\times[s_0-\zeta,s_0+\zeta])\cap K\subset A$. In particular, $(y_0,s)\in A$ and $\pi_1(y_0,s)=y_0$ for each $s\in I:=[\max\{s_0-\zeta,0\},\min\{s_0+\zeta,\frac{3}{4}\veps(y_0)\}]$. Let $s_1\in(0,\frac{3}{4}\veps(y_0))\cap(s_0-\zeta,s_0+\zeta)$ be the midpoint of $I$. Let $\zeta'$ be such that $(s_1-\zeta',s_1+\zeta')\subset(0,\frac{3}{4}\veps(y_0))\cap(s_0-\zeta,s_0+\zeta)$. As $\veps$ is continuous and $s_1+\zeta'<\frac{3}{4}\veps(y_0)$, we may assume (after shrinking $B$ if necessary) that $\frac{3}{4}\veps(B)\subset[s_1+\zeta',+\infty)$. Thus, $B\times (s_1-\zeta',s_1+\zeta')\subset(B\times[s_0-\zeta,s_0+\zeta])\cap K\subset A$, because if $(y,s)\in B\times (s_1-\zeta',s_1+\zeta')$, then $0<s_1-\zeta'<s<s_1+\zeta'\leq\frac{3}{4}(y)$. As $\pi_1(B\times (s_1-\zeta',s_1+\zeta'))=B$ is an open neighborhood of $y_0$ in $\partial\Hh$, we deduce $\pi_1(A)$ is an open subset of $\partial\Hh$, so $\pi_1$ is an open map. 

We show now that $\pi_1$ is proper. Let $C\subset\partial\Hh$ be a compact set and let $\lambda:=\max\{\frac{3}{4}\veps|_C\}$. Then $\pi_1^{-1}(C)\subset C\times [0,\lambda]$, which is a compact set, so $\pi^{-1}(C)$ is also compact. Thus, $\pi_1$ is proper. 

As $\pi_1$ is open, closed and surjective, by \cite[Const.3.1]{fg4} the function 
$$
a:\partial\Hh\to\R,\ y\mapsto\frac{1}{2}\min\Big\{b(y,s)^2:\ s\in\Big[0,\frac{3}{4}\veps(y)\Big]\Big\}
$$ 
is strictly positive, continuous and semialgebraic. Fix $y\in\partial\Hh$. As the last component of $F_{a,k}$ is strictly increasing on $[0,a(y)]$ and $f_a(a(y))=\sqrt{a(y)}$, we have
\begin{align*}
\|F_{a,k}(y,s)&-(y,\sqrt{s})\|\\
&=\begin{cases}
0<b(y,s)&\text{if $a(y)\leq s\leq\frac{3}{4}\veps(y)$,}\\
|f_a(s)-\sqrt{s}|=\sqrt{s}-f_a(s)<\sqrt{s}<\sqrt{a(y)}\leq b(y,s)&\text{if $0\leq s<a(y)$.}
\end{cases}
\end{align*}

Let us check: {\em $\phi_{a,k}$ is close to $(\rho,0)$ on $M\setminus\Hh$}. If we denote $\tau:=\psi\circ(\rho,0)\circ\varphi^{-1}$, it is enough to prove for each $y\in\partial\Hh$: {\em $\|F_{a,k}(y,s)-\tau(y,s)\|<b(y,s)$ if $-a(y)<s<0$.}

Observe that $\tau(y,s)=(y,0)$. Indeed, let $x\in U$ be such that $(y,s)=\varphi(x)=(\rho(x),h(x))$. Thus, $\psi(\rho(x),0)=(\rho(\rho(x)),0)=(\rho(x),0)=(y,0)$, so $\tau(y,s)=(y,0)$. In addition, if $-a(y)<s<0$, we have (as $a<\frac{3}{4}\veps(y)<1$)
$$
\|F_{a,k}(y,s)-\tau(y,s)\|=\|(y,s)-(y,0)\|=-s<a(y)<\sqrt{a(y)}\leq b(y,s).
$$

\noindent{\sc Step 5.} Let us see now how to obtain the desired Nash embedding. Recall that $Y_1,\ldots,Y_r$ are the irreducible components of $Y$ that meet $\partial\Hh$ but are not contained in $\partial\Hh$. Consider the Nash normal-crossings divisor $\overline{Y}:=\bigcup_{i=1}^r Y_i\subset M$. Recall that $U\cap\bigcup_{j=r+1}^s Y_j=\varnothing$, so $\phi_{a,k}|_{\bigcup_{j=r+1}^s Y_j}=(\pi|_{\Hh_+})^{-1}|_{\bigcup_{j=r+1}^s Y_j}$. Fix $k\geq 1$ such that $2k-1\geq d\cdot\max\{1,(\binom{d}{[d/2]}-1)\}+1$ (for instance $k=d\cdot\max\{1,(\binom{d}{[d/2]}-1)\}+1$), where $d$ is the dimension of both $\Hh$ and $M$. By \cite[Thm.1.6, Thm 1.7]{bfr} we can approximate the restriction $\phi_{a,k}|_Y:Y\to \phi_{a,k}(M)\subset D(\Hh)$ by a Nash map $\widehat{\phi}:Y\to\phi_{a,k}(M)\subset D(\Hh)$ such that $\widehat{\phi}|_{\bigcup_{j=r+1}^sY_j}=(\pi|_{\Hh_+})^{-1}|_{\bigcup_{j=r+1}^s Y_j}$, $\widehat{\phi}|_{\partial\Hh}=(\id_{\partial\Hh},0)$ and $\widehat{\phi}|_{\overline{Y}}:\overline{Y}\to \phi_{a,k}(M)\subset D(\Hh)$ satisfies $\widehat{\phi}(Y_i)\subset D(Y_i\cap\Hh)$ for $i=1,\ldots,r$. By \cite[Prop.8.2]{bfr} we can extend $\widehat{\phi}$ to a global Nash map $\phi:M\to\phi_{a,k}(M)\subset D(\Hh)$, that is close to $\phi_{a,k}$ (which is an ${\mathcal S}^{2k-1}$ diffeomorphism onto its image $\phi_{a,k}(M)$) in the ${\mathcal S}^1$ topology. Thus, by \cite[Lem.II.1.7]{sh} the map $\phi:M\to\phi_{a,k}(M)$ is a Nash diffeomorphism, so in particular $\phi(M)=\phi_{a,k}(M)$.

\noindent{\sc Step 6.} We show: {\em There exists an open semialgebraic neighborhood $W\subset U$ of $\partial\Hh$ such that $\cl(W)\subset M$ and $\phi|_{\cl(W)\cup\Hh}:\cl(W)\cup\Hh\to D(\Hh)$ is proper.\em} Once this is done we conclude that, up to take a smaller semialgebraic function $a>0$ if necessary, the Nash embedding $\phi:M\to D(\Hh)$ defined in {\sc Step 5} satisfies all the other required properties.

As $\Hh$ is a closed (semialgebraic) subset of $\R^n$, the Nash double $D(\Hh)$ is a closed (semialgebraic) subset of $\R^{n+1}$. As $M$ is locally compact, there exists by Lemma \ref{co} an open semialgebraic neighborhood $W_0\subset M$ of $\partial\Hh$ such that $\cl(W_0)\subset M$. As $\phi(M)$ is open in $D(\Hh)$, the set $C:=D(\Hh)\setminus\phi(M)$ is a closed (semialgebraic) subset of $D(\Hh)$ that does not intersect $\partial\Hh$. Recall that we have assumed $M=U\cup\Hh$ at the beginning of {\sc Step 4}. Let $V\subset\phi(U\cap W_0)$ be an open semialgebraic neighborhood of $\partial\Hh$ such that $\cl(V)\cap C=\varnothing$, so $\cl(V)\subset\phi(M)$. The set $W:=\phi^{-1}(V)$ is an open semialgebraic neighborhood of $\partial\Hh$ in $U$ such that $\phi(\cl(W))=\cl(V)$, because $\phi:M\to\phi(M)$ is a semialgebraic homeomorphism, $\cl(W)\subset\cl(W_0)\subset M$ and $\cl(V)\subset\phi(M)$. The restriction $\phi|_{\cl(W)\cup\Hh}:\cl(W)\cup\Hh\to D(\Hh)$ satisfies $\phi(\cl(W)\cup\Hh)=\cl(V)\cup \Hh_+$ because $\phi(\cl(W))=\cl(V)$ and $\phi(\Hh)=\Hh_+$. As $\cl(V)\cup\Hh_+$ is closed in $D(\Hh)$, if $K\subset D(\Hh)$ is compact, then the set $K\cap (\cl(V)\cup \Hh_+)$ is compact. The map $\phi|_{\cl(W)\cup\Hh}:\cl(W)\cup\Hh\to\cl(V)\cup \Hh_+\subset D(\Hh)$ is a semialgebraic homeomorphism, so $\phi^{-1}(K\cap (\cl(V)\cup \Hh_+))$ is a compact set. We conclude that the restriction $\phi|_{\cl(W)\cup\Hh}:\cl(W)\cup\Hh\to D(\Hh)$ is proper, as required.
\end{proof}

As a straightforward consequence of the previous embedding theorem, we have the following result. We keep all the notations introduced in the statement of Theorem \ref{embedding}.

\begin{cor}\label{piego}
The composition $f:=\pi\circ\phi:M\to\Hh$ is a Nash map with the following properties:
\begin{itemize}
\item[\rm{(i)}] $f(Y_i)\subset Y_i$ for $i=1,\ldots,r$ and $f|_{Y_i}=\id_{Y_i}$ for each $i=r+1,\ldots,s$.
\item[\rm{(ii)}] $f|_{\Hh}$ is a semialgebraic homeomorphism close to $\id_{\Hh}$. Moreover $f|_{\Int(\Hh)}$ is a Nash diffeomorphism and $f|_{\partial\Hh}=\id_{\partial\Hh}$.
\item[\rm{(iii)}] $f|_{M\setminus\Hh}:M\setminus\Hh\to\Hh$ is a Nash embedding close to $\rho|_{M\setminus\Hh}$. 
\item[\rm{(iv)}] At each point $x\in\partial\Hh$ the map $f$ has a local presentation of the type 
$$
(y_1,\ldots,y_d)\mapsto(y_1^2,y_2,\ldots,y_d).
$$
\item[\rm{(v)}] There exists an open semialgebraic neighborhood $V\subset M$ of $\partial\Hh$, such that $f|_{\cl(V)\cup\Hh}$ is a proper Nash map. 
\end{itemize}
\end{cor}
\begin{proof}
Statements (i), (ii) and (iii) follow from Theorem \ref{embedding}, whereas statement (iv) is a consequence of Lemma \ref{doubleii}(iii) and the fact that $\phi:M\to\phi(M)\subset D(\Hh)$ is a Nash diffeomorphism onto its image. To prove (v), we use that $\pi$ is proper (see Lemma \ref{doubleii}(iv)) and that $\phi$ is proper on $\cl(V)\cup\Hh$ where $V\subset M$ is a small semialgebraic neighborhood of $\partial\Hh$ in $M$. 
\end{proof}

\subsection{Folding all the components of the boundary.} 
We are almost ready to prove Theorem \ref{fold}. Before, we need a technical lemma to confront local models:

\begin{lem}\label{change}
Consider the Nash map $\varphi:\R^d:\to\R^d,\ (x_1,\ldots,x_d)\mapsto(x_1^2,x_2,\ldots,x_d)$ and let $\psi_1,\psi_2:\R^d\to\R^d$ be Nash diffeomorphisms such that $\psi_i(\{\x_j=0\})=\{\x_j=0\}$ for $j=1,\ldots,s$ and $i=1,2$. Then there exist Nash functions $h_j:\R^d\to\R$ each one strictly positive around $\{\x_j=0\}$ and a Nash map $g:\R^d\to\R^{d-s}$ such that $g|_{\{\x_1=0,\ldots,\x_s=0\}}:\{\x_1=0,\ldots,\x_s=0\}\to\R^{d-s}$ is a Nash diffeomorphism and $\psi_2\circ\varphi\circ\psi_1=((\x_1h_1)^2,\x_2h_2,\ldots,\x_sh_s,g)$.
\end{lem}
\begin{proof}
As $\psi_i(\{\x_j=0\})=\{\x_j=0\}$ for $j=1,\ldots,s$ and $\psi_i$ is a Nash diffeomorphism, then $\psi_i|_{\{\x_1=0,\ldots,\x_s=0\}}:\{\x_1=0,\ldots,\x_s=0\}\to\{\x_1=0,\ldots,\x_s=0\}$ is a Nash diffeomorphism. Thus, 
$$
\psi_2\circ\varphi\circ\psi_1|_{\{\x_1=0,\ldots,\x_s=0\}}:\{\x_1=0,\ldots,\x_s=0\}\to\{\x_1=0,\ldots,\x_s=0\}
$$ 
is a Nash diffeomorphism. Let $\pi:\R^d\to\R^{d-s},\ (x_1,\ldots,x_d)\mapsto(x_{s+1},\ldots,x_d)$ and let $g:=\pi\circ\psi_2\circ\varphi\circ\psi_1$. We have 
$$
g|_{\{\x_1=0,\ldots,\x_s=0\}}:\{\x_1=0,\ldots,\x_s=0\}\to\R^{d-s}
$$ 
is a Nash diffeomorphism and $\psi_2\circ\varphi\circ\psi_1=(g_1,\ldots,g_s,g)$ for some Nash functions $g_i:\R^d\to\R$. As $\psi_i(\{\x_j=0\})=\{\x_j=0\}$, its $j$th component $\psi_{ij}$ is divisible by $\x_j$ and thus, there exists a Nash function $h_{ij}:\R^d\to\R$ that is strictly positive around $\{\x_j=0\}$ such that $\psi_{ij}=\x_jh_{ij}$ for $i=1,2$ and $j=1,\ldots,s$. Thus, 
\begin{multline*}
\psi_2\circ\varphi\circ\psi_1=\\
(\psi_{11}^2h_{21}(\psi_{11}^2,\psi_{12},\ldots,\psi_{1d}),\psi_{12}h_{22}(\psi_{11}^2,\psi_{12},\ldots,\psi_{1d}),\cdots,\psi_{1s}h_{2s}(\psi_{11}^2,\psi_{12},\ldots,\psi_{1d}),g)\\
=((\x_1h_{11})^2h_{21}(\psi_{11}^2,\psi_{12},\ldots,\psi_{1d}),\x_2h_{12}h_{22}(\psi_{11}^2,\psi_{12},\ldots,\psi_{1d}),\ldots\\
\ldots,\x_sh_{1s}h_{2s}(\psi_{11}^2,\psi_{12},\ldots,\psi_{1d}),g).\end{multline*}
It is enough to take $h_1:=h_{11}\sqrt{h_{21}(\psi_{11}^2,\psi_{12},\ldots,\psi_{1d})}$ and $h_j:=h_{1j}h_{2j}(\psi_{11}^2,\psi_{12},\ldots,\psi_{1d})$ for $j=2,\ldots,s$.
\end{proof}

We next prove Theorem \ref{fold}.

\begin{proof}[Proof of Theorem \em \ref{fold}]
The proof os conducted in several steps:

\noindent{\sc Step 1.} {\em Initial preparation and construction of the folding Nash maps.}
Let $M$ be a Nash envelope of $\Qq$ such that the Nash closure $Y$ of $\partial \Qq$ is a Nash normal-crossings divisor of $M$ and $\Qq\cap Y=\partial\Qq$ (Section \ref{envelope}). Let $Y_1,\ldots,Y_\ell$ be the irreducible components of $Y$. By Lemma \ref{equation} there exist, after shrinking $M$ if necessary, Nash equations $h_i:M\to\R$ of $Y_i$ such that $d_xh_i:T_xM\to\R$ is surjective for each $x\in Y_i$ and $\Hh_i:=h_i^{-1}([0,+\infty))$ is a Nash manifold with boundary whose boundary is $Y_i$ for each $i=1,\ldots,\ell$. By Mostowski's trick (Proposition \ref{Mos}) we may assume $M$ is a closed (semialgebraic) set of $\R^n$, so $\Hh_i$ is a closed (semialgebraic) set of $\R^n$ for $i=1,\ldots,\ell$. By Corollary \ref{piego} there exist open semialgebraic neighborhoods $V_i\subset U_i\subset M$ of $\partial\Hh_i$ such that $\cl(V_i)\subset U_i$ and proper Nash maps $f_i:\cl(V_i)\cup\Hh_i\to\Hh_i$ and Nash retractions $\rho_i:U_i\to Y_i$ for $i=1,\ldots,\ell$ such that: 
\begin{itemize}
\item[(i)] $f_i(Y_j)\subset Y_j$ for $j=1,\ldots,\ell$ and $f_i|_{Y_i}=\id_{Y_i}$, 
\item[(ii)] $f_i|_{\Hh_i}:\Hh_i\to\Hh_i$ is a semialgebraic homeomorphism close to the identity map, whose restriction $f_i|_{\Int(\Hh_i)}:\Int(\Hh_i)\to\Int(\Hh_i)$ is a Nash diffeomorphism and $f_i|_{\partial\Hh_i}=\id_{\partial\Hh_i}$, 
\item[(iii)] $f_i|_{V_i\setminus\Hh_i}:V_i\setminus\Hh_i\to\Int(\Hh_i)$ is a semialgebraic embedding close to $\rho_i|_{V_i\setminus\Hh_i}:V_i\setminus\Hh_i\to\partial\Hh_i$,
\item[(iv)] $f_i$ has local representations $(y_1,\ldots,y_d)\mapsto(y_1^2,y_2,\ldots,y_d)$ at each point $x\in Y_i$. 
\end{itemize}

\noindent{\sc Step 2.} {\em Properties of the folding Nash maps with respect to $\Qq$.} 
Let us check: {\em $f_i|_{\Qq}:\Qq\to\Qq$ is a semialgebraic homeomorphism close to the identity map, whose restriction $f_i|_{\Int(\Qq)}:\Int(\Qq)\to\Int(\Qq)$ is a Nash diffeomorphism for $i=1,\ldots,\ell$}. We show first: \em $f_i(\Qq)=\Qq$ and $f_i(\Int(\Qq))=\Int(\Qq)$ for $i=1,\ldots,\ell$\em. Once this is done, as $\Qq=\Hh_1\cap\cdots\cap\Hh_\ell$, $\Int(\Qq)=\Int(\Hh_1)\cap\cdots\cap\Int(\Hh_\ell)$, $f_i|_{\Hh_i}:\Hh_i\to\Hh_i$ is a semialgebraic homomorphism close to the identity map and $f_i|_{\Int(\Hh_i)}:\Int(\Hh_i)\to\Int(\Hh_i)$ is a Nash diffeomorphism for each $i=1,\ldots,\ell$, we deduce that $f_i|_\Qq:\Qq\to\Qq$ is a semialgebraic homeomorphism close to the identity map and its restriction $f_i|_{\Int(\Qq)}:\Int(\Qq)\to\Int(\Qq)$ is a Nash diffeomorphism for $i=1,\ldots,\ell$.

Fix $i=1,\ldots,\ell$. As $f_i:\cl(V_i)\cup\Hh_i\to\Hh_i$ is proper and $\Qq=\cl(\Int(\Qq))\subset\Hh_i$, we have $f_i(\Qq)=f_i(\cl(\Int(\Qq)))=\cl(f_i(\Int(\Qq)))$, so it is enough to prove: $f_i(\Int(\Qq))=\Int(\Qq)$. 

As the map $f_i|_{\Int(\Hh_i)}:\Int(\Hh_i)\to\Int(\Hh_i)$ is a Nash diffeomorphism and $\Int(\Qq)\subset\Int(\Hh_i)$, we have $f_i(\Int(\Qq))\subset\Int(\Hh_i)$. As $f_i(Y_j)\subset Y_j$ for $j=1,\ldots,\ell$, we deduce $f_i(Y_j\cap\Int(\Hh_i))$ is a closed Nash submanifold of the Nash manifold $Y_j\cap\Int(\Hh_i)$ and both have the same dimension. Consequently, the image of each connected component of $f_i(Y_j\cap\Int(\Hh_i))$ is a connected component of $Y_j\cap\Int(\Hh_i)$. As $f_i$ is close to the identity map, $f_i(D)=D$ for each connected component of $Y_j\cap\Int(\Hh_i)$. Thus, $f_i(Y_j\cap\Int(\Hh_i))=Y_j\cap\Int(\Hh_i)$ and $f_i(Y\cap\Int(\Hh_i))=Y\cap\Int(\Hh_i)$. Observe that $\Int(\Qq)$ is a union of connected components of $\Int(\Hh_i)\setminus Y$ because $\Qq\cap Y=\partial\Qq$. As $f_i|_{\Int(\Hh_i)}:\Int(\Hh_i)\to\Int(\Hh_i)$ is a Nash diffeomorphism and $f_i(Y\cap\Int(\Hh_i))=Y\cap\Int(\Hh_i)$, we deduce that $f_i(\Int(\Qq))$ is also a union of connected components of $\Int(\Hh_i)\setminus Y$. As $f_i|_{\Int(\Qq)}$ is close to the identity map, $f_i(C)=C$ for each connected component of $\Int(\Qq)$. In particular, $f_i(\Int(\Qq))=\Int(\Qq)$, as claimed.

\noindent{\sc Step 3.} {\em Construction of suitable neighborhoods of $\Qq$.}
As $V_i\cup\Hh_i=V_i\cup\Int(\Hh_i)$ for $i=1,\ldots,\ell$, the intersection $W_0:=\bigcap_{i=1}^\ell(V_i\cup\Hh_i)$ is an open semialgebraic neighborhood of $\Qq$ in $M$. For each $i=1,\ldots,\ell$ there exists by Lemma \ref{co} an open semialgebraic neighborhood $W_i$ of $\Qq$ in $W_{i-1}$ such that $\cl(W_i)\subset\Omega_i:=\bigcap_{j=1}^\ell f_j^{-1}(W_{i-1})\cap W_{i-1}$ and observe that $f_j(\cl(W_i))\subset W_{i-1}$ for $1\leq i,j\leq\ell$. Let $\{C_{i\alpha}\}_\alpha$ be the family of the connected components of $W_i\setminus Y$ such that $\cl(C_{i\alpha})\cap\Qq=\varnothing$. Then $W_i\setminus\bigcup_{\alpha}\cl(C_{i\alpha})$ is an open semialgebraic neighborhood of $\Qq$ in $M$. Let us substitute $W_i$ by $W_i':=W_i\setminus\bigcup_{\alpha}\cl(C_{i\alpha})$. Now, the connected components $C$ of $W_i\setminus Y$ satisfy $\cl(C)\cap\Qq\neq\varnothing$, so $\cl(C)\cap\Hh_i\neq\varnothing$ for each $i=1,\ldots,\ell$ (because $\Qq=\Hh_1\cap\cdots\cap\Hh_\ell$ by Lemma \ref{inthi}). As we can do this reduction when construct $W_i$ and before we construct $W_{i+1}$, we keep the property that $f_j(\cl(W_i))\subset W_{i-1}$ for $1\leq i,j\leq\ell$.

\noindent{\sc Step 4.} {\em Crossed properties of the folding Nash maps.}
We claim: {\em $f_i(\Hh_j\cap W_k)\subset(\Hh_i\cap\Hh_j)\cap W_{k-1}$ for $1\leq i,j,k\leq\ell$.}

For simplicity we show: $f_2(\Hh_1\cap W_k)\subset\Hh_1\cap W_{k-1}$ for $k=1,\ldots,\ell$. As $f_2$ is continuous and $f_2(W_k)\subset W_{k-1}$, it is enough to prove: {\em $f_2(\Int(\Hh_1\cap\Hh_2))\subset\Int(\Hh_1\cap\Hh_2)$ and $f_2((\Int(\Hh_1)\cap W_k)\setminus\Hh_2)\subset\Int(\Hh_1)$ for $k=1,\ldots,\ell$}. 

To prove the inclusion $f_2(\Int(\Hh_1\cap\Hh_2))\subset\Int(\Hh_1\cap\Hh_2)$, it is enough to consider the Nash manifold with corners $\Hh_1\cap\Hh_2$ and to proceed analogously as above with $f_i$ and $\Qq$. 

We show next: $f_2((\Int(\Hh_1)\cap W_k)\setminus\Hh_2)\subset\Int(\Hh_1)$. To that end, we prove: {\em $f_2(C)\subset\Int(\Hh_1)$ for each connected component $C$ of $(\Int(\Hh_1)\cap W_k)\setminus\Hh_2$.}

As $f_2|_{W_k\setminus\Hh_2}:W_k\setminus\Hh_2\to\Int(\Hh_2)$ is a Nash embedding and $f_2(Y_1)\subset Y_1$, we deduce that each connected component $C$ of $(\Int(\Hh_1)\cap W_k)\setminus\Hh_2$ is mapped under $f_2$ into a connected component of $\Int(\Hh_2)\setminus Y_1$. Let $C$ be a connected component of $(\Int(\Hh_1)\cap W_k)\setminus\Hh_2\subset W_k\setminus(Y_1\cup Y_2)$. As 
$$
(\Int(\Hh_1)\cap W_k)\setminus\Hh_2=(\Hh_1\setminus\Int(\Hh_2))\cap(W_k\setminus Y_1)\cap(W_k\setminus Y_2)=(\Hh_1\setminus\Int(\Hh_2))\cap(W_k\setminus(Y_1\cup Y_2))
$$ 
is open and closed in $W_k\setminus (Y_1\cup Y_2)$, we deduce that $C$ is a connected component of $W_k\setminus (Y_1\cup Y_2)$. 

\noindent{\sc Case 1.} $\cl(C)\cap Y_2\neq\varnothing$. As $C$ is a connected component of $W_k\setminus (Y_1\cup Y_2)$ and $Y_1\cup Y_2$ is a Nash normal-crossings divisor, there exists a point $x\in\cl(C)\cap (Y_2\setminus Y_1)$. Let $U^x$ be a semialgebraic neighborhood of $x$ in $W_k$ that does not meet $Y_1$. As $x\in\cl(C)$ and $C\subset\Int(\Hh_1)$, we deduce $U^x\subset\Int(\Hh_1)$. As $\rho_2(x)=x\in\Int(\Hh_1)$, and $f_2|_{W_k\setminus\Hh_2}:W_k\setminus\Hh_2\to\Int(\Hh_2)$ is close to $\rho_2|_{W_k\setminus\Hh_2}:W_k\setminus\Hh_2\to\partial\Hh_2$, we may assume $f_2(U^x)\subset\Int(\Hh_1)$. Thus, $f_2(C)$ meets $\Int(\Hh_1)$. As $f_2(C)$ is a connected subset of $\Int(\Hh_2)\setminus Y_1$ and $\Int(\Hh_1)\cap\Int(\Hh_2)=\Hh_1\cap(\Int(\Hh_2)\setminus Y_1)$ is open and closed in $\Int(\Hh_2)\setminus Y_1$ and $f_2(C)$ meets one of the connected components of $\Int(\Hh_1)\cap\Int(\Hh_2)$, we deduce $f_2(C)$ is contained in one of the connected components of $\Int(\Hh_1)\cap\Int(\Hh_2)$. Consequently, $f_2(C)\subset\Int(\Hh_1)\cap\Int(\Hh_2)\subset\Int(\Hh_1)$.

\noindent{\sc Case 2.} $\cl(C)\cap Y_2=\varnothing$. As $C\subset \Hh_1\setminus\Hh_2$, we have
\begin{multline*}
\cl(C)\cap\Hh_2=\cl(C)\cap (Y_2\cup\Int(\Hh_2))=\cl(C)\cap\Int(\Hh_2)\\
=\cl(C\cap\Int(\Hh_2))\cap\Int(\Hh_2)\subset\cl(C\cap\Hh_2)=\varnothing.
\end{multline*}
Observe that $C$ is the closure in $W_k\setminus (Y_1\cup Y_2)$ of the union of some connected components of $W_k\setminus Y$, which is a contradiction, because the closure of each connected components of $W_k\setminus Y$ meets $\Qq\subset\Hh_2$. 

Thus, $\cl(C)\cap Y_2\neq\varnothing$ for each connected component of $(\Int(\Hh_1)\cap W_k)\setminus\Hh_2$. Consequently, $f_2(C)\subset\Int(\Hh_1)$ for each connected component $C$ of $(\Int(\Hh_1)\cap W_k)\setminus\Hh_2$, so $f_2((\Int(\Hh_1)\cap W_k)\setminus\Hh_2)\subset\Int(\Hh_1)$, as claimed.

\noindent{\sc Step 5.} {\em Construction and requested properties of the joint folding Nash map.}
For $i=1,\ldots,\ell$ define $f^\bullet_i:=f_i|_{\cl(W_{\ell-i+1})}:\cl(W_{\ell-i+1})\to\cl(W_{\ell-i})$, which is a proper map, because $f_i:\cl(V_i)\cup\Qq\to\Hh_i$ is a proper map such that $f(\cl(W_{\ell-i+1}))\subset\cl(W_{\ell-i})$ and $\Hh_i$ is a closed semialgebraic subset of $\R^n$. Consider the composition $f:=f_\ell^\bullet\circ\cdots\circ f_1^\bullet:\cl(W_\ell)\to\cl(W_0)$, which is a proper map, because it is a composition of proper maps. Denote $W:=W_\ell$ and let us check that $f:\cl(W)\to\cl(W_0)$ satisfies the following properties:
\begin{itemize}
\item[(1)] $f(Y_i\cap\cl(W_\ell))\subset Y_i$ for each $i=1,\ldots,\ell$ (because the corresponding property is satisified by each $f_i$),
\item[(2)] $f|_{\Qq}=f_\ell|_\Qq\circ\cdots\circ f_1|_\Qq:\Qq\to\Qq$ is a semialgebraic homeomorphism close to the identity map, whose restriction $f|_{\Int(\Qq)}:\Int(\Qq)\to\Int(\Qq)$ is a Nash diffeomorphism (because the corresponding property is satisified by each $f_i$),
\item[(3)] $f(\cl(W))=\Qq$ (so $f:\cl(W)\to\Qq$ is a proper Nash map),
\item[(4)] $f$ has local representations $(y_1,\ldots,y_d)\mapsto(y_1^2,\ldots,y_s^2,y_{s+1},\ldots,y_d)$ at each point $x\in\partial\Qq$. The integer $s\geq1$ depends on the point $x$ and corresponds to the number of irreducible components of $Y$ that pass through $x$.
\end{itemize}

We first prove (3). Using {\sc Step 4} recursively we have the following:
\begin{equation*}
\begin{split}
f(W)&=(f_\ell^\bullet\circ\cdots\circ f_1^\bullet)(W_\ell)=(f_\ell\circ\cdots\circ f_2)(W_{\ell-1}\cap\Hh_1)\\
&\subset(f_\ell\circ\cdots\circ f_3)(W_{\ell-2}\cap\Hh_1\cap\Hh_2)\subset(f_\ell\circ\cdots\circ f_4)(W_{\ell-3}\cap\Hh_1\cap\Hh_2\cap\Hh_3)\\
&\subset\cdots\subset W_0\cap\Hh_1\cap\cdots\cap\Hh_\ell=\Qq.
\end{split}
\end{equation*}
As $f$ is proper and $f(\Qq)=\Qq$, we conclude $f(\cl(W))=\cl(f(W))=\cl(\Qq)=\Qq$.

We next show (4). Pick a point $x\in\partial\Qq$ and assume that $x$ belongs exactly to $Y_1,\ldots,Y_s$. Recall that the analytic closure of $\partial\Qq_x$ is $Y_{1,x}\cup\cdots\cup Y_{s,x}$ and there exists an open semialgebraic neighborhood $U\subset M$ of $x$ equipped with a Nash diffeomorphism $u:U\to\R^d$ such that $u(x)=0$, $u(Y_j\cap U)=\{\x_j=0\}$ and $u(\Qq\cap U)=\{\x_1\geq0,\ldots,\x_s\geq0\}$. 

Observe that $\eta:=u\circ f\circ u^{-1}=\eta_\ell\circ\cdots\circ\eta_1$ where $\eta_k:=u\circ f_k^\bullet\circ u^{-1}:\R^d\to\R^d$ is either a Nash diffeomorphism or a Nash map that has a local representation 
$$
(y_1,\ldots,y_d)\mapsto(y_1,\ldots,y_{k-1},y_k^2,y_{k+1},\ldots,y_d)
$$ 
at each point $y\in\{\y_k=0\}$ for each $k=1,\ldots,\ell$. In addition, these local representations preserve the local representations of the Nash hypersurfaces $Y_1,\ldots,Y_s$, so they correspond to coordinates' hyperplanes in these (local) coordinates. By Lemma \ref{change} we may assume that $\eta=\sigma_s\circ\cdots\circ\sigma_1$ is a composition of Nash maps of the type $\sigma_m:\R^d\to\R^d$,
$$
(x_1,\ldots,x_d)\mapsto(x_1h_{1m},\ldots,x_{m-1}h_{m-1,m},x_m^2h_{mm}^2,x_{m+1}h_{m+1,m},\ldots,x_sh_{sm},g_m)
$$
where $h_{jm}:\R^d\to\R$ is a Nash function that does not vanish around $\{\x_j=0\}$ and $g_m:\R^d\to\R^{d-s}$ satisfies $g_m|_{\{\x_1=0,\ldots,\x_s=0\}}:\{\x_1=0,\ldots,\x_s=0\}\to\R^{d-s}$ is a Nash diffeomorphism. Thus, there exist Nash functions $w_j:\R^d\to\R$ that do not vanish around $\{\x_j=0\}$ and a Nash map $g:\R^d\to\R^{d-s}$ such that $g|_{\{\x_1=0,\ldots,\x_s=0\}}:\{\x_1=0,\ldots,\x_s=0\}\to\R^{d-s}$ is a Nash diffeomorphism satisfying
$$
\eta:\R^d\to\R^d,\ z:=(z_1,\ldots,z_d)\mapsto(z_1^2w_1^2(z),\ldots,z_s^2w_s^2(z),g(z)).
$$
The Nash map 
$$
\psi:\R^d\to\R^d,\ z:=(z_1,\ldots,z_d)\mapsto(z_1w_1(z),\ldots,z_sw_s(z),g(z))
$$
is a Nash diffeomorphism around the origin because the determinant of its Jacobian matrix at the origin
$$
J_{\psi}(0)=\left[
\begin{array}{ccc|c}
w_1(0)&\cdots&0&0\\ 
\vdots&\ddots&\vdots&\vdots\\
0&\cdots &w_s(0)&0\\\hline
*&\ldots &*&J_g(0)
\end{array}\right]
$$
is $w_1(0)\cdots w_s(0)J_g(0)\neq0$. After shrinking the open semialgebraic neighborhood $U\subset M$ of $x$, there exists a Nash diffeomorphism $\phi:\R^d\to\R^d$ such that $\psi\circ\phi=\id_{\R^d}$. As $\eta=\theta\circ\psi$, where
$$
\theta:\R^d\to\R^d,\ (z_1,\ldots,z_d)\mapsto(z_1^2,\cdots,z_s^2,z_{s+1},\ldots,z_d),
$$
we deduce that $\eta\circ\phi=\theta$, so $f$ has local representation $(y_1,\ldots,y_d)\mapsto(y_1^2,\ldots,y_s^2,y_{s+1},\ldots,y_d)$ around $x\in\partial\Qq$. Recall that $s\geq1$ corresponds to the number of irreducible components of $Y$ that passes through $x$, as required.
\end{proof}

\section{Nash double of a Nash manifold with corners.}\label{s4}

Let $\Qq\subset\R^n$ be a $d$-dimensional Nash manifold with corners and $M$ a Nash envelope of $\Qq$ such that the Nash closure $Y$ of $\partial\Qq$ is a Nash normal-crossings divisor of $M$ and $\Qq\cap Y=\partial\Qq$. Let $Y_1,\ldots,Y_\ell$ be the irreducible components of $Y$ and $h_i:M\to\R$ Nash equations of $Y_i$ such that $d_xh_i:T_xM\to\R$ is surjective for $i=1,\ldots,\ell$. By Lemma \ref{equation} the sets $\Hh_i:=h_i^{-1}([0,+\infty))$ are Nash manifolds with boundary $Y_i$ that contain $\Qq$ as a closed subset. Write $\Qq=\Hh_1\cap\ldots\cap\Hh_\ell=\{h_1\geq0,\ldots,h_\ell\geq0\}\subset M$ (Lemma \ref{inthi}), $t:=(t_1,\ldots,t_\ell)$ and define the {\em Nash double of $\Qq$} as:
$$
D(\Qq):=\{(x,t)\in M\times\R^\ell:\ t_1^2-h_1(x)=0,\ldots,t_\ell^2-h_\ell(x)=0\}.
$$
Observe that $D(\Qq)$ only depends on the Nash equations $h_1,\ldots,h_\ell$, which are unique up to multiplication by strictly positive Nash functions. Thus, $D(\Qq)$ is unique up to Nash diffeomorphism. In addition, the reader can check that $D(\Qq)$ is the inverse image of the value $0$ under the Nash map $f:M\times\R^\ell\to\R^\ell, (x,t)\mapsto(t_1^2-h_1(x),\ldots,t_\ell^2-h_\ell(x))$ and that $0$ is a regular value for $f$, so $D(\Qq)$ is a $d$-dimensional Nash manifold. We keep this notations along the whole section. We first analyze the local structure of $D(\Qq)$.

\subsection{Local structure of the double of a Nash manifold with corners}
Let us show how to construct a Nash chart of $D(\Qq)$ from a Nash chart of $\Qq$. To lighten notations we only consider the first indices ($1,\ldots,m$ for some $1\leq m\leq d$) to construct the Nash charts, but the analogous construction can be done if one chooses other subsets of indices $1\leq i_1<\cdots<i_m\leq\ell$ for some $1\leq m\leq d$ such that $D(\Qq)\cap\{\t_{i_1}=0,\ldots,\t_{i_m}=0\}\neq\varnothing$.

\begin{lem}[Nash charts for $D(\Qq)$]\label{dq}
Let $U\subset M$ be an open semialgebraic set endowed with a Nash diffeomorphism $u:U\to\R^d$ such that $u(\Qq\cap U)=\{\x_1\geq0,\ldots,\x_m\geq0\}$ for some $1\leq m\leq d$. Let $\lambda_1,\ldots,\lambda_m:U\to\R$ be strictly positive Nash functions such that $u_k=h_k\lambda_k$ for $k=1,\ldots,m$. Then
\begin{multline*}
v:(U\times\R^\ell)\cap D(\Qq)\cap\{\t_{m+1}>0,\ldots,\t_\ell>0\}\to\R^d,\\ 
(x,t)\mapsto\Big(t_1\sqrt{\lambda_1(x)},\ldots,t_m\sqrt{\lambda_m(x)},u_{m+1}(x),\ldots,u_d(x)\Big)
\end{multline*}
is a Nash diffeomorphism.
\end{lem}
\begin{proof}
Denote $\varphi:=u^{-1}:\R^d\to U$, 
$$
\varphi_m^*:\R^d\to U,\ y:=(y_1,\ldots,y_d)\to\varphi(y_1^2,\ldots,y_m^2,y_{m+1},\ldots,y_d)
$$ 
and define
\begin{multline*}
\phi:\R^d\to(U\times\R^\ell)\cap D(\Qq)\cap\{\t_{m+1}>0,\ldots,\t_\ell>0\},\ 
y:=(y_1,\ldots,y_d)\mapsto\\
\Big(\varphi_m^*(y),\frac{y_1}{\sqrt{\lambda_1(\varphi_m^*(y))}},\ldots,\frac{y_m}{\sqrt{\lambda_m(\varphi_m^*(y))}},\sqrt{h_{m+1}(\varphi_m^*(y))},\ldots,\sqrt{h_\ell(\varphi_m^*(y))}\Big).
\end{multline*}
The reader can check that both $v$ and $\phi$ are well-defined Nash maps and that they are mutually inverse, so $v$ is a Nash diffeomorphism, as required.
\end{proof}
\begin{remark}
As $h_k$ is a Nash equation of $Y_k$ such that $d_yh_k:T_yY_k\to\R$ is surjective for $y\in Y_k$ and $\Qq=\{h_1\geq0,\ldots,h_\ell\geq0\}$ we deduce that if $U\subset M$ is an open semialgebraic set and $u_k:U\to\R$ is a Nash function such that $Y_k\cap U=\{u_k=0\}$ and $d_yu_k:T_yY_k\to\R$ is surjective for $y\in Y_k\cap U$, then there exists a Nash function $\lambda_k:U\to\R$ such that $u_k=h_k|_{U}\lambda_k$ and $\{\lambda_k=0\}=\varnothing$.
\end{remark}

\begin{cor}\label{hyper}
For each $i=1,\ldots,\ell$ define $h_i^*:D(\Qq)\to\R,\ (y,s)\mapsto h_i(y)$. Then the intersection $Y_i^*:=D(\Qq)\cap\{h_i^*=0\}=D(\Qq)\cap\{\t_i=0\}$ is a Nash hypersurface of $D(\Qq)$ for each $i=1,\ldots,\ell$. 
\end{cor}
\begin{proof}
Consider the Nash chart of $D(\Qq)$ introduced in Lemma \ref{dq}. The image of $\{\t_1=0\}\cap(U\times\R^\ell)\cap D(\Qq)\cap\{\t_{m+1}>0,\ldots,\t_\ell>0\}$ under $v$ is $\{0\}\times\R^{d-1}$. As the previous charts cover $Y_i^*$, we deduce that $Y_i^*$ is a Nash hypersurface of $D(\Qq)$, as required.
\end{proof}

\subsection{Global structure of the double of a Nash manifold with corners}
We next analyze the main properties of $D(\Qq)$ that will be useful for the applications proposed in Section \ref{s5}.

\begin{thm}[Nash double of a Nash manifold with corners]\label{ndnmwc}
The Nash subset $D(\Qq)$ of $M\times\R^\ell$ is a $d$-dimensional Nash manifold that satisfies the following properties:
\begin{itemize}
\item[(i)] The projection $\pi:D(\Qq)\to\Qq,\ (x,t)\mapsto x$ is proper and surjective.
\item[(ii)] The semialgebraic set $\Qq^\ell:=D(\Qq)\cap\{t_1\geq0,\ldots,t_\ell\geq0\}$ is a Nash manifold with corners, which is Nash diffeomorphic to $\Qq$ and the Nash closure of $\partial\Qq^\ell$ in $D(\Qq)$ is the Nash normal-crossings divisor $Y^\ell:=\bigcup_{i=1}^\ell\{t_i=0\}$ of $D(\Qq)$.
\item[(iii)] The restriction $\pi|_{\Qq^\ell}:\Qq^\ell\to\Qq$ is a semialgebraic homeomorphism.
\item[(iv)] The restriction $\pi|_{\Int(\Qq^\ell)}:\Int(\Qq^\ell)\to\Int(\Qq)$ is a Nash diffeomorphism.
\item[(v)] There exists an open semialgebraic neighborhood $W\subset D(\Qq)$ of $\Qq^\ell$ Nash diffeomorphic to $M$.
\end{itemize}
\end{thm}

\begin{proof}

The proof is conducted in several steps:

\noindent{\sc Step 1.} Let $(D(\Hh_1),\pi_1)$ be the Nash double of $\Hh_1:=\{h_1\geq0\}$. Recall that
$$
M^1:=D(\Hh_1)=\{(x,t_1)\in M\times\R:\ t_1^2-h_1(x)=0\}\quad\text{and}\quad\pi _1:D(\Hh_1)\to\Hh_1,\ (x,t_1)\to x.
$$
The Nash map $\pi_1$ is proper by Lemma \ref{doubleii}(iv). Consider the $(d-1)$-dimensional Nash submanifolds $Y_1^1:=Y_1\times\{0\}$ and $Y_i^1:=D(Y_i\cap\Hh_1)=(Y_i\times\R)\cap D(\Hh_1)$ (see \eqref{yi}) of $M^1$ for $i=2,\ldots,\ell$. Consider the Nash functions 
$$
h_i^1:M^1\to\R,\ (x,t_1)\mapsto h_i^1(x,t_1):=
\begin{cases}
t_1&\text{if $i=1$,}\\
h_i(x)&\text{if $i=2,\ldots,\ell$.}
\end{cases}
$$ 
The tangent space to $M^1$ at a point $(x,t_1)\in M^1$ is:
$$
T_{(x,t_1)}M^1=\{(v,s_1)\in T_xM\times\R:\ d_xh_1(v)-2t_1s_1=0\}.
$$
We claim: {\em $h_i^1:M^1\to\R$ is a Nash equation of $Y_i^1$ in $M^1$ such that $d_{(x,t_1)}h_i^1:T_{(x,t_1)}M^1\to\R$ is surjective for each $(x,t_1)\in Y_i^1$ and $i=1,\ldots,\ell$.}

Suppose first $i=1$. Then $\{h_1^1=0\}=\{h_1(x)=0,t_1=0\}=Y_1\times\{0\}=Y_1^1$, so $h_i^1$ is a Nash equation of $Y_i^1$ in $M^1$. In addition, for $(x,0)\in Y_1^1=Y_1\times\{0\}$ we have
$$
T_{(x,0)}M^1=\{(v,s_1)\in T_xM\times\R:\ d_xh_1(v)=0\}=T_xY_1\times\R
$$
and $d_{(x,0)}h_1^1:T_xY_1\times\R\to\R,\ (v,s_1)\mapsto s_1$ is surjective.

Suppose next $i=2,\ldots,\ell$ and observe that 
$$
Y_i^1=\{(x,t_1)\in M\times\R:\ t_1^2-h_1(x)=0, h_i(x)=0\}=\{(x,t_1)\in Y_i\times\R:\ t_1^2-h_1(x)=0\},
$$
so $h_i^1=0$ is a Nash equation of $Y_i^1$ in $M^1$. Recall that
$$
T_{(x,t_1)}M^1=\{(v,s_1)\in T_xM\times\R:\ d_xh_1(v)-2t_1s_1=0\}
$$
and $d_{(x,t_1)}h_i^1:T_{(x,t_1)}M^1\to\R,\ (v,s_1)\mapsto d_xh_i(v)$. Let us check that it is surjective:

\noindent{\sc Case 1:} $t_1\neq 0$. We pick a vector $v\in T_xM$ such that $d_xh_i(v)\neq 0$ and $(v,\frac{d_x h_1(v)}{2t_1})\in T_{(x,t_1)}M^1$. We have
$$
d_{(x,t_1)}h_i^1\Big(v,\frac{d_xh_1(v)}{2t_1}\Big)=d_xh_i(v)\neq 0.
$$

\noindent{\sc Case 2:} $t_1= 0$. We have $T_{(x,0)}M^1=\{(v,s)\in T_xM\times\R:\ d_xh_1(v)=0\}$. As $Y_1\cup Y_i$ is a normal-crossings divisor, $d_xh_1$ and $d_xh_i$ are linearly independent, so there exists $v\in T_xM$ such that $d_xh_i(v)\neq 0$ and $d_xh_1(v)=0$. Observe that $(v,0)\in T_{(x,0)}M^1$ and $d_{(x,t_1)}h_i^1(v,0)=d_xh_i^1(v)\neq 0$.

Consequently, $d_{(x,t_1)}h_i^1:T_{(x,t_1)}M^1\to\R$ is surjective (in both cases), as claimed.

By Theorem \ref{embedding} there exist an open semialgebraic neighborhood $N^0\subset M$ of $\Hh_1$ and a Nash embedding $\phi_1:N^0\to M^1$ such that $\phi_1(\Hh_1)=\Hh_{1+}$ and $W^1:=\phi_1(N^0)$ is an open semialgebraic subset of $M^1$ that contains the Nash manifold with corners $\phi_1(\Qq)$ as a closed subset. In addition, $\phi_1|_{Y_1}=\id_{Y_1\times\{0\}}$, $\phi_1(Y_i)\subset Y_i^1$ for $i=2,\ldots,\ell$ and $\phi_1|_{\Hh_1}:\Hh_1\to\Hh_{1+}$ is close to $(\pi_1|_{\Hh_{1+}})^{-1}$. Consider the Nash subset $Y^1:=\bigcup_{i=1}^\ell Y_i^1$ of $M^1$, whose Nash irreducible components are Nash submanifolds of $M^1$. We have $\phi_1(Y\cap N^0)=Y^1\cap W^1$. In addition, $\pi_{1-}|_{\Int(\Hh_{1-})}:\Int(\Hh_{1-})\to\Int(\Hh_1)$ is a Nash diffeomorphism and $\pi_{1-}(Y^1\cap\Int(\Hh_{1-}))=Y\cap\Int(\Hh_1)$. As $Y$ is a Nash normal-crossings divisor of $M$ and $M^1=D(\Hh_1)=W^1\cup\Int(\Hh_{1-})$, we conclude that $Y^1$ is a Nash normal-crossings divisor of $M^1$. 

Denote $\Hh_1^{1}:=\Hh_{1+}$ and $\Hh_i^{1}:=\pi_1^{-1}(\Hh_i)=\{h_i^1\geq0\}$ for $i=2,\ldots,\ell$. By Lemma \ref{inthi} the intersection
$$
\Qq^1:=\bigcap_{i=1}^\ell \Hh_i^{1}=\{(x,t_1)\in M\times\R: t_1^2-h_1(x)=0,t_1\geq0,h_2(x)\geq0,\ldots,h_\ell(x)\geq0\}\subset M^1
$$
is a Nash manifold with corners and the Nash closure in $M^1$ of $\partial\Qq^1$ is $Y^1$. Observe that $\pi_1(\Qq^1)=\Qq$, the restriction $\pi_1|_{\Qq^1}:\Qq^1\to\Qq$ is a semialgebraic homeomorphism and the restriction $\pi_1|_{\Int(\Qq^1)}:\Int(\Qq^1)\to\Int(\Qq)$ is a Nash diffeomorphism. We claim: $\phi_1(\Qq)=\Qq^1$. As $\phi_1$ is a Nash diffeomorphism, it is enough to prove: {\em $\phi_1(\Hh_1\cap\Hh_i)=\Hh_1^1\cap\Hh_i^1=\Hh_{1+}\cap\pi_1^{-1}(\Hh_i)$ for $i=2,\ldots,\ell$.} 

As $\phi_1:N^0\to W^1$ is a Nash diffeomorphism, it is enough to check
$$
\phi_1(\Int(\Hh_1)\cap\Int(\Hh_i))=\Int(\Hh_{1+})\cap\pi_1^{-1}(\Int(\Hh_i)).
$$
Observe that $\Int(\Hh_1)\cap\Int(\Hh_i)$ is a union of connected components of $N^0\setminus(Y_1\cup Y_i)$. In addition, $\phi_1(Y_1)=Y_1\times\{0\}$ and $\phi_1(Y_i)\subset D(Y_i\cap\Hh_1)$, so 
$$
\phi_1(Y_i\cap\Int(\Hh_1))\subset D(Y_i\cap\Hh_1)\cap\Int(\Hh_{1+}).
$$
As $\phi_1|_{\Int(\Hh_1)}:\Int(\Hh_1)\to\Int(\Hh_{1+})$ is a Nash diffeomorphism, $\phi_1(Y_i\cap\Int(\Hh_1))$ is a closed Nash submanifold of the Nash manifold $D(Y_i\cap\Hh_1)\cap\Int(\Hh_{1+})$ of the same dimension. As $\phi_1$ is close to $(\pi_1|_{\Hh_1,+})^{-1}$, both Nash manifolds have the same number of connected components and 
$$
\phi_1(Y_1\cap\Int(\Hh_1))=D(Y_i\cap\Hh_1)\cap\Int(\Hh_{1+}).
$$
Note that $\Int(\Hh_{1+})\cap\pi _1^{-1}(\Hh_i)$ is a union of connected components of $\Int(\Hh_{1+})\setminus D(Y_i\cap\Hh_1)$. As $\phi_1$ is close to $(\pi_1|_{\Hh_{1+}})^{-1}$, we conclude that the Nash diffeomorphism $\phi_1$ provides a bijection between the connected components of $\Int(\Hh_1)\cap\Int(\Hh_i)$ and the connected components of $\Int(\Hh_{1+})\cap\pi _1^{-1}(\Hh_i)$. Consequently, $\phi_1(\Int(\Hh_1)\cap\Int(\Hh_i))=\Int(\Hh_{1+})\cap\pi _1^{-1}(\Hh_i)$, as claimed.

Summarizing, the objects we have constructed in this step are the following: 
\begin{itemize}
\item The $d$-dimensional Nash manifold $M^1:=D(\Hh_1)$ and the projection $\pi_1:D(\Hh_1)\to\Hh_1,\ (x,t_1)\mapsto x$. 
\item The Nash manifolds with boundary $\Hh_1^{1}:=\Hh_{1+}$ and $\Hh_i^{1}:=\pi_1^{-1}(\Hh_i)=\{h_i^1\geq0\}$ for $i=2,\ldots,\ell$. The boundary of $\Hh_1^{1}$ is $Y_1^1:=Y_1\times\{0\}$ and the boundary of $\Hh_i^{1}$ is $Y_i^1:=D(Y_i\cap\Hh_1)$ for $i=2,\ldots,\ell$.
\item The Nash normal-crossings divisor $Y^1:=\bigcup_{i=1}^\ell Y_i^1$ of $M^1$.
\item The Nash equations
$$
h_i^1:M^1\to\R,\ (x,t_1)\mapsto\begin{cases}
t_1&\text{if $i=1$}\\
h_i(x)&\text{if $i=2,\ldots,\ell$,}
\end{cases}
$$
which satisfy: $d_{(x,t_1)}h_i^1:T_{(x,t_1)}M^1\to\R$ is surjective for each $(x,t_1)\in Y_i^1$ and $i=1,\ldots,\ell$.
\item The Nash manifold with corners $\Qq^1:=\bigcap_{i=1}^\ell \Hh_i^1$, which is mapped by the projection $\pi_1$ onto $\Qq$. In addition, the restriction $\pi_1|_{\Qq^1}:\Qq^1\to\Qq$ is a semialgebraic homeomorphism and the restriction $\pi_1|_{\Int(\Qq^1)}:\Int(\Qq^1)\to\Int(\Qq)$ is a Nash diffeomorphism.
\item The open semialgebraic neighborhoods $N^0\subset M$ of $\Hh_1$ and $W^1\subset M^1$ of $\Qq^1$ and the Nash diffeomorphism $\phi_1:N^0\to W^1$ that maps $\Qq$ onto $\Qq^1$, $Y_1$ onto $Y_1^1$ and $Y_i$ inside $Y_i^1$ for $i=2,\ldots,\ell$.
\end{itemize}

\noindent{\sc Step 2.} Let $(D(\Hh_2^1),\pi_2)$ be the Nash double of $\Hh_2^1=\{h_2^1\geq0\}$. We have
\begin{multline*}
M^2:=D(\Hh_2^1)=\{(x,t_1,t_2)\in(M\times\R)\times\R:\ t_1^2-h_1(x)=0, t_2^2-h_2^1(x,t_1)=0\}=\\
\{(x,t_1,t_2)\in M\times\R^2:\ t_1^2-h_1(x)=0, t_2^2-h_2(x)=0\}
\end{multline*}
and $\pi_2:D(\Hh_2^1)\to\Hh_2^1,\ (x,t_1,t_2)\mapsto(x,t_1)$. The Nash map $\pi_2$ is proper by Lemma \ref{doubleii}(iv). Consider the $(d-1)$-dimensional Nash submanifolds $Y_1^2:=D(Y_1^1\cap\Hh_2^1)$, $Y_2^2=Y_2^1\times\{0\}$ and $Y_i^2:=D(Y_i^1\cap\Hh_2^1)=(Y_i^1\times\R)\cap D(\Hh_2^1)$ (see \eqref{yi}) of $M^2$ for $i=3,\ldots,\ell$. Define
$$
h_i^2:M^2\to\R,\ (x,t_1,t_2)\mapsto h_i^2(x,t_1,t_2)=
\begin{cases}
h_1^1(x,t_1)=t_1&\text{if $i=1$,}\\
t_2&\text{if $i=2$,}\\
h_i^1(x,t_1)=h_i(x)&\text{if $i=3,\ldots,\ell$.}\\
\end{cases}
$$
The tangent space to $M^2$ at a point $(x,t_1,t_2)\in M^2$ is:
$$
T_{(x,t_1,t_2)}M^1=\{(v,s_1,s_2)\in T_xM\times\R^2:\ d_xh_1(v)-2t_1s_1=0,d_xh_2(v)-2t_2s_2=0\}.
$$
The reader can check (analogously as we have done in {\sc Step 1}): {\em $h_i^2:M^2\to\R$ is a Nash equation of $Y_i^2$ in $M^2$ such that $d_{(x,t_1,t_2)}h_i^2:T_{(x,t_1,t_2)}M^2\to\R$ is surjective for each $(x,t_1,t_2)\in Y_i^2$ and $i=1,\ldots,\ell$.}

By Theorem \ref{embedding} there exist an open semialgebraic neighborhood $N^1\subset M^1$ of $\Hh_2^1$ and a Nash embedding $\phi_2:N^1\to M^2$ such that $\phi_2(\Hh_2^1)=\Hh_{2+}^1$ and $W^2:=\phi_2(N^1)$ is an open semialgebraic subset of $M^2$ that contains the Nash manifold with corners $\phi_2(\Qq^1)$ as a closed subset. In addition, $\phi_2|_{Y_2^1}=\id_{Y_2^1\times\{0\}}$, $\phi_2(Y_i^1)\subset Y_i^2$ for $i=1,3,\ldots,\ell$ and $\phi_2|_{\Hh_2^1}:\Hh_2^1\to\Hh_{2+}^1$ is close to $(\pi_2|_{\Hh_{2+}^1})^{-1}$. Consider the Nash subset $Y^2:=\bigcup_{i=1}^\ell Y_i^2$ of $M^2$, whose Nash irreducible components are Nash submanifolds of $M^2$. We have $\phi_2(Y^1\cap N^1)=Y^2\cap W^2$. In addition, $\pi_{2-}|_{\Int(\Hh_{2-}^1)}:\Int(\Hh_{2-}^1)\to\Int(\Hh_2^1)$ is a Nash diffeomorphism and $\pi_{2-}(Y^2\cap\Int(\Hh_{2-}^1))=Y^1\cap\Int(\Hh_2^1)$. As $Y^1$ is a Nash normal-crossings divisor of $M^1$ and $M^2=D(\Hh_2^1)=W^2\cup\Int(\Hh_{2-}^1)$, we conclude that $Y^2$ is a Nash normal-crossings divisor of $M^2$. 

Denote $\Hh_2^{2}:=\Hh_{2+}^1$ and $\Hh_i^{2}:=\pi_2^{-1}(\Hh_i^1)=\{h_i^2\geq0\}$ for $i=1,3,\ldots,\ell$. By Lemma \ref{inthi} the intersection
\begin{multline*}
\Qq^2:=\bigcap_{i=1}^\ell \Hh_i^{2}=\{(x,t_1,t_2)\in M\times\R^2: t_1^2-h_1(x)=0,t_2^2-h_2(x)=0,t_1\geq0,t_2\geq0,\\
h_3(x)\geq0,\ldots,h_\ell(x)\geq0\}\subset M^2
\end{multline*}
is a Nash manifold with corners and the Nash closure in $M^2$ of $\partial\Qq^2$ is $Y^2$. Observe that $\pi_2(\Qq^2)=\Qq^1$, the restriction $\pi_2|_{\Qq^2}:\Qq^2\to\Qq^1$ is a semialgebraic homeomorphism and the restriction $\pi_2|_{\Int(\Qq^2)}:\Int(\Qq^2)\to\Int(\Qq^1)$ is a Nash diffeomorphism. In particular, $(\pi_1\circ\pi_2)(\Qq^2)=\Qq$. The reader can check (analogously as we have done in {\sc Step 1}): $\phi_2(\Qq^1)=\Qq^2$. Consequently, $(\phi_2\circ\phi_1)(\Qq)=\Qq^2$.

\noindent{\sc Recursive steps.} 
We proceed recursively for $k=1,\ldots,\ell$ and we construct: 
\begin{itemize}
\item The $d$-dimensional Nash manifold $M^k$ and the projection 
$$
\pi_k:M^k\to\Hh_k^{k-1},\ (x,t_1,\ldots,t_k)\mapsto(x,t_1,\ldots,t_{k-1}),
$$
which is proper by Lemma \ref{doubleii}(iv).
\item The Nash manifolds with boundary $\Hh_i^k$ and the $(d-1)$-dimensional Nash submanifolds $Y_i^k=\partial\Hh_i^k$ of $M^k$. 
\item The Nash normal-crossings divisor $Y^k:=\bigcup_{i=1}^\ell Y_i^k$ of $M^k$. 
\item The Nash equations
$$
h_i^k:M^k\to\R,\ (x,t_1,\ldots,t_k)\mapsto\begin{cases}
t_i&\text{if $i=1,\ldots,k$}\\
h_i(x)&\text{if $i=k+1,\ldots,\ell$,}
\end{cases}
$$
which satisfy: $d_{(x,t_1,\ldots,t_k)}h_i^k:T_{(x,t_1,\ldots,t_k)}M^k\to\R$ is surjective for each $(x,t_1,\ldots,t_k)\in Y_i^k$ and $i=1,\ldots,\ell$.
\item The Nash manifold with corners contained in $M^k$
\begin{equation*}
\begin{split}
\Qq^k:=\bigcap_{i=1}^\ell \Hh_i^k=\{(x,t_1,\ldots,t_k)&\in M\times\R^k:\ t_1^2-h_1(x)=0,\ldots,t_k^2-h_k(x)=0,\\
&t_1\geq0,\ldots,t_k\geq0,h_{k+1}(x)\geq0,\ldots,h_\ell(x)\geq0\},
\end{split}
\end{equation*}
which is mapped by the projection $\pi_k$ onto $\Qq^{k-1}$. Consequently, $(\pi_1\circ\cdots\circ\pi_k)(\Qq^k)=\Qq$. In addition, the restriction $\pi_k|_{\Qq^k}:\Qq^k\to\Qq^{k-1}$ is a semialgebraic homeomorphism and the restriction $\pi_k|_{\Int(\Qq^k)}:\Int(\Qq^k)\to\Int(\Qq^{k-1})$ is a Nash diffeomorphism.
\item The open semialgebraic neighborhoods $N^{k-1}\subset M^{k-1}$ of $\Hh_k^{k-1}$ and $W^k\subset M^k$ of $\Qq^k$ and the Nash diffeomorphism $\phi_k:N^{k-1}\to W^k$ that maps $\Qq^{k-1}$ onto $\Qq^k$, $Y_k^{k-1}$ onto $Y_k^k$ and $Y_i^{k-1}$ inside $Y_i^k$ for $i\neq k$. Consequently, $(\phi_k\circ\cdots\circ\phi_1)(\Qq)=\Qq^k$.
\end{itemize}

At this point observe that 
$$
D(\Qq)=M^\ell=\{(x,t)\in M\times\R^\ell: t_1^2-h_1(x)=0,\ldots, t_\ell^2-h_\ell(x)=0\}
$$
is a $d$-dimensional Nash manifold and $\Qq^\ell=\{t_1\geq0,\ldots,t_\ell\geq0\}\subset D(\Qq)$ is a Nash manifold with corners. The Nash closure of the boundary
$$
\partial\Qq^\ell=\bigcup_{i=1}^\ell\{t_1\geq0,\cdots,t_{i-1}\geq0,t_i=0,t_{i+1}\geq0,\cdots,t_\ell\geq0\}
$$
is the Nash normal-crossings divisor $Y^\ell=\bigcup_{i=1}^\ell\{t_i=0\}$ of $D(\Qq)$. After shrinking $N^0$ if necessary, the composition $\phi:=\phi_\ell\circ\cdots\circ\phi_1:N^0\to W:=\phi(N^0)$ is a Nash diffeomorphism such that $W$ is an open semialgebraic neighborhood of $\Qq^\ell$ in $D(\Qq)$ and $\phi(\Qq)=\Qq^\ell$. The composition $\pi:=\pi_1\circ\cdots\circ\pi_\ell:D(\Qq)\to\Qq$ is a proper surjective Nash map such that the restriction $\pi|_{\Qq^\ell}:\Qq^\ell\to\Qq$ is a semialgebraic homeomorphism and the restriction $\pi|_{\Int(\Qq^\ell)}:\Int(\Qq^\ell)\to\Int(\Qq)$ is a Nash diffeomorphism, as required.
\end{proof}

\section{Applications}\label{s5}

In this section we prove several applications proposed in the introduction as consequences of \cite[Cor.1.10]{cf2} and Theorem \ref{fold}.

\subsection{Weak Nash uniformization of closed semialgebraic sets.}

We prove Theorem \ref{red3}.

\begin{proof}[Proof of Theorem \em \ref{red3}]
By \cite[Cor.1.10]{cf2} applied to the components of $\Ss$ connected by analytic paths and \cite[Rmk.II.1.5]{sh} to the strictly positive continuous semialgebraic functions $\veps_i$ we may assume that:
\begin{itemize}
\item $\Qq:=\Ss$ is a connected Nash manifolds with corners, which is closed in $\R^n$, (so $\Ss$ has only one component connected by analytic paths that is a closed semialgebraic subset of $\R^n$, use \cite[Main Thm.1.8]{fe4}).
\item The Zariski closure $X\subset\R^n$ of $\Qq$ is a $d$-dimensional non-singular real algebraic set.
\item The Zariski closure $Y$ of $\partial\Qq$ is a normal-crossings divisor of $X$.
\end{itemize}
If $\Ss$ is compact, we may assume in addition that $X$ is compact and connected.

By \cite[Thm.1.11, 1.12]{fgr} there exists an open semialgebraic neighborhood $U\subset X$ of $\Qq$ such that:
\begin{itemize}
\item The Nash closure $Y$ of $\partial\Qq$ in $U$ is a Nash normal-crossings divisor of $U$ and $\Qq\cap Y=\partial\Qq$.
\item For every $x\in\partial\Qq$ the analytic closure of the germ $\partial\Qq_x$ is $Y_x$.
\end{itemize}
Choose a strictly positive continuous semialgebraic function $\veps_0:X\to\R$ and an open semialgebraic neighborhood $U_1\subset X$ of $\Qq$ such that $\cl(U_1)\subset U$ (Lemma \ref{co}). In case $\Qq$ is compact, we may choose $\veps_0$ constant. Define $M:=\{x\in U_1:\ \dist(x,\Qq)<\veps_0(x)\}$, which is a Nash manifold such that $\Qq\subset M\subset\cl(M)\subset U$. 

By \cite[Thm.VI.2.1]{sh} and its proof there exists an ${\mathcal S}^1$ function $g:U\to\R$ such that $M$ is diffeomorphic to $g^{-1}((-1,+\infty))$ and $\Hh:=g^{-1}([0,+\infty))\subset M$ is an ${\mathcal S}^1$ manifold with boundary that contains $\Qq$ in its interior. In case $\Qq$ is compact, also $\Hh$ is compact. We may assume in addition that the differential $d_xg:T_xM\to\R$ is surjective for each $x\in g^{-1}(0)$ (see Lemma \ref{equation}). As $\cl(M)\subset U$ is a closed subset of $X$, there exists an ${\mathcal S}^1$ extension $G$ of $g|_{\cl(M)}$ to $X$ such that $G^{-1}(0)=g^{-1}(0)$. Let $H:X\to\R$ be a Nash approximation of $G$ in the ${\mathcal S}^1$ topology (in case $X$ is compact, we may assume $H$ is a polynomial). As $G$ is strictly positive on the closed semialgebraic set $\Qq$, we may assume $H$ is strictly positive on $\Qq$, so $H^{-1}(0)\subset M\setminus\Qq$. As we deal with ${\mathcal S}^1$ approximation, we may assume $d_xH:T_xX\to\R$ is surjective for each $x\in H^{-1}(0)$. Thus, $\Hh_{\veps_0}':=H^{-1}([0,+\infty))\subset M$ is a Nash manifold with (non-singular) boundary $Z_{\veps_0}:=H^{-1}(0)$, which is a Nash submanifold of $X$ of dimension $d-1$. In case $X$ is compact, also $\Hh_{\veps_0}'$ is compact and $Z_{\veps_0}:=H^{-1}(0)$ is a (compact) non-singular real algebraic subset of $X$. As $\Qq$ is connected, it is contained in one of the connected components $\Hh_{\veps_0}$ of $\Hh_{\veps_0}'$. Observe that $\Hh_{\veps_0}$ is a Nash manifold with boundary that contains $\Qq$ in its interior $\Int(\Hh_{\veps_0})$, which is connected. Thus, $\Int(\Hh_{\veps_0})$ is a connected component of $X\setminus Z_{\veps_0}$. 

In case $\Qq$ is non-compact and $Z_{\veps_0}$ is not a non-singular algebraic subset of $X$, we apply \cite[Lem.C.1]{fe3} to assume up to a suitable Nash embedding of $X$ in some affine space $\R^p$ that $X\subset\R^p$ is a non-singular algebraic set and there exists an (algebraic) normal-crossings divisor $Y_{\veps_0}$ such that the Nash submanifold $Z_{\veps_0}$ is a union of connected components of $Y_{\veps_0}$. By \cite[\S3.E.1, Cor.4.4]{fgh3} all the points of $Z_{\veps_0}$ are regular in the sense of \cite[\S3.A.3,\S3.B.3]{fgh3}. By \cite{bm} (which solves exactly the singularities of $Y_{\veps_0}\setminus Z_{\veps_0}$ and keep $X\setminus(Y_{\veps_0}\setminus Z_{\veps_0})$ invariant up to a biregular diffeomorphism) we may assume in addtion that $Y_{\veps_0}$ is a non-singular algebraic subset of $X$.

By Theorem \ref{fold} there exists a proper Nash map $f_0:\cl(M)\to\Qq$ such that $f_0(M)=f_0(\Qq)=\Qq$, the restriction $f_0|_{\Hh_{\veps_0}\setminus f_0^{-1}(\Tt_{\veps_0})}:\Hh_{\veps_0}\setminus f_0^{-1}(\Tt_{\veps_0})\to\Qq\setminus\Tt_{\veps_0}$ is a Nash diffeomorphism, where $\Tt_{\veps_0}:=\{x\in\Qq:\ \dist(x,\partial\Qq)<\veps_0\}$, because $\Hh_{\veps_0}\setminus f_0^{-1}(\Tt_{\veps_0})\subset\Int(\Qq)$ (after shrinking $M$ at the beginning if needed). As $\Qq\subset\Hh_{\veps_0}\subset M$, we have $f_0(\Hh_{\veps_0})=\Qq$. As $\Hh_{\veps_0}\subset\cl(M)$ is closed and $f_0:\cl(M)\to\Qq$ is proper, $f:=f_0|_{\Hh_{\veps_0}}:\Hh_{\veps_0}\to\Qq$ is also a proper map, as required. 
\end{proof}

\subsection{An alternative proof to Theorem \ref{main1}}\label{altcontru}
In order to give an alternative proof of Theorem \ref{main1} based on Theorem \ref{red3}, we need some preliminary results. We start with the following Lemma that extends \cite[Lem.2.8]{fe3} to Nash manifolds with boundary.

\begin{lem}\label{approxb}
Let $\Hh_1\subset\R^m$ and $\Hh_2\subset\R^n$ be Nash manifolds with (non-singular) boundary. Let $f:\Hh_1\to\Hh_2$ be a semialgebraic homeomorphism. Then every continuous semialgebraic map $g:\Hh_1\to\Hh_2$ close to $f$ in the ${\mathcal S}^0$ topology, such that $g(\partial\Hh_1)\subset\partial\Hh_2$, is surjective. 
\end{lem}
\begin{proof}
As $f:\Hh_1\to\Hh_2$ is a semialgebraic homeomorphism, $f(\partial\Hh_1)=\partial\Hh_2$ (use invariance of domain). Consider the Nash doubles $(D(\Hh_i),\pi_i)$ and $\Hh_{i,\epsilon}=D(\Hh_i)\cap\{\epsilon t\geq0\}$ where $\epsilon=\pm$. Recall that $\pi_{i\epsilon}:\Hh_{i,\epsilon}\to\Hh_i$ is a semialgebraic homeomorphism.
As $f(\partial\Hh_1)=\partial\Hh_2$, the map
$$
F:D(\Hh_1)\to D(\Hh_2),\ (x,t)\mapsto\begin{cases}
(\pi_{2+})^{-1}\circ f\circ\pi_{1+}(x,t)&\text{ if } t\geq0,\\
(\pi_{2-})^{-1}\circ f\circ\pi_{1-}(x,t)&\text{ if } t\leq 0,
\end{cases}
$$
is well-defined and semialgebraic. Observe that $F$ is bijective and
$$
F^{-1}:D(\Hh_2)\to D(\Hh_1),\ (y,s)\mapsto\begin{cases}
(\pi_{1+})^{-1}\circ f^{-1}\circ\pi_{2+}(y,s)&\text{ if } s\geq0,\\
(\pi_{1-})^{-1}\circ f^{-1}\circ\pi_{2-}(y,s)&\text{ if } s\leq 0.
\end{cases}
$$
Let us check: {\em $F$ is continuous}. Once this is done, the same proof shows that $F^{-1}$ is continuous, so $F$ is a semialgebraic homeomorphism. 

As $F$ is continuous on both $\Hh_{1,+}$, $\Hh_{1,-}$ and 
$
\pi_1|_{\Hh_{1+}\cap\Hh_{1-}}=\pi_1|_{\partial\Hh_1}=\id_{\partial\Hh_1},
$
we conclude by the pasting lemma that $F$ is continuous on $D(\Hh_1)$.

Let $g:\Hh_1\to\Hh_2$ be a continuous semialgebraic map. If $g(\partial\Hh_1)\subset\partial\Hh_2$, the map
$$
G:D(\Hh_1)\to D(\Hh_2),\ (x,t)\mapsto 
\begin{cases}
(\pi_{2+})^{-1}\circ g\circ\pi_{1+}(x,t)&\text{if $t\geq0$}\\
(\pi_{2-})^{-1}\circ g\circ\pi_{1-}(x,t)&\text{if $t\leq0$}
\end{cases}
$$
is well-defined, continuous and semialgebraic for $\epsilon=\pm$. The proof is analogous to the one for $F$. Next, by \cite[Lem.2.8]{fe3} there exists a strictly positive continuous semialgebraic function $\veps:D(\Hh_1)\to\R$ such that if $\|G-F\|<\veps$, then $G$ is surjective. 

For $\epsilon=\pm$, consider the strictly positive continuous semialgebraic functions 
$$
\veps_{\epsilon}:=\veps\circ\pi_{1\epsilon}^{-1}:\Hh_1\to\R
$$ 
and $\veps^*=\min\{\veps_+,\veps_-\}$. Let $\delta:\Hh_1\to\R$ be a strictly positive continuous semialgebraic function such that if $g:\Hh_1\to\Hh_2$ is a continuous semialgebraic map such that $\|f-g\|<\delta$, then $\|(\pi_{2\epsilon})^{-1}\circ f-(\pi_{2\epsilon})^{-1}\circ g\|<\veps^*$ for $\epsilon=\pm$ (see \cite[Rmk.II.1.5]{sh}). Thus, $\|(\pi_{2\epsilon})^{-1}\circ f\circ\pi_{1\epsilon}-(\pi_{2\epsilon})^{-1}\circ g\circ\pi_{1\epsilon}\|<\veps^*\circ\pi_{1\epsilon}\leq\veps$ for $\epsilon=\pm$. Consequently, if $\|f-g\|<\delta$, then $\|F-G\|<\veps$, so $G$ is surjective. Taking into account the definition of $G$, we conclude that $g$ is also surjective, as required.
\end{proof}

\subsubsection{Preliminary reductions to prove Theorem \em \ref{main1}}
The proof of Theorem \ref{main1} requires some additional preparation. As the image of a compact semialgebraic set connected by analytic paths under a Nash map is a compact semialgebraic set connected by analytic paths \cite[Main Thm.1.4]{fe3}, we are reduced to prove: \em If $\Ss\subset\R^n$ is a compact semialgebraic set connected by analytic paths, there exists a Nash map $f:\R^d\to\R^n$ such that $f(\ol{\Bb}_d)=\Ss$\em. 

By Theorem \ref{red3} we may assume that $\Hh:=\Ss$ is a connected compact Nash manifold with smooth boundary. Let $M:=D(\Hh)$ be the Nash double of $\Hh$ (which is a connected compact Nash manifold) and consider the proper surjective Nash map $\pi:D(\Hh)\to\Hh$ introduced in Lemma \ref{doubleii}. Thus, we may assume that $M:=\Ss$ is a connected compact Nash manifold. Alternatively, we may use \cite[Cor.1.10]{cf2} and Theorem \ref{ndnmwc} to obtain a similar reduction.

Consider the simplex $\Delta_d:=\{\x_1\geq0,\ldots,\x_d\geq0,\x_1+\cdots+\x_d\leq1\}\subset\R^d$. In \cite[Cor.2.8]{fu6} we proved that the closed ball $\ol{\Bb}_d$ is a polynomial image of the simplicial prism $\Delta_{d-1}\times[0,1]$. Thus, the proof of Theorem \ref{main1} is reduced to show the following: \em If $M\subset\R^n$ is a connected compact Nash manifold, there exists a Nash map $f:\R^d\to\R^n$ such that $f(\sigma\times[a,b])=M$, where $\sigma\subset\R^{d-1}$ is a $(d-1)$-dimensional simplex and $[a,b]\subset\R$ a compact interval\em. 

Let $\{U_i\}_{i=1}^r$ be a finite covering of $M$ equipped with Nash diffeomorphisms $u_i:U_i\to\R^d$ such that $M=\bigcup_{i=1}^mu_i^{-1}(\Delta_d)$. Define 
$$
\Delta_d':=\{\x_1\geq-2,\ldots,\x_d\geq-2,\x_1+\cdots+\x_d\leq d+1\}
$$
and observe that $\Delta_d\subset\ol{\Bb}_d\subset\Int(\Delta_d')$. Let $\psi:\R^d\to\R^d$ be an affine isomorphism such that $\psi(\Delta_d)=\Delta_d'$. 
 
\begin{lem}\label{simp3}
Consider the $(d-1)$-simplex $\sigma:=\{\x_1\geq0,\ldots,\x_{d-1}\geq0,\x_1+\cdots+\x_{d-1}\leq1\}\subset\R^{d-1}$. Then the surjective Nash map $f:\sigma\times[0,1]\to\Delta_d',\ (x,t)\mapsto\psi(((1-t)x,t))$ restricts to a Nash diffeomorphism $f|_{\sigma\times[0,1)}:\sigma\times[0,1)\to\Delta_d'\setminus\{\psi((0,\ldots,0,1))\}$.
\end{lem}
\begin{proof}
As $\psi$ is an affine isomorphism such that $\psi(\Delta_d)=\Delta_d'$, it is enough to prove that the surjective Nash map $\sigma\times[0,1]\to\Delta_d,\, (x,t)\mapsto((1-t)x,t)$ restricts to a Nash diffeomorphism 
$$
\varphi:\sigma\times[0,1)\to\Delta_d\setminus\{(0,\ldots,0,1)\}.
$$
This is straightforward, because the Nash map 
$$
\phi:\Delta_d\setminus\{(0,\ldots,0,1)\}\to\sigma\times[0,1),\ (x_1,\ldots,x_d)\mapsto\Big(\frac{x_1}{1-x_d},\ldots,\frac{x_{d-1}}{1-x_d},x_d\Big)
$$
is the inverse of $\varphi$.
\end{proof}

We are ready to prove Theorem \ref{main1}.

\subsubsection{Proof of Theorem \em \ref{main1}}

We keep the notations quoted above. In view of the contruction made in \cite[\S2.3]{cf1}, to obtain a global Nash map $f:\R^d\to\R^n$ such that $f(\sigma\times[0,2r-1])=M$, it is enough to prove: {\em The connected compact Nash manifold $M\subset\R^n$ is the image of $\sigma\times[0,2r-1]$ under a Nash map $f:\sigma\times[0,2r-1]\to\R^n$}.

Let us construct first a surjective continuous semialgebraic function 
$$
h:\sigma\times[0,2r-1]\to M=\bigcup_{i=1}^ru_i^{-1}(\Delta_d)=\bigcup_{i=1}^ru_i^{-1}(\ol{\Bb}_d)=\bigcup_{i=1}^ru_i^{-1}(\Delta_d')
$$
that is Nash on $\sigma\times(\bigcup_{i=1}^r(2(i-1),2(i-1)+1)$ and satisfies 
$$
u_i^{-1}(\ol{\Bb}_d)\subset h(\Delta_d\times(2(i-1),2(i-1)+1)),
$$ 
for each $i=1,\ldots,r$. 

Let $f:\sigma\times[0,1]\to\Delta_d'$ be the surjective Nash map introduced in Lemma \ref{simp3} and define 
$$
h_i:\sigma\times[2(i-1),2(i-1)+1]\to u_i^{-1}(\Delta_d'),\ (x,t)\mapsto u_i^{-1}(f(x,t-2(i-1))),
$$
which is a surjective Nash map whose restriction to $\sigma\times\{2(i-1)+1\}$ is constant and its restriction to $\sigma\times[2(i-1),2(i-1)+1)$ is a Nash diffeomorphism onto its image. Write $\{p_i\}:=h_i(\sigma\times\{2(i-1)+1\})$. Let $\alpha_i:[2(i-1)+1,\tfrac{1}{2}+2(i-1)+1]\to M$ be a continuous semialgebraic map such that $\alpha_i(2(i-1)+1)=p_i$ and $\alpha_i(\tfrac{1}{2}+2(i-1)+1)=u_{i+1}^{-1}(0)$. Consider the continuous semialgebraic map
$$
g_i:\sigma\times[2(i-1)+1,\tfrac{1}{2}+2(i-1)+1]\to M,\ (x,t)\mapsto\alpha_i(t).
$$
Let $b_1,\ldots,b_d$ be the vertices of $\Delta_d'$ different from $\psi(0,\ldots,0,1)$. Define $x_d:=1-x_1-\cdots-x_{d-1}$ and consider the continuous semialgebraic map 
\begin{align*}
g_i':\sigma\times\Big[\tfrac{1}{2}+2(i-1)+1,2(i-1)+2\Big]&\to\R^d,\\
 (x,t)&\mapsto u_{i+1}^{-1}\Big(\sum_{i=1}^dx_ib_i2\Big(t-\tfrac{1}{2}-2(i-1)-1\Big)\Big),
\end{align*}
which satisfies $g_i'(x,\tfrac{1}{2}+2(i-1)+1)=u_{i+1}^{-1}(0)$ and $g_i'(x,2(i-1)+2)=u_{i+1}^{-1}(\sum_{i=1}^dx_ib_i)$.

Define $h:\sigma\times[0,2r-1]\to M$, as
$$
(x,t)\mapsto\begin{cases}
h_i(x,t)&\text{if $(x,t)\in\sigma\times[2(i-1),2(i-1)+1]$,}\\
g_i(x,t)&\text{if $(x,t)\in\sigma\times[2(i-1)+1,\tfrac{1}{2}+2(i-1)+1]$,}\\
g_i'(x,t)&\text{if $(x,t)\in\sigma\times[\tfrac{1}{2}+2(i-1)+1,2(i-1)+2]$,}\\
h_r(x,t)&\text{if $(x,t)\in\sigma\times[2(r-1),2r-1]$}
\end{cases}
$$ 
for $i=1,\ldots,r-1$. Observe that $h$ is a continuous semialgebraic map such that for each $i=1,\ldots,r$:
\begin{itemize}
\item the restriction $h|_{\sigma\times(2(i-1),2(i-1)+1)}$ is a Nash diffeomorphism onto its image.
\item $h(\sigma\times[2(i-1),2(i-1)+1])=u_i^{-1}(\Delta_d')$. 
\item $u_i^{-1}(\ol{\Bb}_d)\subset u_i^{-1}(\Int(\Delta_d'))\subset h(\sigma\times(2(i-1),2(i-1)+1))$.
\end{itemize} 
We conclude that the restriction of $h$ to $\sigma\times(\bigcup_{i=1}^r(2(i-1),2(i-1)+1)$ is Nash and surjective, because $M=\bigcup_{i=1}^mu_i^{-1}(\ol{\Bb}_d)$. In addition, $h_i^{-1}(u_i^{-1}(\ol{\Bb}_d))\subset\Int(\sigma)\times(2(i-1),2(i-1)+1)$, so the sets $h_i^{-1}(u_i^{-1}(\ol{\Bb}_d))$ (for $i=1,\ldots,r$) are pairwise disjoint. Let $\Hh:\R^d\to M$ be a continuous semialgebraic extension of $h$ to $\R^d$. Define $X_1:=\bigcup_{i=1}^r(u_i\circ h_i)^{-1}(\partial \ol{\Bb}_d)$, which is a Nash subset of $\R^d$ contained in $\Omega:=\Int(\sigma)\times\bigcup_{i=1}^r(2(i-1),2(i-1)+1)$ and observe that $\Hh$ is Nash on the open semialgebraic set $\Omega$. By \cite[Thm.II.5.2]{sh} there exists a Nash map $F:\R^d\to M$ close to $\Hh$ such that $F|_{(u_i\circ h_i)^{-1}(\partial\ol{\Bb}_d)}=H|_{(u_i\circ h_i)^{-1}(\partial\ol{\Bb}_d)}$ for $i=1,\ldots,r$. As the restriction $\Hh|_{(u_i\circ h_i)^{-1}(\ol{\Bb}_d)}=h_i|_{h_i^{-1}(u_i^{-1}(\ol{\Bb}_d))}:h_i^{-1}(u_i^{-1}(\ol{\Bb}_d))\to u_i^{-1}(\ol{\Bb}_d)$ is by Lemma \ref{simp3} a Nash diffeomorphism, the restriction $F|_{(u_i\circ h_i)^{-1}(\ol{\Bb}_d)}$ is by Lemma \ref{approxb} surjective for $i=1,\ldots,r$. Thus, $M=\bigcup_{i=1}^mu_i^{-1}(\ol{\Bb}_d)\subset F(\sigma\times[0,2r-1])\subset M$, so $F(\sigma\times[0,2r-1])=M$, as required.
\qed

\subsection{Nash structure of orientable compact smooth manifolds}\label{surfaces}
It is well-known that the compact orientable surface of genus $g$ admits a Nash model. We show in the following example how to deduce straightforwardly this fact from Theorem \ref{ndnmwc}.

\begin{example}
Let $\Pp_n\subset\R^2$ be a convex polygon with $n$ edges. As $\Pp_n$ is convex, there exist polynomials $h_1,\ldots,h_n\in\R[\x,\y]$ of degree 1 such that $\Pp_{n}=\{h_1\geq0,\ldots,h_{n}\geq0\}$. Fix $2+\frac{1}{2}(1-(-1)^n)\leq s\leq n$ and let $\Jj:=\{J_k\}_{k=1}^s$ be a partition of the set $\{1,\ldots,n\}$ such that
\begin{equation}\label{compatibili}
\Pp_n\cap\{h_i=0\}\cap\{h_j=0\}=\varnothing 
\end{equation}
for each $i,j\in J_k$ with $i\neq j$. For each $k=1,\ldots,s$ define the polynomial $h_{J_k}:=\prod_{i\in J_k} h_i$. As the partition $\Jj$ satisfies \eqref{compatibili}, it holds that $\Pp_n$ is a connected component of the semialgebraic set $\{x\in\R^2:\ h_{J_1}(x)\geq0,\ldots,h_{J_s}(x)\geq0\}$. Thus, 
$$
D_s(\Pp_n):=\{(x,t)\in\Pp_n\times\R^s:\ t_1^2-h_{J_1}(x)=0,\ldots,t_s^2-h_{J_s}(x)=0\}\subset\R^{s+2}
$$
is a connected compact Nash surface, which is in addition a connected component of the (maybe singular) real algebraic set
$$
X_{s,n}=\{(x,t)\in\R^2\times\R^s:\ t_1^2-h_{J_1}(x)=0,\ldots,t_s^2-h_{J_s}(x)=0\}\subset\R^{s+2}.
$$
We claim: $D_s(\Pp_n)\subset\Sth(X_{s,n})$. 

Pick $(x,t)\in D_s(\Pp_n)$ and consider the Jacobian matrix
$$
J_{s,n}(x,t):=\begin{pmatrix}
2t_1&0&\cdots&0&\frac{\partial h_{J_1}}{\partial\x_1}(x)&\frac{\partial h_{J_1}}{\partial\x_2}
(x)\\
0&2t_2&\cdots&0&\frac{\partial h_{J_2}}{\partial\x_1}(x)&\frac{\partial h_{J_2}}{\partial\x_2}
(x)\\
\vdots&\vdots&\ddots&\vdots&\vdots&\vdots\\
0&0&\cdots&2t_s&\frac{\partial h_{J_s}}{\partial\x_1}(x)&\frac{\partial h_{J_s}}{\partial\x_2}
(x)
\end{pmatrix},
$$
which has rank $\leq s$. As $(x,t)\in D_s(\Pp_n)$, then $x\in\Pp_n$ and there exists at most two indices $k,\ell$ such that $x\in\{h_{J_k}=0,h_{J_\ell}=0\}$. If such is the case, the vectors $(\frac{\partial h_{J_k}}{\partial\x_1}(x),\frac{\partial h_{J_k}}{\partial\x_2}(x))$ and $(\frac{\partial h_{J_\ell}}{\partial\x_1}(x),\frac{\partial h_{J_\ell}}{\partial\x_2}(x))$ are linearly independent. If there exists only one index $k$ such that $x\in\{h_{J_k}=0\}$, the vector $(\frac{\partial h_{J_k}}{\partial\x_1}(x),\frac{\partial h_{J_k}}{\partial\x_2}(x))$ is non-zero because the partition $\Jj$ satisfies \eqref{compatibili}. As $t_k^2-h_{J_k}(x)=0$ for $k=1,\ldots,s$, the matrix $J_{s,n}(x,t)$ has rank $s$, so $(x,t)\in\Sth(X_{s,n})$, as claimed.

As $D_s(\Pp_n)$ is obtained recursively by doubling orientable Nash manifolds with boundary, $D_s(\Pp_n)$ is an orientable Nash surface. By the (smooth) classification of surfaces (see for instance \cite[Ch.9]{hir3}) $D_s(\Pp_n)$ is diffeomorphic to a connected sum of $g$ tori and the genus $g$ completely characterize its diffeomorphism class. 

For each $\veps:=(\veps_1,\ldots,\veps_s)\in\{-1,1\}^s$ the set $D_s(\Pp_n)\cap\{\veps_1t_1\geq0,\ldots,\veps_st_s\geq0\}$ is Nash diffeomorphic to $\Pp_n$. Thus, $D_s(\Pp_n)$ is obtained (topologically) by glueing $2^s$ copies of $\Pp_n$, one for each choice of $\veps\in\{-1,1\}^s$. The polygon $\Pp_n$ has a (natural) structure of CW complex with $n$ vertices, $n$ edges and 1 face. This CW complex structure induces a CW complex structure on $D_s(\Pp_n)$ with $2^{s-2}n$ vertices (because each vertex belongs exactly to 4 polygons of the CW complex), $2^{s-1}n$ edges (because each edge belongs exactly to 2 polygons of the CW complex) and $2^s$ faces. The Euler characteristic of $D_s(\Pp_n)$ is $2^{s-2}n-2^{s-1}n+2^s=2^{s-2}(4-n)$ and its genus is $g=\frac{2-2^{s-1}(4-n)}{2}=2^{s-3}(n-4)+1$, that is, $D_s(\Pp_n)$ is diffeomorphic to the connected sum of $g=2^{s-3}(n-4)+1$ tori, whenever the number $2^{s-3}(n-4)+1$ is a non-negative integer. In particular, the compact orientable surface of genus $g\geq1$ is diffeomorphic to the Nash surfaces $D_2(\Pp_{2g+2})\subset\R^4$ and $D_3(\Pp_{g+3})\subset\R^5$. Note that the genus $g$ of the surface $D_s(\Pp_n)$ depends both on $n$ and $s$ (see Table \ref{superfici}).

\begin{table}[!ht]
\centering
\begin{tabular}{|c|a|b|c|c|c|c|c|}
\hline
 & $s=2$ & $s=3$ & $s=4$ & $s=5$ & $s=6$ & $s=7$\\
\hline
$n=3$ & -- & 0 & -- & -- & -- & --\\
\hline
$n=4$ & 1 &1 & 1 & -- & -- & --\\
\hline
$n=5$ & -- &2 & 3 & 5 & -- & --\\
\hline
$n=6$ & 2 & 3 & 5 & 9 & 17 & --\\
\hline
$n=7$ & -- & 4 & 7 & 13 & 25 & 49\\
\hline
\end{tabular}
\caption{\small{The genus of the Nash surface $D_s(\Pp_n)$ for $n\leq 7$.}}
\label{superfici}
\end{table}

By Lemma \ref{cnma} there exist irreducible non-singular real algebraic sets $X$ with at most two connected components that are both Nash diffeomorphic to either $D_2(\Pp_{2g+2})\subset\R^4$ or $D_3(\Pp_{g+3})\subset\R^5$. Consequently, our procedure provides algebraic sets with at most two connected components (mutually diffeomorphic) such that their connected components are models for the orientable compact smooth surfaces.
\hfill$\sqbullet$
\end{example}

\subsection{Nash approximation.}
We make use of Theorem \ref{ndnmwc} to prove Theorem \ref{approxn}. Let $\Qq\subset\R^n$ be a Nash manifold with corners, let $M\subset\R^n$ be a Nash envelope of $\Qq$ and $Y\subset M$ a Nash normal-crossings divisor such that $\Qq\cap Y=\partial\Qq$. Let $h_1,\ldots,h_\ell\in{\mathcal N}(M)$ be Nash equations for the irreducible components $Y_1,\ldots,Y_\ell$ of $Y$, that is, $Y_i:=\{h_i=0\}$ and $d_xh_i:T_xM\to\R$ is surjective for each $x\in Y_i$. We have $\Qq=\{h_1\geq0,\ldots, h_\ell\geq0\}$ and consider the Nash double of $\Qq$, that is, the Nash manifold $D(\Qq):=\{(x,t)\in M\times\R^\ell:\ t_1^2-h_1(x)=0,\ldots,t_\ell^2-h_\ell(x)=0\}$ together with the projection $\pi:D(\Qq)\to M, (x,t)\mapsto x$. It holds $\pi(D(\Qq))=\Qq$ and if we define $\pi_+^{-1}:\Qq\to D(\Qq),\ x\mapsto(x,\sqrt{h_1(x)},\ldots,\sqrt{h_\ell(x) })$, then $\pi_+^{-1}$ is a semialgebraic homeomorphism onto its image $\Qq^+:=D(\Qq)\cap\{t_1\geq0,\ldots,t_\ell\geq0\}$ such that $\pi_+^{-1}\circ\pi |_{\Qq^+}=\id_{\Qq^+}$ and $\pi\circ\pi _+^{-1}=\id_{\Qq}$.

\begin{proof}[Proof of Theorem \em \ref{approxn}]
By \cite{dk} there exist an open semialgebraic neighborhood $U\subset\R^m$ of $\Ss$ (in which $\Ss$ is closed) and a semialgebraic retraction $\nu:U\to\Ss$. Consider the continuous semialgebraic function $F:=\pi_+^{-1}\circ f\circ\nu:U\to D(\Qq)$. Let $\veps_0:\Ss\to\R$ be a strictly positive continuous semialgebraic function and let $\veps:U\to\R$ be a strictly positive continuous semialgebraic extension of $\veps_0$ to $U$. We consider the following commutative diagram:
$$
\xymatrix{
U\ar[r]^{\hspace{-1mm}F}\ar[d]_{\nu} &\Qq^+\ar@{^{(}->}[r] & D(\Qq)\ar[dl]^{\pi}\\
\Ss\ar[r]_{f} &\Qq\ar[u]^{\pi_+^{-1}} & 
}
$$
As the map $\pi_*:{\mathcal S}^0(U,D(\Qq))\to{\mathcal S}^0(U,\Qq), F\mapsto\pi\circ F$ is continuous by \cite[Rmk.II.1.5]{sh}, there exists $\delta:U\to\R$ such that if $G\in{\mathcal S}^0(U,D(\Qq))$ and $\|F-G\|<\delta$, then $\|\pi\circ F-\pi\circ G\|<\veps$. Let $G:U\to D(\Qq)$ be a Nash map such that $\|F-G\|<\delta$, so $\|\pi\circ F-\pi\circ G\|<\veps$. As $F|_{\Ss}=\pi_+^{-1}\circ f$, we deduce
$$
\|f-\pi\circ G|_{\Ss}\|=\|\id_{\Qq}\circ f-\pi\circ G|_{\Ss}\|=\|\pi\circ\pi _+^{-1}\circ f-\pi\circ G|_{\Ss}\|=\|\pi\circ F|_{\Ss}-\pi\circ G|_{\Ss}\|<\veps|_{\Ss}=\veps_0,
$$
as required.
\end{proof}






\bibliographystyle{amsalpha}

\end{document}